\pdfoutput=1
\RequirePackage{ifpdf}
\ifpdf 
\documentclass[pdftex]{sigma}
\else
\documentclass{sigma}
\fi

\usepackage[protrusion=true,expansion=true]{microtype}	
\usepackage{amsmath,amsfonts,amsthm,amssymb,mathtools}
\usepackage{wrapfig}	
\usepackage{commath}
\usepackage[nomessages]{fp}
\usepackage{relsize}

\usepackage[absolute]{textpos}
\usepackage{dsfont}
\usepackage{makecell}
\usepackage{array}

\usepackage{tikz-cd}

\usepackage{xstring}
\usepackage{simpler-wick}

\newtheorem{Theorem}{Theorem}[section]
\newtheorem{Corollary}[Theorem]{Corollary}
\newtheorem{Lemma}[Theorem]{Lemma}
\newtheorem{Proposition}[Theorem]{Proposition}
 { \theoremstyle{definition}
	\newtheorem{Definition}[Theorem]{Definition}
 \newtheorem{Observation}[Theorem]{Observation}
 \newtheorem{Example}[Theorem]{Example}
	\newtheorem{Remark}[Theorem]{Remark} }

\numberwithin{equation}{section}

\newcommand{\superimpose}[2]{\mathrel{\ooalign{${#1}$\cr\hidewidth\raise.235ex\hbox{${#2}\mkern-1.1mu$}\cr}}}

\newcommand{\stset}[1]{\mathbb{#1}}

\newcommand{\stcat}[1]{\mathcal{#1}}
\newcommand{\stmove}[1]{\mathbf{#1}}
\newcommand{\stsimple}[1]{\mathbf{#1}}

\newcommand{\field}{\stset{K}}

\newcommand{\id}{\mathrm{id}}
\newcommand{\si}[1]{\stsimple{#1}}

\newcommand{\one}{\mathds{1}}
\newcommand{\ewidI}{\mathrm{I}}
\newcommand{\ewidIcirc}{\superimpose{\ewidI}{\circ}}
\newcommand{\coendbal}{\beta}

\newcommand{\halfbraid}[1]{\mathrm{br}^{#1}}
\newcommand{\tenhomid}{\star}
\newcommand{\sigsurf}{\Sigma}
\newcommand{\evaluate}[1]{\lvert {#1} \rvert}
\newcommand{\totaleval}[3]{( {#1}, \lvert {#2} \rvert {#3})}
\newcommand{\alteval}[1]{\lvert\lvert {#1} \rvert\rvert}
\newcommand{\gencomp}[1]{\circ_{#1}}
\newcommand{\mitosis}[1]{\bullet_{#1}}
\newcommand{\bal}{\gamma}
\newcommand{\eqdesc}[1]{\stackrel{\mathclap{#1}}{=}}
\newcommand{\eqwithref}[1]{\hspace{1mm}\eqdesc{\text{\eqref{#1}}}\hspace{1mm}}

\newcommand{\holidem}{h}
\newcommand{\holsurj}{\pi}

\newcommand{\opcat}[1]{\overline{#1}}
\newcommand{\revcat}[1]{{#1}^{\mathrm{rev}}}

\newcommand{\bimod}[3]{{}_{#1}{#2}_{#3}}
\newcommand{\lmod}[2]{{}_{#1}{#2}}
\newcommand{\rmod}[2]{{#1}_{#2}}

\newcommand{\acat}{\stcat{A}}
\newcommand{\bcat}{\stcat{B}}
\newcommand{\ccat}{\stcat{C}}
\newcommand{\dcat}{\stcat{D}}

\newcommand{\mcat}{\stcat{M}}
\newcommand{\ncat}{\stcat{N}}
\newcommand{\kcat}{\stcat{K}}
\newcommand{\lcat}{\stcat{L}}

\newcommand{\xcat}{\stcat{X}}
\newcommand{\ycat}{\stcat{Y}}

\newcommand{\modulecent}[2]{Z_{#1}(#2)}
\newcommand{\bimodfun}[4]{\mathrm{Fun}_{{#1}|{#2}}({#3}, {#4})}

\newcommand{\lmodfun}[3]{\mathrm{Fun}_{{#1}}({#2}, {#3})}
\newcommand{\funcat}[2]{\mathrm{Fun}({#1}, {#2})}
\newcommand{\balfun}[2]{\mathrm{Fun}^\bal({#1}, {#2})}

\newcommand{\leftind}{\rexreldel}
\newcommand{\rightind}{\lexreldel}
\newcommand{\eq}{\mathrm{eq}}

\newcommand{\ew}{\mathrm{EW}}

\newcommand{\mew}{\mathrm{EW}_\mathrm{m}}

\newcommand{\ewid}[1]{\ewidI_{#1}}
\newcommand{\ewcoid}[1]{\ewid{#1}}
\newcommand{\zewid}[1]{\ewidIcirc_{#1}}
\newcommand{\zewcoid}[1]{\zewid{#1}}

\newcommand{\raycat}{\mathbb{T}^\mathrm{R}}
\newcommand{\nodecat}{\mathbb{T}^\mathrm{N}}
\newcommand{\nodespace}{\mathrm{N}}
\newcommand{\preblock}{\mathrm{T}^\mathrm{p}}
\newcommand{\blockspace}{\mathrm{T}}

\newcommand{\moveOR}{{\bf OR}}
\newcommand{\moveC}{{\bf C}}
\newcommand{\moveEF}{{\bf EF}}

\newcommand{\moveL}{{\bf L}}
\newcommand{\moveDE}{{\bf TE}}
\newcommand{\moveDV}{{\bf SN}}

\newcommand{\movemap}{\psi}
\newcommand{\Pmovemap}{\Psi^\mathrm{p}}
\newcommand{\Movemap}{\Psi}
\newcommand{\postmove}[1]{{{#1}'}}

\DeclareMathOperator*{\btimes}{\boxtimes}

\DeclareMathOperator*{\stimes}{\square}
\newcommand{\deligne}{\btimes}
\newcommand{\bigdeligne}{\mathlarger{\deligne}}
\newcommand{\reldel}{\square}
\newcommand{\reldelov}[1]{\stimes\limits_{#1}}

\newcommand{\rexreldel}{\reldel}
\newcommand{\lexreldel}{\reldel}

\newcommand{\indadjl}[2]{{\mathrm{adj}_{{#1},{#2}}^{\mathrm{l}}}}
\newcommand{\indadjr}[2]{{\mathrm{adj}_{{#1},{#2}}^{\mathrm{r}}}}

\newcommand{\homset}[2]{\langle #1 , #2 \rangle}
\newcommand{\homsetin}[3]{{_{#1}} \langle #2 , #3 \rangle}

\newcommand{\basisel}[2]{\tenhomid_{\homset{#1}{#2}}}

\newcommand{\soc}{\mathrm{sil}}

\newcommand{\ct}[2]{{#1}_{#2}}
\newcommand{\et}[2]{{#1}_{#2}} 
\newcommand{\tc}[2]{{}_{#2}{#1}}
\newcommand{\te}[2]{{}_{#2}{#1}} 
\newcommand{\ctc}[3]{{}_{#3}{#1}_{#2}}
\newcommand{\cte}[3]{{}_{#3}{#1}_{#2}}

\newcommand{\ete}[3]{{}_{#3}{#1}_{#2}} 
\newcommand{\coincl}[1]{\ct{\mathrm{incl}}{#1}}
\newcommand{\projend}[1]{\te{\mathrm{pr}}{#1}}

\newcommand{\sweed}[2]{{#1}_{(#2)}}

\newcommand{\trace}[1]{\mathrm{Tr}_{#1}}

\newcommand{\catdim}[1]{\mathcal{D}_{#1}}
\newcommand{\pdim}[1]{\mathrm{d}_{#1}}
\newcommand{\ev}[1]{\ct{\mathrm{ev}}{#1}}
\newcommand{\coev}[1]{\te{\mathrm{coev}}{#1}}
\newcommand{\brev}[2]{\ct{\mathrm{brev}_{#1}}{#2}}
\newcommand{\cobrev}[2]{\te{\mathrm{cobrev}_{#1}}{#2}}

\newcommand{\baleqat}[3]{\mathrm{eq}{#1}({{#2},{#3}})}

\newcommand{\class}[2]{\left\{ #1  \middle|  #2 \right\}}

\newcommand{\Left}[2]{\left#1 \vphantom{\del[#2]{X^X}}\right. }
\newcommand{\Right}[2]{ \left. \vphantom{\del[#2]{X^X}}\right#1}

\renewcommand{\cal}[1]{\mathcal{#1}}

\newcommand{\rest}{\mathrm{R}}

\newcommand{\ihom}[3]{\relihom{#1}{ {#2} , {#3} }}
\newcommand{\relihom}[2]{{}_{#1}[ #2 ]}

\newcommand{\vect}{\mathrm{vect}}

\newcommand{\pica}[2]{\raisebox{-0.5\height}{\includegraphics[scale=#1]{#2}}}

\renewcommand{\Left}[2]{\left#1 \vphantom{\del[#2]{X^X}}\right. }
\renewcommand{\Right}[2]{ \left. \vphantom{\del[#2]{X^X}}\right#1}

\newenvironment{mequation*}{\begin{align*}\begin{aligned}}{\end{aligned}\end{align*}}

\begin{document}
	
\allowdisplaybreaks

\newcommand{\arXivNumber}{2312.01946}

\renewcommand{\PaperNumber}{102}

\FirstPageHeading

\ShortArticleName{The Evaluation of Graphs on Surfaces for State-Sum Models with Defects}

\ArticleName{The Evaluation of Graphs on Surfaces\\ for State-Sum Models with Defects}

\Author{Julian FARNSTEINER and Christoph SCHWEIGERT}

\AuthorNameForHeading{J.~Farnsteiner and C.~Schweigert}

\Address{Fachbereich Mathematik, Universit\"at Hamburg, Bereich Algebra und Zahlentheorie,\\ Bundesstra\ss{}e 55, 20146 Hamburg, Germany}
\Email{\href{mailto:julfarn@gmx.de}{julfarn@gmx.de}, \href{mailto:christoph.schweigert@uni-hamburg.de}{christoph.schweigert@uni-hamburg.de}}
\URLaddress{\url{https://www.math.uni-hamburg.de/home/schweigert/}}

\ArticleDates{Received December 10, 2023, in final form November 06, 2024; Published online November 18, 2024}
	
\Abstract{The evaluation of graphs on 2-spheres is a central ingredient of the Turaev--Viro construction of three-dimensional topological field theories. In this article, we introduce a~class of graphs, called extruded graphs, that is relevant for the Turaev--Viro construction with general defect configurations involving defects of various dimensions. We define the evaluation of extruded graphs and show that it is invariant under a set of moves. This ensures the computability and uniqueness of our evaluation.}


\Classification{18M30; 18M20; 81T45}

\section{Introduction} \label{sec_intro}
Topological field theories have mediated many relations between topology and representation theory: they also have applications in
the theory of topological phases of matter, in quantum computing and for the
construction of correlators of 2-dimensional conformal field theories.
As a~consequence, topological field theories, in particular in three dimensions,
have been intensely
studied in the last few decades. Various constructions of topological field
theories have been developed; the state-sum construction combines the virtues
of providing fully local theories, together explicit tools for concrete computations.

It is by now well accepted that topological field theories should be studied
on stratified manifolds where strata of various codimensions can be endowed with representation-theoretic
decorations. Such strata, also called defects,
are central for applications to representation theory; they also
provide insight in
symmetries and dualities of the topological field theories; for an early
study in the case of 3-dimensional Dijkgraaf--Witten theories see, for example,~\cite{fpsv}.

In the context of state-sum constructions, defects have recently been studied
intensely; a~selection of papers includes \cite{bms-tricatdiagrams,cms-tricat-tfts,crs-orbifolds,dkr-defect-tfts,ks-surfaceops,kmrs-defect-rt,meusburger-statesumdefects}. The
stratifications discussed in these papers are, however, of limited generality.
We describe the more general case which has also been treated recently in \cite{cm-orb-tricats} (see also \cite{carqueville-orbifolds-tft, czr-top-def}):
consider, for simplicity, a closed oriented 3-manifold, together with a skeleton in the sense
of \cite{tv}. We label the 3-cells of the skeleton by spherical fusion categories
and the 2-cells by appropriate bimodule categories with trace. Figure~\ref{pic_intro_1} shows two 3-cells
labeled by $\cal A$ and $\cal B$ as well as 2-cells labeled by $\cal K$, $\cal M$ and
$\cal N$.
\begin{figure}[t]
\centering
\includegraphics[scale=0.9]{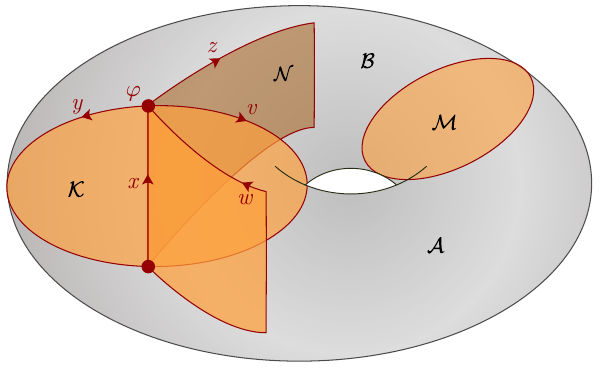}
\caption{}\label{pic_intro_1}
\end{figure}

In the simplest situation, which corresponds to the theory described in, e.g., \cite{tv},
all 3-cells are labeled by the same spherical fusion category $\cal A$ and all
2-cells by the regular $\acat$-$\acat$-bimodule category, which is the category $\cal A$ with left and right action
given by the tensor product in $\cal A$.
We refer to this situation as the {\em monochromatic} labeling.
In this way, we recover the classical Turaev--Viro construction.

At a 1-cell (or edge) of a skeleton, finitely many 2-cells labeled by possibly different bimodule categories meet.
For example, at the 1-cell in Figure~\ref{pic_intro_1}
that is labeled by $x$, four 2-cells meet.
The orientation of the 1-cells, together with the orientation of the ambient 3-manifold, provides a cyclic
order for the adjacent 2-cells.
The labeling is such that we can assign the cyclically balanced Deligne
product of the bimodule categories to the 1-cell.
We decorate each 1-cell with an object in this balanced Deligne product.
These objects are denoted by $v$, $w$, $x$, $y$, $z$ in Figure~\ref{pic_intro_1}. We call such a labeled 1-cell a
generalized Wilson line.

In the monochromatic case, the balanced Deligne product is equivalent to
the Drinfeld center of $\cal A$. Indeed, ordinary bulk Wilson lines are labeled
by objects in the Drinfeld center.
Since the Drinfeld center is monoidal, there is a distinguished {\em transparent} label for a bulk Wilson line: the monoidal unit of the Drinfeld center.
(It should be appreciated that the treatment of bulk Wilson lines in \cite{tv}
is slightly different: they enter as additional geometric data and are not realized as 1-cells in the skeleton.)

Finally, we consider vertices where several 1-cells meet. An example
in Figure~\ref{pic_intro_1} is the vertex labeled by $\varphi$ in which 1-cells labeled
by $v$, $w$, $x$, $y$ and $z$ meet.
In the monochromatic case of ordinary Wilson lines labeled by
objects $x_1,\dots, x_n$ in a Drinfeld center $\mathcal{Z}(\mathcal{A})$,
the label for the vertex is an
invariant tensor $\mathrm{Hom}_{\mathcal{Z}(\mathcal{A})}(1,x_1\otimes \dots \otimes x_n)$.
In fact, such a labeling requires auxiliary choices like a linear ordering of the objects
$x_i$. This can be avoided by considering a small sphere around the vertex;
this sphere has labeled marked points where
the bulk Wilson lines pierce the sphere. In this situation, we have the corresponding
block space for $\mathcal{Z}(\mathcal{A})$ on the sphere at our disposal, which does not depend on auxiliary
choices. We therefore assign an element in the block space to the vertex. To apply a similar strategy beyond the monochromatic case,
we use that block spaces, even in the presence of defects, have been
constructed in \cite{fss-statesum}.
While the block spaces in \cite{fss-statesum} have been constructed for framed surfaces, we here use
a variant for oriented surfaces. As explained in Remark~\ref{rem_fss_not_needed}, this is not a problem in the
situations we consider.

We call a 3-manifold with embedded labeled skeleton a {\em labeled defect $3$-manifold};
a topological field theory
has to assign a scalar to such a manifold.

The state-sum construction of such an invariant then proceeds as follows: one chooses as state-sum variables
objects in the bimodule categories that are assigned to the 2-cells of the
skeleton. (At a later stage in the state-sum construction, there will be a summation
over these variables.) For a fixed choice of state-sum variables,
one constructs a vector space $V_{(e,\iota)}$ for each half-edge~${(e,\iota)}$, i.e., for each pair consisting of a 1-cell $e$
and a homotopy class $\iota$ of embeddings~${[0,1) \to e}$ that send 0 to an end point of $e$.
The reader is invited to think about this
space as a~space of invariants. This, however, is only literally true in the transparent
case; we construct the relevant vector spaces in Section~\ref{sec_fssrecap}.
Each edge gives rise to a pair of half-edges. The trace structure of the bimodule
categories incident to the edge gives rise to a non-degenerate bilinear
pairing on the two spaces of invariants.

To the labeled defect manifold, together with a choice of skeleton and a
choice of state-sum variables, we associate the tensor product $V:=\bigotimes_{(e,\iota)} V_{(e,\iota)}$
of these vector spaces over all half-edges.
There is a non-degenerate bilinear pairing
for the two vector spaces associated to the two half-edges $(e,\iota_+)$, $(e,\iota_-)$ of an edge $e$.
Hence the vector space $V_{(e,\iota_+)}\otimes V_{(e,\iota_-)}$
and also $V$ as a tensor product of such spaces comes with a distinguished non-zero vector $\star$.

The vector space $V$ can also be organized as a tensor product over vertices,
\[
 V=\bigotimes_{v} \bigotimes_{(e,\iota) \text{ incident to } v} V_{(e,\iota)}   .
\]
Around each vertex, we obtain the
following situation which is depicted in Figure~\ref{pic_intro_2} and which we call an extruded graph.

\begin{figure}[ht]
\centering
\includegraphics[scale=0.9]{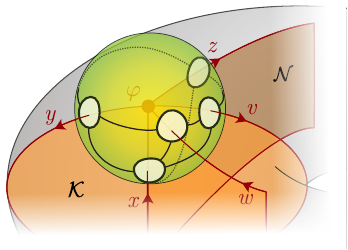}

\caption{}\label{pic_intro_2}
\end{figure}

In the monochromatic situation, without bulk Wilson lines, we obtain a labeled graph
on a small sphere surrounding the vertex. Its edges are the intersection of the
2-cells with the sphere and hence labeled by the state-sum objects. Its vertices
are the intersection of 1-cells with the sphere and hence labeled with
invariant tensors. By the standard graphical calculus of spherical fusion categories
\cite{bw-sphcats}
such a graph can be evaluated to a scalar, giving a linear form~${\evaluate{-} \colon V \to \field}$. Applying all these vertex evaluations to the
distinguished vector $\star\in V$ gives a scalar $\evaluate{\star}$; the weighted sum over these scalars for different choices state-sum variables is expected to be independent of the choice
of skeleton; it is a natural candidate for the Turaev--Viro invariant.

The vertex evaluation is thus a central ingredient of the Turaev--Viro
construction. A Turaev--Viro construction for a general defect 3-manifold
needs such an evaluation for extruded graphs.
The main result of this paper is such an evaluation, along with the proof that
it has all the properties that make it algorithmically computable.
This work does not contain a proof that a~Turaev--Viro theory can be built using our evaluation.
In particular, proving the independence of the invariant for 3-manifolds from the choice of skeleton involves combinatorial geometry for stratified manifolds for which the authors are not experts.
Some more discussion on the assembly into a state-sum model can be found in \cite[Section~5.4]{jf-thesis}.
We expect the resulting theory to be in agreement with the one presented recently in \cite{cm-orb-tricats}.

To define such a vertex evaluation one could
envisage the following (standard) strategy: exhibit a series of moves that reduces the
extruded graph to a standard graph on the sphere. A good candidate for such a standard graph would
be a loop on the sphere with a marked point labeled by an endomorphism. The
standard graph could then
be evaluated by a trace. This strategy has a crucial problem: one has to show that
any combination of moves that simplifies a graph gives a trace with the same numerical value.
This requires a careful analysis of relations between different moves. In our situation, this
is a very difficult problem, since the combinatorics of extruded graphs is involved and, to our knowledge, has not been developed.

Hence, in this paper, we use a rather different strategy: we {\em define} an
evaluation for {\em all} extruded graphs and then
prove theorems stating that a certain set of natural moves does not change the value
of the evaluation. The main idea of our definition is as follows: the block space $\blockspace_v$
in which the label $\varphi$ for a vertex $v$ is chosen comes by definition
\cite{fss-statesum} as a subspace
of a {\em pre-block space} $\preblock_v$. We show that the labels we have chosen
for the state-sum variables and the invariant tensors
combine to an element in the pre-block space $\preblock$. We then use the spherical structure on the fusion categories
to exhibit $\blockspace$ not only as a subspace, but as a retract of $\preblock$. This defines
the evaluation. We then introduce a set of moves and {\em prove} that
the moves leave the evaluation
invariant. The set of moves allows us to reduce the graph so that in the end,
we evaluate a trace. This ensures that our evaluation procedure
can be explicitly implemented in practice.

For the practitioner, the content of this article can thus be reduced
to the following statements: there is a well-defined evaluation procedure for extruded graphs, and using the
moves described in Section~\ref{sec_moves_overview}, the evaluation of any extruded graph can be reduced to the computation
of a~trace on a bimodule category.

This paper is organized as follows:
Section~\ref{sec_cats} fixes conventions and notation; and some algebraic notions, such as bimodule categories with trace and relative Deligne products, are recalled.
Moreover, diagrammatic notation used to manipulate morphisms is introduced in Section~\ref{sec_balfuns}, and in Section~\ref{sec_spleqs}, we define (split) \emph{equalizers of bi-balanced functors}.

Section~\ref{sec_construction} starts with the Definition~\ref{def_extgraph} of extruded graphs.
We continue by describing the block space from Section~\ref{sec_fssrecap}, which is associated to an extruded graph, as a split equalizer of a bi-balanced functor; this uses semisimplicity and the existence of spherical structures on the involved fusion categories.
In Definition~\ref{def_eval}, the evaluation of extruded graphs is introduced as a linear map associated to an extruded graph.
We then show in Theorem~\ref{thm_elementary_graph} that loop graphs are evaluated to traces.

Section~\ref{sec_moves} is dedicated to moves of invariance, see Definition~\ref{def_invariance}.
We give a list of six selected moves in Section~\ref{sec_moves_overview}, among which three are sufficient to make a uniqueness statement: our evaluation is unique in that a) it is left invariant by these moves and b) loop graphs are evaluated to traces; this is proved in Theorem~\ref{thm_uniqueness}.

The two appendices are reserved for more technical proofs, in particular for the proof that the moves are really moves of invariance.

\section{Algebraic preliminaries}
\label{sec_cats}
We need to fix notation, and recall standard concepts, such as bimodule categories, categorical centers and the relative Deligne product.
Lesser-known structures which are reviewed include bimodule traces in Section~\ref{sec_traced_bimods}.

Throughout this paper, fix an algebraically closed field $\field$ of characteristic 0.
We will mainly work with $\field$-linear categories, using standard terminology that can be found in \cite{egno}.
When not stated otherwise, all linear categories are also assumed finite and semisimple.

The following notational conventions serve to shorten the presentation:
\begin{itemize}\itemsep=0pt
	\item Given any category $\ccat$ with objects $x,y \in \ccat$, we denote Hom-sets by angled brackets
	\begin{equation*}
		\homset{x}{y} := \homsetin{\ccat}{x}{y} := \mathrm{Hom}_\ccat (x,y).
	\end{equation*}
	
	\item Lowercase bold letters ($\si{x}, \si{y}, \dots$) are used for simple objects.
	In a finite semisimple category~$\ccat$, the notation
$\sum_{\si{x}\in \ccat} \cdots
$
	stands for a sum over the set of isomorphism classes of simple objects of $\ccat$, with $\si{x}$ assuming the value of a representative for each class in the sum.
	
	\item We denote the opposite category $\opcat{\ccat}$ to any category $\ccat$ by an overline.
	If \smash{$
		x \xlongrightarrow{g} y \xlongrightarrow{f} z$} are objects and morphisms in $\ccat$, then we also denote the corresponding objects and morphism in $\opcat{\ccat}$ by an overline:
	\smash{$
			\opcat{z} \xlongrightarrow{\opcat{f}} \opcat{y} \xlongrightarrow{\opcat{g}} \opcat{x}$}.
	In other words,
	\smash{$
		\opcat{f \circ g} = \opcat{g} \circ \opcat{f}
	$}.
	
	\item We treat the Deligne product $\xcat \deligne \ycat$ of (as always, finite) linear categories $\xcat$ and $\ycat$ as strictly associative and use tacitly that the equivalence $\xcat \deligne \ycat \cong \ycat \deligne\xcat$ is part of a~symmetric structure.
	
	\item A morphism between finite direct sums $f\colon \bigoplus_{i\in I} X_i \to \bigoplus_{j\in J} Y_j$ is determined by the family of morphisms
$\ete{f}{i}{j} := \projend{j} \circ f \circ \coincl{i}\colon X_i \to Y_j$,
	which we call the \emph{matrix elements} or \emph{components} of $f$.
\end{itemize}

\subsection{Calabi--Yau categories}\label{sec_calabiyau}

\begin{Definition}[cf.\ \cite{costello-calabiyau,ms-dbranes-ktheory-tft, schaumann-modtraces}]
	A linear category $\ccat$ equipped with a \emph{trace} $\trace{\ccat}$, that is, a~collection of maps $\trace{\ccat}\colon \homsetin{\ccat}{x}{x} \to \field$ for each object $x\in \ccat$ satisfying
	\begin{itemize}\itemsep=0pt
		\item \emph{Symmetry}: $\trace{\ccat}(f\circ g) = \trace{\ccat}(g\circ f)$ for morphisms $f\colon x \to y$, $g\colon y\to x$ in $\ccat$; and
		\item \emph{Non-degeneracy}: The assignment $f \mapsto \trace{\ccat}(f\circ -)$ defines an isomorphism $\homsetin{\ccat}{x}{y} \cong \homsetin{\ccat}{y}{x}^*$ for each $f \colon x \to y$ in $\ccat$,
	\end{itemize}
	is called a \emph{Calabi--Yau category}.
	For any object $x\in \ccat$, the scalar $\pdim{x}:=\trace{\ccat}(\id_x)$ is called the \emph{dimension} of $x$.
\end{Definition}

\begin{Lemma}[{\cite[Proposition~5.2]{schaumann-modtraces}}]
	The trace of a $($finite semisimple$)$ Calabi--Yau category $\ccat$ is determined by its \emph{dimension vector} $(\pdim{\si{x}})$, that is, the finite list of dimensions $\pdim{\si{x}}$ of $($representatives of$)$ simple objects $\si{x} \in \ccat$. The entries of the dimension vector are non-zero, $\pdim{\si{x}}\neq 0$.
	Conversely, any list of non-zero scalars for each isomorphism class of simple objects defines a trace.
\end{Lemma}
The ``squared norm''
$\catdim{\ccat} := \sum_\si{x} \pdim{\si{x}}^2
$
of the dimension vector is called the \emph{dimension} of the Calabi--Yau category $\ccat$.
If $\ccat$ and $\dcat$ are two Calabi--Yau categories, then the Deligne product~${\ccat \deligne \dcat}$ is a Calabi--Yau category with trace
$\trace{\ccat \deligne\dcat} (f \otimes g) := \trace{\ccat}(f)  \trace{\dcat}(g)
$
for endomorphisms $f$ in~$\ccat$ and $g$ in $\dcat$.
The dimension of the product category is the product of the dimensions of the factors
\begin{equation*}\label{eq_dimension_deligne_product}
	\catdim{\ccat \deligne\dcat} = \sum_{\si{z}\in \ccat\deligne\dcat} \pdim{\si{z}}^2 = \sum_{\substack{\si{x} \in \ccat \\ \si{y} \in \dcat}} \pdim{\si{x} \deligne \si{y}}^2 = \sum_{\substack{\si{x} \in \ccat \\ \si{y} \in \dcat}} \pdim{\si{x}}^2 \pdim{\si{y}}^2 = \catdim{\ccat} \catdim{\dcat}.
\end{equation*}

\subsection{Monoidal categories}
We treat all monoidal categories as strict.
For the monoidal product of two objects $a$ and $b$, we simply write $ab$, with the exception of the tensor product $\otimes$ of vector spaces.
Reversing the product of a monoidal category $\acat$ leads to another monoidal category, denoted $\revcat{\acat}$.

The monoidal categories we consider are pivotal, i.e., they are a rigid and a monoidal isomorphism from the bidual functor to the identity functor has been chosen \cite[Definition~4.7.7]{egno}.
Thanks to the coherence theorem \cite[Theorem~2.2]{ns-basic}, we treat the pivotal structure as strict as well.
We will therefore not distinguish between left and right dualities, and denote by $a^*$ the dual of an object $a$ in a pivotal category.
We occasionally make use of the fact that a pivotal structure on a category $\acat$ gives rise to a distinguished monoidal equivalence
$\opcat{\acat} \cong \revcat{\acat}$,
which is the identity on objects and sends morphisms $\opcat{f}$ to the dual morphism $f^*$.
We denote the evaluation and coevaluation maps by
\begin{equation}\label{eq_ev_coev}
	\ev{a}\colon\ a  a^* \to \one \qquad\text{and} \qquad \coev{a} \colon\ \one \to a^* a.
\end{equation}
Instead of distinguishing between left and right evaluation and coevaluation, we stick to the morphisms \eqref{eq_ev_coev} and use{\samepage
\begin{equation*}
	\ev{a^*}\colon\ a^*  a \to \one \qquad\text{and} \qquad \coev{a^*}\colon\ \one \to a a^*
\end{equation*}
for the other duality.}

Both the left and the right pivotal trace endow a given pivotal fusion category $\acat$ with the structure of a Calabi--Yau category.
If the two traces agree, the category $\acat$ is called \emph{spherical}, and it carries the structure of a Calabi--Yau category in a canonical way.

\subsection{Diagrammatic notations} \label{sec_pictures}
We use the standard diagrammatic notation of string diagrams. These are always read bottom-to-top: for example, given morphisms $f \colon x\otimes y^* \to z$, $g\colon y \to w^*$ in a pivotal category $\acat$, we write
\begin{equation}\label{eq_strdiag_ex}
	(z\otimes g) \circ (f \otimes y) \circ (x\otimes \coev{y}) = \pica{1}{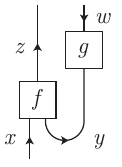}.
\end{equation}
For generalities on string diagrams, we refer to \cite{tv}.

In the string diagram \eqref{eq_strdiag_ex}, all strands are labeled by objects and all coupons are labeled by appropriate morphisms.
We need to work with linear maps between Hom-spaces and therefore introduce and extension of the graphical calculus to diagrams in which we leave one coupon empty, drawn with a double outline.
Such a diagram represents a linear map from the Hom-space associated with the empty coupon to the Hom-space in which the fully labeled string diagram evaluates to a morphism.
The following example shows a linear map between Hom-spaces~${\homset{x\otimes y^*}{z} \to \homset{x}{z \otimes w^*}}$:
\begin{equation*}
	\pica{1}{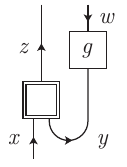} = (f \mapsto (z\otimes g) \circ (f \otimes y) \circ (x\otimes \coev{y}) ).
\end{equation*}
Since they describe maps of Hom-spaces, these diagrams with empty coupons are composed by inserting the pre-composed diagram into the empty box of the post-composed diagram.
This is illustrated below, in an example which describes a composition of linear maps
\begin{gather}
	\homsetin{\vect}{\homset{x y^* y}{z w^*}}{\homset{x}{z w^*}} \otimes \homsetin{\vect}{\homset{x y^*}{z}}{\homset{x y^* y}{z w^*}} \xrightarrow{\circ} \homsetin{\vect}{\homset{x y^*}{z}}{\homset{x}{z w^*}},\nonumber\\
	\pica{1}{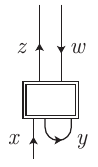} \circ \pica{1}{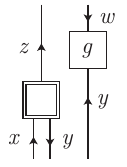} = \pica{1}{pic_emptybox_ex}.\label{eq_emptybox_comp}
\end{gather}
In Sections~\ref{sec_centers} and \ref{sec_balfuns}, we will see further generalizations of these diagrams.

\subsection{Hom-space contractions} \label{sec_homspace_contractions}
For a finite-dimensional vector space $V$, we follow \cite{tv} and denote the image of $\id_V$ under the identification $\homset{V}{V} \cong V \otimes V^*$ by $\tenhomid \in V \otimes V^*$.
It is useful to recall that an explicit form of~$\tenhomid$ is given as follows: Pick a basis $(\varphi_\alpha)$ of~$V$ and denote the dual basis of $V^*$ by $(\varphi^*_\alpha)$. Then~${\tenhomid = \sum_\alpha \varphi_\alpha \otimes \varphi_\alpha^*}$.
The following Sweedler-type notation is convenient:
$\tenhomid = \sweed{\tenhomid}{V} \otimes \sweed{\tenhomid}{V^*} \in V \otimes V^*$.

Let $V$ be a Hom-space $V = \homsetin{\ccat}{x}{y}$ of a Calabi--Yau category $\ccat$.
It then makes sense to consider the image of $\tenhomid$ under the isomorphism $\homsetin{\ccat}{x}{y} \otimes \homsetin{\ccat}{x}{y}^* \cong \homsetin{\ccat}{x}{y} \otimes \homsetin{\ccat}{y}{x}$.
We denote this image by $\tenhomid$ as well, and its Sweedler components by $\tenhomid = \basisel{x}{y} \otimes \basisel{y}{x}$.
When using graphical notation, we write
\begin{equation*}
	\tenhomid = \pica{1}{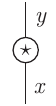} \otimes \pica{1}{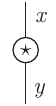}.
\end{equation*}

This allows us to express \emph{resolutions of the identity} of an object $x$ in a Calabi--Yau category~$\ccat$ diagrammatically
\begin{equation}\label{eq_res_of_id}
	\sum_\si{x} \pdim{\si{x}} \basisel{\si{x}}{x} \circ \basisel{x}{\si{x}} = \sum_\si{x} \pdim{\si{x}} \pica{1}{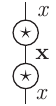} = \id_x.
\end{equation}
Recall that according to our conventions, this is a sum over isomorphism classes of simple objects~$\si{x}$ of $\ccat$.
The following variation of \eqref{eq_res_of_id} will also be useful
\begin{equation}\label{eq_parallel_res}
	\sum_\si{x} \pdim{\si{x}} (\basisel{\si{x}}{y} \circ \basisel{x}{\si{x}}) \otimes (\basisel{\si{x}}{x} \circ \basisel{y}{\si{x}}) = \basisel{x}{y} \otimes \basisel{y}{x}.
\end{equation}

The notation can also be used to conveniently express the multiplicity $\dim \homset{x}{\si{x}}$ of a simple object $\si{x}$ in an object $x$ in a Calabi--Yau category $\ccat$
$\dim \homset{x}{\si{x}} = \trace{\ccat} (\basisel{\si{x}}{x} \circ \basisel{x}{\si{x}})$.

\subsection{Traced bimodule categories}\label{sec_traced_bimods}
Given a monoidal category $\acat$, we consider \emph{left $\acat$-module categories} $\mcat$.
This is a category equipped with an (exact) \emph{action functor} $\acat \times \mcat \to \mcat$, denoted $(a,m) \mapsto am$, together with coherence data and axioms.
For details, we refer to \cite{egno}.

We are interested only in finite and semisimple module categories over fusion categories.
In particular, all of our module categories are exact.

There are the obvious related notions of right module categories and bimodule categories.
A~left $\acat$-module category can be equivalently described as a right $\revcat{\acat}$-module category.
An~$\acat$-$\bcat$-bimodule category is the same as a left $\acat \deligne \revcat{\bcat}$-module category.
Lastly, every linear category is canonically equipped with the structure of a left or right $\vect$-module category. Hence, a left $\acat$-module category can also be seen as an $\acat$-$\vect$-bimodule category.

We will make tacit use of these identifications.
To keep in mind which type of module category a particular category is, we sometimes write, for instance, $\bimod{\acat}{\mcat}{\bcat} = \mcat$ for an $\acat$-$\bcat$-bimodule category, and similar for left and right module categories.

For every fusion category $\acat$, we can consider the \emph{regular bimodule category} $\bimod{\acat}{\acat}{\acat}$, whose action functors are given by the monoidal product.

Because we work over \emph{finite} tensor categories, the bimodule categories we consider have internal Homs with respect to each action.
This means that given objects $m,n$ in a bimodule category $\bimod{\acat}{\mcat}{\bcat}$, there are \emph{internal Hom-objects} $\ihom{\acat}{m}{n} \in \acat$ and $\ihom{\bcat}{m}{n} \in \bcat$, characterized by the existence of isomorphisms
$\homsetin{\acat}{a}{\ihom{\acat}{m}{n}} \cong \homsetin{\mcat}{am}{n}$, $ \homsetin{\bcat}{b}{\ihom{\acat}{m}{n}} \cong \homsetin{\mcat}{bm}{n}$.

Pivotal structures on monoidal categories lead to a unique module structure on the opposite of a module category:
For a bimodule category $\bimod{\acat}{\mcat}{\bcat}$ over pivotal categories $\acat$ and $\bcat$, the opposite category $\opcat{\mcat}$ is canonically a $\bcat$-$\acat$-bimodule category, with actions defined by $b \opcat{m} a := \opcat{a^*  m  b^*}$.

If a bimodule category $\bimod{\acat}{\mcat}{\bcat}$ over pivotal categories $\acat$ and $\bcat$ also carries the structure of a~Calabi--Yau category, it is natural to impose the following consistency condition:
Let~${a\in \acat}$, ${b\in \bcat}$, and $m \in \mcat$ be objects, and let $f\colon amb \to amb$ be an endomorphism in $\mcat$.
The morphism
\begin{gather}\label{eq_ptr_def}
	\mathrm{ptr}(f) := (\ev{a^*}  m  \ev{b}) \circ (a^*  f  b^*) \circ (\coev{a}  m  \coev{b^*}),
\end{gather}
which is an endomorphism of $m$, is called the \emph{partial trace} of $f$ with respect to $a$ and $b$.
The consistency condition between the Calabi--Yau structure and the module actions on $\mcat$ mentioned above is
\begin{equation}\label{eq_module_trace_cond}
	\trace{\mcat} (f) = \trace{\mcat} (\mathrm{ptr}(f)).
\end{equation}
Equation \eqref{eq_module_trace_cond} can be illustrated in a string diagram \cite[equation~(3.31)]{schaumann-modtraces}
\begin{equation*}
	\trace{\mcat} \left(\pica{1}{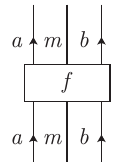}\right) = \trace{\mcat} \underbrace{\left(\pica{1}{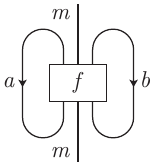}\right)}_{\mathrm{ptr}(f)}.
\end{equation*}
Following \cite{schaumann-modtraces}, we call a Calabi--Yau structure on a bimodule category over pivotal fusion categories which satisfies the partial trace property \eqref{eq_module_trace_cond} a \emph{bimodule trace}.
The corresponding notion for a one-sided module category is called a \emph{module trace}.
We refer to a (bi-)module category equipped with a (bi-)module trace as a \emph{traced $($bi-$)$module category}.
The proof of the next lemma is left to the reader.
\begin{Lemma}\label{lem_traces_and_sphericity}
	Let $\acat$ be a pivotal fusion category. The following statements are equivalent.
	\begin{enumerate}\itemsep=0pt
		\item[$(1)$] $\acat$ is a spherical fusion category.
		\item[$(2)$] The regular bimodule category $\bimod{\acat}{\acat}{\acat}$ admits a bimodule trace.
	\end{enumerate}
	If either statement is true, then the bimodule trace on $\bimod{\acat}{\acat}{\acat}$ is unique up to a normalization factor.
	If in addition, the bimodule trace is normalized such that the dimension of the monoidal unit is $1$, then the bimodule trace is equal to the $($left or right$)$ pivotal trace.
\end{Lemma}

\begin{Remark}\label{rem_modtrace_unique_exist}
	Given two $\acat$-module categories $\mcat$ and $\ncat$, the product category $\mcat \times \ncat$ can be equipped with the structure of an $\acat$-module category by simultaneous action on the components, which is denoted $\mcat \boxplus \ncat$.
	This module category is called the direct sum of~$\mcat$ and~$\ncat$.
	A~module category $\mcat$ is called \emph{indecomposable} if the existence of an equivalence of module categories~${\mcat \cong \ncat \boxplus \kcat}$ implies that either $\ncat$ or $\kcat$ is equivalent to $\mcat$, with the other being trivial.
	The structure of a bimodule trace on an indecomposable bimodule category is, if it exists, unique up to a scalar \cite[Proposition~4.4]{schaumann-modtraces}.
	The question whether or not a (bi-)module category admits a bimodule trace can be stated as an eigenvalue problem \cite[Propositions~5.4 and~5.7]{schaumann-modtraces}.
\end{Remark}

\begin{Remark}\label{rem_modtrace_fixed_point}
	The relevance of the structure of a bimodule trace for TFTs can also be understood from the perspective of homotopy fixed points:
	It is expected that spherical fusion categories are $\mathrm{SO}(3)$-fixed points in the homotopy category of fusion categories \cite{dss-dualizable-TCs}.
	We expect traced bimodule categories to be the corresponding $\mathrm{SO}(3)$-fixed point morphisms.
\end{Remark}

A module trace also provides a relation between the (pivotal) dimension of an internal Hom object, and the dimensions of the objects in the module category.

\begin{Lemma}\label{lem_ihom_dimensions}
	Given a left $\acat$-module category $\mcat$ over a spherical fusion category $\acat$ with a~module trace, the dimension of the internal Hom for simple objects $\si{m},\si{n} \in \mcat$ can be computed as follows
	\begin{equation*}
		\pdim{\ihom{\acat}{\si{m}}{\si{n}}} = \frac{\catdim{\acat}}{\catdim{\mcat}}  \pdim{\si{m}}\pdim{\si{n}}.
	\end{equation*}
\end{Lemma}
\begin{proof}
	Following \cite[Section~5]{schaumann-modtraces}, we call the square matrix $Q$ with entries
$Q_{\si{nm}} := \pdim{\ihom{\acat}{\si{m}}{\si{n}}}
$
	for (representatives of isoclasses of) simple objects $\si{m}, \si{n} \in \mcat$ the \emph{dimension matrix} of $\mcat$.
	From \cite[Proposition~5.7]{schaumann-modtraces}, we know that $Q$ is of rank 1 with only non-zero entries.
	This implies that~$Q$ is of the form $Q_\si{nm} = \xi_\si{n}\zeta_\si{m}$ for non-zero scalars $\xi_\si{n}, \zeta_\si{m} \in \field$.
	
	Moreover, the module trace on $\mcat$ causes the dual of the internal Hom to behave much like the dual of the Hom-space
	\begin{equation} \label{eq_dual_inthom}
		\ihom{\acat}{m}{n}^* = \ihom{\acat}{n}{m}
	\end{equation}
	for all objects $m,n \in \mcat$.
	This is easy to check using the Yoneda lemma; for $a \in \acat$, we have
	\begin{align*}
		\begin{aligned}
			\homset{a}{\ihom{\acat}{n}{m}} &\cong \homset{an}{m} \cong \homset{m}{an}^* \cong \homset{a^*m}{n}^* \\
			&\cong \homset{a^*}{\ihom{\acat}{m}{n}}^* \cong \homset{\ihom{\acat}{m}{n}}{a^*} \cong \homset{a}{\ihom{\acat}{m}{n}^*}.
		\end{aligned}
	\end{align*}
	Due to \eqref{eq_dual_inthom}, the dimension matrix $Q$ is symmetric, which means that $\xi_\si{m} = \zeta_\si{m}$ for all simple representatives $\si{m} \in \mcat$.
	
	The dimension vector is an eigenvector of the dimension matrix $Q$ with eigenvalue $\catdim{\acat}$, see \cite[Proposition~5.4]{schaumann-modtraces}
	\begin{equation*}
		\sum_\si{m} Q_\si{nm} \pdim{\si{m}} = \sum_\si{m} \xi_\si{n}\xi_\si{m} \pdim{\si{m}} = \catdim{\acat} \pdim{\si{n}}.
	\end{equation*}
	As $\xi_\si{n}$ and $\catdim{\acat}$ are non-zero, this is equivalent to the following:
\[
\frac{\sum_\si{m} \xi_\si{m} \pdim{\si{m}}}{\catdim{\acat}} = \frac{\pdim{\si{n}}}{\xi_\si{n}} =: \lambda,
\]
	where we introduced a non-zero scalar $\lambda$ that does not depend on $\si{n}$.
	Substituting $\xi_\si{m} = \frac{\pdim{\si{m}}}{\lambda}$, we find
	\begin{equation*}
		\frac{\sum_\si{m} \pdim{\si{m}} \pdim{\si{m}}}{\lambda  \catdim{\acat}} = \lambda, \qquad \text{and hence} \qquad \frac{\catdim{\mcat}}{\catdim{\acat}} = \lambda^2.
	\end{equation*}
	Finally, we obtain the statement of the lemma
	\begin{align*}
			\pdim{\ihom{\acat}{\si{m}}{\si{n}}} &= Q_{\si{nm}} = \xi_\si{n}\xi_\si{m} = \frac{\pdim{\si{n}}\pdim{\si{m}}}{\lambda^2} = \frac{\catdim{\acat}}{\catdim{\mcat}}  \pdim{\si{m}}\pdim{\si{n}}.\tag*{\qed}
	\end{align*}\renewcommand{\qed}{}
\end{proof}

The following lemma extends the $\star$-notation from Section~\ref{sec_homspace_contractions} to string diagrams in traced module categories.
More, precisely, the following lemma holds.
\begin{Lemma}\label{lem_dual_homs_and_duality}
	Let $\mcat$ be a left $\acat$-module category over a spherical fusion category $\acat$ with a~module trace.
	Let there be objects $m$, $n\in \mcat$ and $a \in \acat$. Denote by $(\varphi_\alpha)$ a basis of the Hom-space~${\homsetin{\mcat}{m}{an}}$, and the elements of the dual base by $\varphi^*_\alpha \in \homset{an}{m}$.
	
	Then the $\alpha$-indexed families of morphisms
$(\ev{a^*} n) \circ (a^* \varphi_\alpha)$ and $ (a^* \varphi_\alpha^*) \circ (\coev{a} n)
$
	form dual bases of the Hom-sets $\homset{a^* m}{n}$ and $\homset{n}{a^* m}$, respectively.
\end{Lemma}
Using the $\star$-notation introduced in Section~\ref{sec_homspace_contractions}, this statement can be expressed as the equality
\begin{equation*}
\basisel{a^* m}{n} \otimes \basisel{n}{a^* m} = ( (\ev{a^*} n) \circ (a^* \basisel{m}{an}) ) \otimes ((a^* \basisel{an}{m}) \circ (\coev{a} n)),
\end{equation*}
or graphically as
\begin{equation*}
\pica{1}{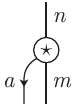} \otimes \pica{1}{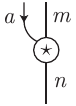} = \pica{1}{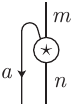} \otimes \pica{1}{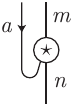}.
\end{equation*}
\begin{proof}
	We merely need to show that the duality relation
	\begin{equation}\label{eq_moduleduality_pf_1}
		\trace{\mcat} ( (\ev{a^*} n) \circ (a^* \varphi_\alpha) \circ (a^* \varphi_{\tilde{\alpha}}^*) \circ (\coev{a} n) ) = \delta_{\alpha, \tilde{\alpha}}
	\end{equation}
	holds.
	Using the compatibility \eqref{eq_module_trace_cond} of the module trace with the partial trace, it is easy to see that the duality relation \eqref{eq_moduleduality_pf_1} is inherited from the duality relation of the original dual bases~$\varphi_\alpha$,~$\varphi_\alpha^*$
	\begin{gather*}
			\trace{\mcat} ( (\ev{a^*} n) \circ (a^* \varphi_\alpha) \circ (a^* \varphi_{\tilde{\alpha}}^*) \circ (\coev{a} n) )
			\\
\qquad \eqwithref{eq_ptr_def} \trace{\mcat} ( \mathrm{ptr}( \varphi_\alpha \circ \varphi_{\tilde{\alpha}}^*) )
			\eqwithref{eq_module_trace_cond} \trace{\mcat} ( \varphi_\alpha \circ \varphi_{\tilde{\alpha}}^*)
			 = \delta_{\alpha, \tilde{\alpha}}.\tag*{\qed}
	\end{gather*}\renewcommand{\qed}{}
\end{proof}

\subsection{Centers of bimodule categories}\label{sec_centers}
Let $\bimod{\acat}{\mcat}{\acat}$ be an $\acat$-$\acat$-bimodule category.
Given an object $m\in \mcat$, a family of isomorphisms $\halfbraid{z}_a \colon am \to ma$ for each object $a\in \acat$, satisfying the conditions $\halfbraid{z}_{ab} = (\halfbraid{z}_a  b) \circ (a  \halfbraid{z}_b)$ and~${\halfbraid{z}_\one = \id_m}$ is called a \emph{balancing} for $m$.
The linear category $\modulecent{\acat}{\mcat}$, whose objects are pairs~${z = (m, \halfbraid{z})}$, where $m\in \mcat$ and $\halfbraid{z}$ is a balancing for $m$, and whose morphisms $f\colon z \to z' = (m', \halfbraid{z'})$ are morphisms $f\colon m\to m'$ satisfying the condition that $\halfbraid{z'}_a \circ (a  f) = (f a) \circ \halfbraid{z}_a$, is called the \emph{center} of $\mcat$
We denote the forgetful functor $z\mapsto m$, which forgets the balancing, by $U \colon \modulecent{\acat}{\mcat} \to \mcat$.
If $\mcat$ is the regular bimodule category $\bimod{\acat}{\acat}{\acat}$, the center $\modulecent{\acat}{\acat}$ is the Drinfeld center of $\acat$.

Given objects $z \in \modulecent{\acat}{\mcat}$ and $a\in \acat$, we define a particular pair of morphisms in $\mcat$
\begin{equation*}
	\brev{z}{a}\colon\ a U(z) a^* \to U(z) \qquad\text{and}\qquad \cobrev{z}{a} \colon\ U(z) \to a^* U(z) a,
\end{equation*}
to be called \emph{braided evaluation} and \emph{braided coevaluation}.
They can be defined in two equivalent ways, which are the left- and right-hand sides of the equations below
\begin{align}
	&\brev{z}{a}:= (U(z)  \ev{a}) \circ (\halfbraid{z}_a a^*) = (\ev{a}  U(z)) \circ \big(a (\halfbraid{z}_{a^*})^{-1}\big), \nonumber\\
	&\cobrev{z}{a}:= \big((\halfbraid{z}_{a^*}\big)^{-1}  a) \circ (U(z)  \coev{a}) = (a^*  \halfbraid{z}_a) \circ (\coev{a}  U(z)).\label{eq_def_brev}
\end{align}
In diagrammatic notation, we represent the braided evaluation and coevaluation as
\begin{equation}\label{eq_brev_cobrev_pics}
	\brev{z}{a} = \pica{1}{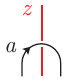} \qquad \text{and} \qquad \cobrev{z}{a} = \pica{1}{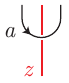},
\end{equation}
where we introduced the convention that objects of the center are drawn as red lines.
Pictures such as \eqref{eq_brev_cobrev_pics} containing objects in the center $\modulecent{\acat}{\mcat}$ of a bimodule category $\mcat$ are always diagrams in $\mcat$, not $\modulecent{\acat}{\mcat}$.
When evaluating such diagrams, the forgetful functor has to be applied to all objects labeling red lines, and crossings between red and black lines need to be replaced by balancings. Unlike in the diagrammatic calculus for a braided monoidal category, there are no over- and under-crossings that need to be distinguished.
\begin{Remark}\label{rem_modulecent_vect_trivial}
	Any linear category $\xcat$ can be seen as a $\vect$-$\vect$-bimodule category.
	With respect to these $\vect$-actions, every object $x\in \xcat$ admits a balancing, and all balancings are isomorphic. Hence, there is an equivalence of categories $\xcat \cong \modulecent{\vect}{\xcat}$.
\end{Remark}
\begin{Remark}
	An alternative model for the center is given by the category of bimodule functors and bimodule natural transformations $\bimodfun{\acat}{\acat}{\acat}{\mcat}$, see \cite[Section~ 2.8]{enom}: There is a canonical equivalence of categories
$
		\bimodfun{\acat}{\acat}{\acat}{\mcat} \cong \modulecent{\acat}{\mcat}\colon  F \mapsto F(\one)$.
\end{Remark}

\subsection{Balanced functors and their diagrams}\label{sec_balfuns}
The center of a module category $\mcat$ over a monoidal category $\acat$
can also be characterized by a universal property:
A \emph{balanced functor} consists of a functor $F\colon \mcat \to \xcat$, equipped with a~family of isomorphisms
$\gamma_{a,m} \colon F(am) \to F(ma)$, natural in $a \in \acat$ and $m\in \mcat$, which moreover satisfies $\gamma_{ab,m} = \gamma_{b,ma} \circ \gamma_{a,bm}$.
Balanced functors form a category $\balfun{\mcat}{\xcat}$, whose morphisms are balanced natural transformations, i.e., those natural transformations $\eta\colon F \Rightarrow G$ which are compatible with the balancings $\gamma^F$ and $\gamma^G$ of $F$ and $G$ in that $\gamma_{a,m}^G \circ \eta_{am} = \eta_{ma} \circ \gamma_{a,m}^F$.
Up to equivalence, the center $\modulecent{\acat}{\mcat}$ is the unique category such that there is an equivalence of functor categories
$\balfun{\mcat}{\xcat} \cong \funcat{\modulecent{\acat}{\mcat}}{\xcat}$.

Let $\mcat$ be a (left) module category over any monoidal category $\acat$, and let $\ncat$ be a right $\acat$-module category.
The category $\ncat\deligne \mcat$ is an $\acat$-$\acat$-bimodule category. The balanced functors out of categories of this form are of particular interest to us.
We usually denote the balancing of a functor $F\colon \ncat\deligne \mcat \to \xcat$ into a category $\xcat$ as $\bal_{n,a,m} \colon F(na,m) \to F(n,am)$.

The prototypical example of such a balanced functor is the Hom-functor of a module category~$\mcat$ over a rigid monoidal category $\acat$:
Let $\ncat = \opcat{\mcat}$, which is a right $\acat$-module category with action \smash{$\opcat{m} a := \opcat{a^*m}$}, where $a^*$ here denotes the \emph{right} dual.
Then the evaluation morphism defines isomorphisms
\begin{equation*}
	\bigl\langle\opcat{m}a,m'\bigr\rangle = \bigl\langle\opcat{a^*m},m'\bigr\rangle \cong \bigl\langle\opcat{m},am'\bigr\rangle,
\end{equation*}
which assemble into a balancing of the Hom-functor.
The Hom-functor of a bimodule category~$\bimod{\acat}{\mcat}{\bcat}$ has even more structure, as it is balanced with respect to the $\acat\deligne\revcat{\bcat}$-action on $\mcat$.
In general, if $\ncat$ is a $\bcat$-$\acat$-bimodule category and $\mcat$ is an $\acat$-$\bcat$-bimodule category, then we say a~functor
\begin{equation*}
	F \colon\ \bimod{\bcat}{\ncat}{\acat} \deligne \bimod{\acat}{\mcat}{\bcat} \to \xcat
\end{equation*}
is \emph{bi-balanced}, if it is equipped with a balancing between the right $\acat\deligne\revcat{\bcat}$-action on $\ncat$ and the left $\acat\deligne\revcat{\bcat}$-action on $\mcat$.

We also call functors
\begin{equation*}
	G \colon\ \bimod{\acat}{\ncat}{\acat} \deligne \bimod{\bcat}{\mcat}{\bcat} \to \xcat,
\end{equation*}
bi-balanced, if they are equipped with two balancings, one for $\mcat$ and one for $\ncat$, which are required to commute.
In the ambiguous case $\acat = \bcat$, a functor can be bi-balanced in either way. We introduce the following terminology
\begin{align}
	&\wick{F \colon\ \c1 {_\acat} {\ncat} \c1 {_\acat} \deligne \c2 {_\acat}{\mcat}\c2 {_\acat} \to \xcat} \qquad \text{``disconnected'' balancings} , \label{eq_bibal_ambiguous_disconn}\\
	&\wick{F \colon\ \c2 {_\acat} {\ncat} \c1 {_\acat} \deligne \c1 {_\acat}{\mcat}\c2 {_\acat} \to \xcat} \qquad \text{``connected'' balancings} .\label{eq_bibal_ambiguous_conn}
\end{align}

We next develop a graphical calculus for bi-balanced functors of the ``connected'' type that is tailored to our purposes.
In particular, given a bi-balanced functor $S\colon \opcat{\ncat} \times \mcat \to \xcat$, we will be able to pictorially represent those morphisms in $\xcat$ which are built from
\begin{itemize}\itemsep=0pt
	\item morphisms of the form $S(f,g)$ where $f$ and $g$ are morphisms in $\ncat$ and $\mcat$, respectively; and
	\item the balancing isomorphisms of $S$, and their inverses.
\end{itemize}
Our diagrams specialize to the diagrams we introduced in Section~\ref{sec_pictures} for the choice of the Hom-functor as bi-balanced functor $S$.

For morphisms $m_1 \xrightarrow{f} m_2 \xrightarrow{g} {m_3}$ in $\mcat$ and $n_2 \xrightarrow{h} n_1$ in $\ncat$, we write
\begin{equation*}
	\pica{1}{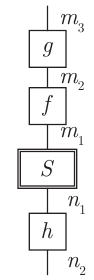} := S\left(h, \pica{1}{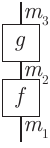} \right) = S(h, g\circ f),
\end{equation*}
which represents a morphism $S(n_1, m_1) \to S(n_2, m_3)$ in $\xcat$.
Given objects $a \in \acat$, $m \in \mcat$, and~${n \in \ncat}$, we draw the following diagram to represent the left balancing of $S$
\begin{equation}\label{eq_funcdiag_bal}
	\pica{1}{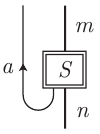} := \bal_{n,a,m} \colon\ S(a^*n,m) \to S(n,am).
\end{equation}
As usual, the juxtaposition of a strands labeled by $a\in \acat$ and $m \in \mcat$ stands for the object~${am \in \mcat}$.
For bi-balanced functors, bending around strands is allowed both on the left and on the right-hand side of diagrams.
To illustrate this, consider the following picture, which represents a morphism $S(a^*n,mb^*) \to S\big(nb,m'\big)$ in $\xcat$
\begin{equation}\label{eq_funcdiag_bibal_ex}
	\pica{1}{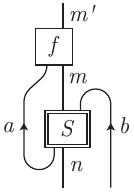} = S(f, nb) \circ \bal_{n,b,am}^{\text{right} -1} \circ \bal_{n,a,mb}^{\text{left}}.
\end{equation}
If we consider the special case where $\ncat = \mcat = \bcat = \bimod{\acat}{\acat}{\acat}$ is the regular bimodule category, and $\xcat = \vect$, we can choose $S$ to be the Hom-functor of $\acat$, which is a bi-balanced functor.
Diagrams such as \eqref{eq_funcdiag_bibal_ex} then specialize to the diagrams we encountered in Section~\ref{sec_pictures}, with the modification that the previously empty double-box is now labeled with the name of a functor, in this case Hom.
The reason why these diagrams really specialize to the diagrams from Section~\ref{sec_pictures} for the Hom-functor is that ``bending around strands'' as in \eqref{eq_funcdiag_bal} encodes the Hom-functor's balancing, which is defined using a coevaluation; on the other hand, the coevaluation is also represented by bending around a strand in the diagrams from Section~\ref{sec_pictures}, which are derived from string diagrams.
The Hom-functor of a module category $\mcat$ over $\acat$ also comes with a linear map $\bigl\langle m,m'\bigr\rangle \to \bigl\langle am,am'\bigr\rangle$ for each triple of objects $a\in \acat$, $m,m'\in \mcat$, which is part of the data of module action functor.
If $\acat$ is a rigid and pivotal, this linear map can be expressed using the balancing of the Hom-functor, together with either an evaluation or a coevaluation.
This is true of any balanced functor $S\colon \opcat{\ncat} \times \mcat \to \xcat$ over a pivotal category $\acat$:
The two expressions
\begin{equation*}
	\bal_{an,a,m} \circ S(\coev{a} n, m)= \pica{1}{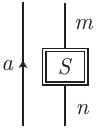} = \bal^{-1}_{n,a^*, am} \circ S(n, \ev{a^*} m)
\end{equation*}
are equal.

Of course, diagrams for bi-balanced functors can be composed, and the composition operation is the same we already encountered in \eqref{eq_emptybox_comp}: inserting the pre-composed diagram into the double-box of the post-composed diagram.
In this way, our diagrams are similar to other approaches to graphical calculi, such as the ``corollas'' appearing, for example, in \cite[for the composition see Example 2.14]{fsy-stringnets}.

We leave it to the reader to convince themself that the evaluation of diagrams for bi-balanced functors is well defined and that isotopic diagrams evaluate to equal morphisms.
In any case, the diagrams throughout this work are meant as illustrations that can, at any point, be translated into standard notation.

\subsection{(Split) equalizers of bi-balanced functors}\label{sec_spleqs}
Let now $\bimod{\acat}{\ncat}{\acat}$ and $\bimod{\acat}{\mcat}{\acat}$ be $\acat$-$\acat$-bimodule categories, and let $S \colon \opcat{\ncat} \times \mcat \to \xcat$ be a bi-balanced functor (in the ``connected'' sense of \eqref{eq_bibal_ambiguous_conn}).
From $S$, we will now construct another bilinear functor $\eq S \colon Z_\acat \bigl(\opcat{\ncat}\bigr) \times \modulecent{\acat}{\mcat} \to \xcat$, to be called the \emph{equalizer} of $S$.

Given objects $x \in \modulecent{\acat}{\ncat}$ and $y \in \modulecent{\acat}{\mcat}$, the value $\eq S(x,y) \in \xcat$ is the essentially unique object such that the diagram
\begin{equation}\label{diag_balanced_limit}
	\begin{tikzcd}
		\eq S(x,y) \arrow[r,hook] & S(U(x),U(y)) \arrow[rrr, shift left=2mm, "f= \bigoplus_\si{a} {\pica{0.8}{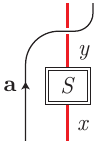}}"{above}] \arrow[rrr, shift right=2mm, "g= \bigoplus_\si{a}{\pica{0.8}{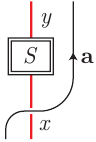}}"{below}] & & & \bigoplus_\si{a} S(\si{a}U(x), U(y)\si{a})
	\end{tikzcd}
\end{equation}
is an equalizer.
It is straightforward to verify that a morphism $\mu\colon y \to y'$ in $\modulecent{\acat}{\mcat}$ defines a~morphism of equalizer diagrams, making $\eq S$ functorial in the second (and in the same way, the first) argument.

\begin{Example}\label{ex_centerhom_is_equalizer}
	If $\mcat = \ncat$ and $S\colon \opcat{\mcat}\times \mcat \to \vect$ is the Hom-functor, then the equalizer of $S$ on $x,y \in \modulecent{\acat}{\mcat}$ is the Hom-functor of the center: $\eq \homsetin{\mcat}{-}{-} \cong \homsetin{\modulecent{\acat}{\mcat}}{-}{-}$.
\end{Example}

\begin{Remark}\label{rem_bal_trafo_maps_equalizers}
	The assignment $S \mapsto \eq S$ is functorial:
	A bi-balanced natural transformation~${\eta \colon S\Rightarrow S'}$ from $S$ to another bi-balanced functor $S'$ defines a morphism of diagrams between the diagram \eqref{diag_balanced_limit} and the diagram obtained from \eqref{diag_balanced_limit} by replacing $S$ with $S'$.
	Therefore, $\eta$~defines a morphism from the equalizer of $S$ to the equalizer of $S'$.
\end{Remark}

To proceed, recall the concept of a split equalizer \cite[Section~VI.6]{maclane-cats}.
\begin{Definition}\label{def_split_equalizer}
	Let $u$, $v$, $w$ be objects of any category $\xcat$.
	A \emph{split equalizer} is a diagram of the form
	\begin{equation*}
		\begin{tikzcd}
			u \arrow[r, "e"{below}] & v \arrow[r, shift left=2mm, "f"{below}] \arrow[r, shift right=2mm, "g"{below}] \arrow[l, bend right, "r"{above}] & w \arrow[l, bend right, shift right=3mm, "t"{above}],
		\end{tikzcd}
	\end{equation*}
	where $t\circ f = \id_v$, $t\circ g = e \circ r$, $f\circ e = g\circ e$, and $r\circ e = \id_u$.
	In particular, $r$ and $t$ are retracts of~$e$ and $f$, respectively.
	
	A \emph{contractible pair} is a pair of parallel morphisms $f,g \colon v\to w$, together with a morphism~${t\colon w \to v}$, such that
	an equalizer of $f$ and $g$ exists, $t\circ f = \id_v$, and $f \circ t \circ g = g\circ t \circ g$.
\end{Definition}
The following result is well known, see \cite[Section~VI.6, Exercise~2]{maclane-cats}.
\begin{Lemma}\label{lem_split_eqs_and_contractible_pairs}
	Split equalizers are in one-to-one correspondence with contractible pairs with a~choice of equalizer.
\end{Lemma}
The endomorphism
\begin{equation}\label{eq_idempotent_abstract}
	h := t \circ g,
\end{equation}
of $v$, which is defined both for split equalizers and contractible pairs, plays a central role in the proof of Lemma~\ref{lem_split_eqs_and_contractible_pairs}.
This is because $h$ is an idempotent, and its image is $u$.
We are now in a position to prove the next proposition, which is the reason why we are interested in split equalizers.
\begin{Proposition} \label{lem_speqs_of_bibals}
	Let $\acat$ be a spherical fusion category, and let $\mcat$ and $\ncat$ be exact $\acat$-$\acat$-bimodule categories.
	Let $S\colon \opcat{\ncat} \times \mcat \to \xcat$ be a bi-balanced functor into a linear category $\xcat$.
	Then the spherical structure on $\acat$ determines a splitting of the equalizer \eqref{diag_balanced_limit}.
	Hence, the spherical structure on $\acat$ exhibits the equalizer $\eq S$ of $S$ not only as a sub-object, but as a retract.
	
	For objects $x \in \modulecent{\acat}{\ncat}$ and $y \in \modulecent{\acat}{\mcat}$, the idempotent
	\begin{equation*}
		h \in \homsetin{\xcat}{S(U(x), U(y))}{S(U(x), U(y))}
	\end{equation*}
	from \eqref{eq_idempotent_abstract} is in this case given by
	\begin{equation}\label{eq_idemp}
		h= \frac{1}{\catdim{\acat}}  \sum_\si{a} \pdim{\si{a}}  \pica{1}{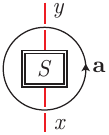}.
	\end{equation}
\end{Proposition}
\begin{proof}
	We call the top and bottom arrows of the equalizer diagram \eqref{diag_balanced_limit} $f$ and $g$, respectively.
	By Lemma~\ref{lem_split_eqs_and_contractible_pairs}, we only need to build a morphism $t\colon \bigoplus_\si{a} S(\si{a} U(x), U(y)\si{a}) \to S(U(x),U(y))$ using the pivotal structure of $\acat$, such that $(f,g,t)$ form a contractible pair.
	We define $t$ by
	\begin{equation*}
		t := \frac{1}{\catdim{\acat}}  \sum_\si{a}\pdim{\si{a}} \pica{1}{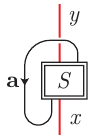}\colon\ S(\si{a} U(x), U(y)\si{a}) \to S(U(x),U(y)).
	\end{equation*}
	To prove that $(f,g,t)$ is a contractible pair, two conditions need to be checked.
	The first is that~$t$ is a retraction of $f$
	\begin{align}\label{eq_tmap_retract}
		\begin{aligned}
			t\circ f
			&= \frac{1}{\catdim{\acat}}  \sum_\si{a} \pdim{\si{a}}  \pica{1}{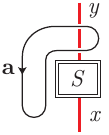} = \id_{S(U(x),U(y))}.
		\end{aligned}
	\end{align}
	
	The second condition is that $f\circ t \circ g = g \circ t \circ g$ holds.
	To verify this, we abbreviate $h:= t\circ g$, which is an endomorphism of $S(U(x),U(y))$, and find by a calculation similar to \eqref{eq_tmap_retract} that~$h$ is indeed given by \eqref{eq_idemp}.
	Since $\frac{1}{\catdim{\acat}}$ is a global factor in the following computation in the Hom-space~$\homsetin{\xcat}{S(U(x),U(y))}{\bigoplus_\si{b} S(\si{b} U(x), U(y)\si{b})}$, we omit it
	\begin{align*}
		\begin{aligned}
			f\circ h &= \bigoplus_\si{b} \sum_\si{a}\pdim{\si{a}}  \pica{1}{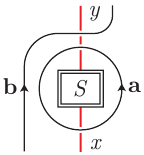} = \bigoplus_\si{b} \sum_{\si{a},\si{c}}\pdim{\si{a}}\pdim{\si{c}}  \pica{1}{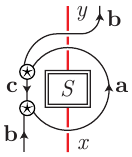} \\
			&= \bigoplus_\si{b} \sum_{\si{a},\si{c}}\pdim{\si{a}}\pdim{\si{c}}  \pica{1}{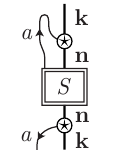} = \bigoplus_\si{b} \sum_{\si{c}}\pdim{\si{c}}  \pica{1}{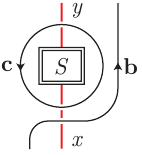} = g\circ h.
		\end{aligned}
	\end{align*}
	Hence, $f\circ h = g \circ h$, and the proof is complete.
\end{proof}

\begin{Remark}\label{rem_retr_indep_of_sph}
	The idempotent $h$ from \eqref{eq_idemp} is independent of the choice of spherical structure on $\acat$.
	This follows from a direct calculation using the insight that the square of the dimension of a~simple object $\si{a} \in \acat$ is equal for all spherical structures, which can be found in \cite[Section~7.21]{egno}.
\end{Remark}

\begin{Lemma}\label{lemma_homspace_projection}
	Given an $\acat$-$\acat$-bimodule category $\mcat$, the inclusion of the Hom-space
	\begin{equation*}
		\homsetin{\modulecent{\acat}{\mcat}}{x}{y} \hookrightarrow \homsetin{\mcat}{U(x)}{U(y)}
	\end{equation*}
	has, for a given spherical structure on $\acat$, a canonical retraction $r \colon \homsetin{\mcat}{U(x)}{U(y)} \to \homsetin{\modulecent{\acat}{\mcat}}{x}{y}$, given by
	\begin{align} \label{eq_homspace_projection}
			r (f) &= \frac{1}{\catdim{\acat}} \sum_\si{a} \pdim{\si{a}} \brev{y}{\si{a}} \circ ( \si{a} f  \si{a}^* ) \circ \cobrev{x}{\si{a}^*} = \frac{1}{\catdim{\acat}} \sum_\si{a} \pdim{\si{a}} \pica{1}{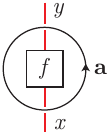}  .
	\end{align}
\end{Lemma}
This generalizes \cite[Lemma~2.2]{balskir}, and is a direct corollary of Lemma~\ref{lem_speqs_of_bibals}: The bi-balanced functor in this case is the Hom-functor in $\mcat$, as discussed in Example~\ref{ex_centerhom_is_equalizer}.

Lemma~\ref{lemma_homspace_projection} can be used to investigate the center of a \emph{traced} bimodule category.
Concretely, we find that the restriction of a bimodule trace to the center is again a trace:
\begin{Lemma}\label{lem_center_calabiyau}
	Let $\acat$ be a spherical fusion category, and let $\mcat$ be a traced $\acat$-$\acat$-bimodule category.
	The center $\modulecent{\acat}{\mcat}$ can be equipped with the structure of a Calabi--Yau category, whose trace function is defined, for an endomorphism $f \colon x \to x$ in $\modulecent{\acat}{\mcat}$, by
	\begin{equation}\label{eq_center_trace}
		\trace{\modulecent{\acat}{\mcat}}(f) := \frac{1}{\catdim{\acat}}\trace{\mcat}(f).
	\end{equation}
\end{Lemma}
Of course, the normalization factor $\frac{1}{\catdim{\acat}}$ in \eqref{eq_center_trace} is a convention.
The reason for this particular choice will be made clear by Remark~\ref{rem_unitor_preserves_trace}.
Before we prove Lemma~\ref{lem_center_calabiyau}, we need some intermediate results.
First, observe that the idempotent $h\colon \homsetin{\mcat}{U(x)}{U(y)} \to \homsetin{\mcat}{U(x)}{U(y)}$ leaves the trace invariant.

\begin{Lemma}\label{lem_trace_of_retraction}
	Let $x \in \modulecent{\acat}{\mcat}$ and let $f$ be an endomorphism of $U(x)$ in $\mcat$.
	Then
	\begin{equation*}
		\trace{\mcat}(f) = \trace{\mcat}(h(f)).
	\end{equation*}
\end{Lemma}
\begin{proof}
	The proof is a direct calculation
	\begin{align*}
			\trace{\mcat}(h(f)) &= \frac{1}{\catdim{\acat}} \sum_\si{a} \pdim{\si{a}} \trace{\mcat}\Left({4}\pica{1}{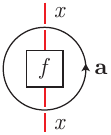}\Right){4} = \frac{1}{\catdim{\acat}} \sum_\si{a} \pdim{\si{a}} \trace{\mcat}\Left({4}\pica{1}{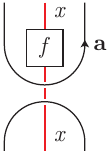}\Right){4}\\
			&= \frac{1}{\catdim{\acat}} \sum_\si{a} \pdim{\si{a}} \trace{\mcat}\Left({4}\pica{1}{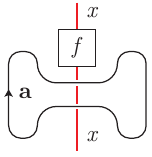}\Right){4} = \trace{\mcat}(f).
	\end{align*}
	The first equality is \eqref{eq_homspace_projection}, the second uses the symmetry of the trace, and the third uses the compatibility \eqref{eq_module_trace_cond} of the bimodule trace on $\mcat$ with the pivotal structure on $\acat$.
\end{proof}

It is an even more straightforward observation that, for objects $x,y,z \in \modulecent{\acat}{\mcat}$, a morphism~${f\colon y \to z}$ in $\modulecent{\acat}{\mcat}$, and a morphism $g\colon U(x) \to U(y)$ in $\mcat$, the equality
\begin{equation}\label{eq_retraction_sesquimult}
	f\circ r(g) = r(f\circ g)
\end{equation}
holds.
Now we are in a position to prove the lemma.
\begin{proof}[Proof of Lemma~\ref{lem_center_calabiyau}]
	It is clear that the candidate $\trace{\modulecent{\acat}{\mcat}}$ for the trace function defined in~\eqref{eq_center_trace} inherits the symmetry property from the trace $\trace{\mcat}$ on $\mcat$.
	It remains to check whether it is also non-degenerate.
	To this end, let $f\colon y\to z$ be a morphism as before.
	We need to check that if for every $g \in \homsetin{\modulecent{\acat}{\mcat}}{z}{y}$, we have $\trace{\modulecent{\acat}{\mcat}}(f \circ \tilde{g}) = 0$, then $f=0$.
	As the normalization factor~$\frac{1}{\catdim{\acat}}$ in the definition of $\trace{\modulecent{\acat}{\mcat}}$ is not relevant for the validity of the statement, the condition $\forall \tilde{g}\in \homset{z}{y}, \trace{\modulecent{\acat}{\mcat}}(f \circ \tilde{g}) = 0$ is equivalent to
	\begin{equation}\label{eq_centercalabiyau_calc_0}
		\forall \tilde{g}\in \homset{z}{y}, \trace{\mcat}(f \circ \tilde{g}) = 0.
	\end{equation}
	Since $r$ is surjective, there is some $g \in \homsetin{\mcat}{U(z)}{U(y)}$ such that $\tilde{g} = r(g)$.
	Hence, \eqref{eq_centercalabiyau_calc_0} is equivalent to $\forall g\in \homset{U(z)}{U(y)}, \trace{\mcat}(f \circ r(g)) = 0$.
	Moreover, from \eqref{eq_retraction_sesquimult} and Lemma~\ref{lem_trace_of_retraction}, we know that
	\begin{equation*}
		\trace{\mcat}(f \circ r(g)) = \trace{\mcat}(r(f \circ g)) = \trace{\mcat}(f \circ g).
	\end{equation*}
	Thus, the statement we want to prove is that
	$\forall g\in \homset{U(z)}{U(y)}, \trace{\mcat}(f \circ g) = 0$ implies $f=0$.
	This is true by the non-degeneracy of the trace on $\mcat$.
\end{proof}

\subsection{Shifting actions under (co-)ends}
In this section, we are concerned with a structure which looks like, but is not quite, a balancing.
Let $\mcat$ be a left module category over a pivotal fusion category $\acat$, and let $S \colon \opcat{\mcat} \times \mcat \to \xcat$ be a~bilinear functor.
We define isomorphisms
\begin{equation*}
	\coendbal_a \colon\ \bigoplus_\si{m} S(\si{m},a\si{m}) \to \bigoplus_\si{m} S(a^*\si{m},\si{m})
\end{equation*}
for $a \in \acat$ in the following way.
Pick an arbitrary (auxiliary) Calabi--Yau structure on $\mcat$.
Then define $\coendbal_a$ in terms of their matrix elements
\begin{align}
		\cte{(\coendbal_a)}{\si{n}}{\si{k}}:={}& \pdim{\si{n}} S(\basisel{a^*\si{k}}{\si{n}}, (\ev{a} \si{k}) \circ (a \basisel{\si{n}}{a^*\si{k}})) \colon\  S(\si{n}, a\si{n}) \to S(a^* \si{k} , \si{k})\nonumber\\
={}& \pdim{\si{n}} \pica{1}{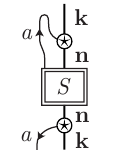}.\label{eq_beta_components}
\end{align}
Although the definition \eqref{eq_beta_components} seems ad hoc, the expression arises naturally when considering the term $\bigoplus_\si{m} S(\si{m},a\si{m})$ as an end, as is done in \cite[Section~2.11]{jf-thesis}.
It is then also obvious that~$\coendbal_a$ does not depend on the auxiliary Calabi--Yau structure.

The morphisms $\coendbal_{a}$ are of interest to us for some particular choices of $S$, which we now discuss.
\begin{itemize}\itemsep=0pt
	\item For $S = - \deligne -$, the family of isomorphisms $\coendbal_{a}$ forms a balancing for the object $\bigoplus_\si{m} \opcat{\si{m}} \deligne \si{m}$ of $\opcat{\mcat} \deligne \mcat$.
	Their components are given by
	\begin{align}
			\cte{(\coendbal_a)}{\si{n}}{\si{k}} &= \pdim{\si{n}} \basisel{a^*\si{k}}{\si{n}} \deligne ((\ev{a} \si{k}) \circ (a  \basisel{\si{n}}{a^*\si{k}})).\label{eq_ewid_balancings_components}
	\end{align}
	
	\item Let $\acat$ be a spherical fusion category, let $\bimod{\acat}{\ncat}{\acat}$ be an $\acat$-$\acat$-bimodule category, fix $m\in \ncat$ and set $S_m := -  m  (-)^*$ as a functor $\lmod{\acat}{\acat} \times \rmod{\opcat{\acat}}{\acat} \to \bimod{\acat}{\ncat}{\acat}$.
	The morphisms $\coendbal_a$ associated with~$S_m$ again assemble into a balancing.
	We can thus define a functor $\leftind\colon \ncat \to \modulecent{\acat}{\ncat}$, called the \emph{induction functor} via
	\begin{equation}\label{eq_induction_functors}
		\leftind(m) := \biggl(\bigoplus_\si{a} \si{a} m \si{a}^*, \coendbal\biggr) .
	\end{equation}
	The components of the balancings are given by
	\begin{equation*}
		\cte{(\coendbal_a)}{\si{b}}{\si{c}} = \pdim{\si{b}}\basisel{a^*\si{c}}{\si{b}} m ((\ev{a} \si{c}) \circ (a  \basisel{\si{b}}{a^*\si{c}})).
	\end{equation*}
	The induction functor is left and right adjoint to the forgetful functor $U \colon \modulecent{\acat}{\ncat} \to \ncat$
	\begin{equation} \label{eq_induction_adjunction}
		\leftind \dashv U \dashv \rightind.
	\end{equation}
	
	The adjunctions \eqref{eq_induction_adjunction} are witnessed by isomorphisms
	\begin{equation} \label{eq_induction_adjunction_r_iso}
		\indadjr{m}{x} \colon\ \homsetin{\ncat}{m}{U(x)} \to \homsetin{\modulecent{\acat}{\ncat}}{\leftind(m)}{x}
	\end{equation}
	and
	\begin{equation} \label{eq_induction_adjunction_l_iso}
		\indadjl{m}{x} \colon\ \homsetin{\ncat}{U(x)}{m} \to \homsetin{\modulecent{\acat}{\ncat}}{x}{\rightind(m)},
	\end{equation}
	natural in $m \in \ncat$ and $x \in \modulecent{\acat}{\ncat}$.
	For morphisms $f \colon m \to U(x)$, $g\colon U(x) \to m$, $f' \colon \leftind(m) \to x$ and $g' \colon x \to \rightind(m)$, these adjunction isomorphisms are given explicitly by the following expressions
	\begin{align}
		\indadjr{m}{x} (f) = \bigoplus_\si{a} \pdim{\si{a}} \brev{x}{\si{a}} \circ (\si{a}  f  \si{a}^*) &\qquad \text{and} \qquad
		(\indadjr{m}{x})^{-1}\big(f'\big) = f' \circ \coincl{\one},\nonumber
		\\
		\indadjl{m}{x} (g) = \bigoplus_\si{a} (\si{a}  g  \si{a}^*) \circ \cobrev{x}{\si{a}^*} &\qquad \text{and} \qquad
		\bigl(\indadjl{m}{x}\bigr)^{-1}\big(g'\big) = \projend{\one} \circ g'.\label{eq_def_indadjr}
	\end{align}
\end{itemize}

\begin{Remark}\label{rem_shifting_actions_are_balancings}
	The morphisms $\coendbal_a$ can be understood as components of a balancing of the functor $\bigoplus_\si{m} S(\si{m} - , - \si{m})\colon \opcat{\acat} \times \acat \to \xcat$.
\end{Remark}

\subsection{Relative Deligne products} \label{sec_tricat}
Let $\bcat$ be a spherical fusion category, and let $\rmod{\mcat}{\bcat}$ be a right and $\lmod{\bcat}{\ncat}$ be a left $\bcat$-module category.
We make use of the \emph{relative Deligne product} of $\mcat$ and $\ncat$ \cite{dss-balanceddeligne}.
For us, it is convenient to realize it as the center of the $\bcat$-$\bcat$-bimodule category $\rmod{\mcat}{\bcat} \deligne \lmod{\bcat}{\ncat}$
\begin{equation}\label{eq_def_reldel}
	\mcat \reldelov{\bcat} \ncat := \modulecent{\bcat}{\rmod{\mcat}{\bcat} \deligne \lmod{\bcat}{\ncat}}.
\end{equation}
This makes it obvious that there is a forgetful functor into the ordinary Deligne product
\[
U\colon\ \mcat \reldelov{\bcat} \ncat \to \mcat \deligne \ncat,
\]
 which has the left and right adjoint $\leftind$ from \eqref{eq_induction_adjunction}.
In notation, we sometimes omit the spherical fusion category $\bcat$ when there is no ambiguity, simply writing
\smash{$\mcat \reldel \ncat := \mcat \reldelov{\bcat} \ncat$}.
By Remark~\ref{rem_modulecent_vect_trivial}, the relative Deligne product specializes to the ordinary Deligne product in the case that $\bcat = \vect$, i.e.,
\smash{$
\mcat \reldelov{\vect} \ncat \cong \mcat \deligne \ncat$}.

Let now $\acat$ and $\ccat$ be additional spherical fusion categories, and suppose that $\bimod{\acat}{\mcat}{\bcat}$ and $\bimod{\bcat}{\ncat}{\ccat}$ are bimodule categories.
The category $\mcat \reldel \ncat$ then still caries a left $\acat$-action inherited from $\mcat$ and a right $\ccat$-action inherited from $\ncat$, and is thus an $\acat$-$\ccat$-bimodule category.
If $\acat = \ccat$, there is a canonical equivalence
\smash{$\modulecent{\acat}{\mcat \reldelov{\bcat} \ncat} \cong \modulecent{\bcat}{\ncat \reldelov{\acat}\mcat}$}.
We also use ``dangling product'' notation for this category, as in
\begin{equation}\label{eq_dangling_product}
	\mcat \reldel \ncat \reldel := \modulecent{\acat}{\mcat \reldelov{\bcat} \ncat}.
\end{equation}
As a shorthand for the induction functor $\leftind$ from \eqref{eq_induction_adjunction}, we write for $m\in \mcat$ and $n \in \ncat$
\begin{equation}\label{eq_induction_reldel_objects}
	m \rexreldel n := \leftind(m\deligne n).
\end{equation}
The dangling product notation from \eqref{eq_dangling_product} can be used here as well: We may write $m \rexreldel n \rexreldel$ for an object in $\mcat \reldel \ncat \reldel$.

The relative Deligne product is associative (in the sense that it comes with an appropriately coherent set of equivalences witnessing associativity) and unital with respect to the regular bimodule category $\bimod{\acat}{\acat}{\acat}$.
We slightly abuse notation by treating the relative Deligne product as if it were strictly associative: we do not pay attention to bracketing and make no explicit use of associators.
Unitality requires an equivalence of $\acat$-$\bcat$-bimodule categories
\begin{equation}\label{eq_reldel_unitality}
	\lmod{\acat}{\acat} \reldel \rmod{\mcat}{\bcat} \cong \bimod{\acat}{\mcat}{\bcat}
\end{equation}
for every $\acat$-$\bcat$-bimodule category $\mcat$.
This equivalence is given, on objects of the form ${a \lexreldel m \in \acat \reldel \mcat}$, by
\begin{equation}\label{eq_reldel_unitality_explicit}
	a \lexreldel m \mapsto am \in \mcat, \qquad \text{with pseudo-inverse} \qquad n \mapsto \one \lexreldel n \in \acat \reldel \mcat,
\end{equation}
for $n \in \mcat$.

A construction of \emph{a} trace on $\mcat \reldel \ncat$ is provided by Lemma~\ref{lem_center_calabiyau}, but it is not immediately clear that this trace respects the $\acat$-$\ccat$-bimodule structure.
On the other hand, a \emph{bimodule} trace for $\mcat \reldel \ncat$ is constructed in \cite[Proposition~4.10.1]{schaumann-bimod}.
These traces are indeed equal, making calculations involving the trace on the relative Deligne product easy to perform:
\begin{Lemma}[{\cite[Lemma~2.28]{jf-thesis}}]\label{lem_trace_on_reldel}
	The bimodule trace on \smash{$\mcat \reldelov{\bcat} \ncat$} defined in {\rm\cite[\emph{Proposition}~4.10.1]{schaumann-bimod}} is the trace for the center of a traced bimodule category defined in Lemma~{\rm\ref{lem_center_calabiyau}}.
	This implies the following property: For endomorphisms $f \colon m\to m$ in $\mcat$ and $g\colon n\to n$ in $\ncat$,
	\begin{equation*}
		\trace{\mcat\reldel\ncat} (f \rexreldel g) = \trace{\mcat}(f) \trace{\ncat}(g).
	\end{equation*}
	In particular,
$
		\pdim{m \rexreldel n} = \pdim{m}  \pdim{n}$.
\end{Lemma}
\begin{Remark}\label{rem_unitor_preserves_trace}
	The unitor equivalence
$\acat \reldel \acat \cong \acat
$
	from \eqref{eq_reldel_unitality} for a spherical fusion category~$\acat$ is an equivalences between two traced bimodule categories. Lemma~\ref{lem_trace_on_reldel} implies that this equivalence preserves the trace.
	This compatibility is the reason for the normalization factor $\frac{1}{\catdim{\acat}}$ in the definition \eqref{eq_center_trace} of the Calabi--Yau structure for the center of a traced bimodule category.
\end{Remark}

\begin{Remark}[Sweedler notation for forgetful functors.]\label{rem_sweedler}
	Due to its definition as a center \eqref{eq_def_reldel}, the relative Deligne product \smash{$\mcat \reldelov{\bcat} \ncat$} comes with a forgetful functor \smash{$U \colon \mcat \reldelov{\bcat} \ncat \to \mcat \deligne \ncat$}.
	We will often use a form of Sweedler notation for this functor, writing
	\begin{equation}\label{eq_reldel_forget_sweed}
		U(x) = \sweed{x}{\mcat} \deligne \sweed{x}{\ncat}
	\end{equation}
	for $x \in \mcat \reldel \ncat$.
\end{Remark}

\subsection{Silent objects}\label{sec_eilenwatts}
\label{sec_silentobs}
\label{sec_genyon}
\label{sec_gencomp}
Let $\xcat$ be a linear category.
In accordance with \cite[Section~5.19]{fss-statesum}, we call the object
\begin{align}\label{eq_def_silentobjs}
		\ewid{\xcat} &:= \bigoplus_\si{x} \opcat{\si{x}} \deligne \si{x} \in \opcat{\xcat} \deligne \xcat
\end{align}
the \emph{silent} object of the category $\opcat{\xcat} \deligne \xcat$.
We are interested in silent objects for two reasons.

The first is the natural isomorphism
\begin{equation*}
	\soc \colon\ \homsetin{\xcat}{-}{-} \cong \homsetin{\opcat{\xcat}\deligne \xcat}{-\deligne-}{\ewcoid{\xcat}}
\end{equation*}
between Hom-functors, which we will use frequently.
Explicitly, $\soc$ is given on a morphism $f\colon x\to x'$ and in components for simple $\si{y}\in \xcat$ by
\begin{equation}\label{eq_sil}
	\tc{\soc(f)}{\si{y}} = \pdim{\si{y}}  \basisel{\si{y}}{x'} \deligne (\basisel{x'}{\si{y}} \circ f)
\end{equation}
with inverse
\begin{gather*}
	\soc^{-1}(g) = \sum_{\si{y}} \sweed{(\tc{g}{\si{y}})}{\opcat{\xcat}} \circ \sweed{(\tc{g}{\si{y}})}{\xcat}
\end{gather*}
with $g = \sweed{g}{\opcat{\xcat}} \otimes \sweed{g}{\xcat}  \colon \opcat{x'} \deligne x \to \bigoplus_\si{y} \opcat{\si{y}} \deligne \si{y}$.

The second reason has to do with the \emph{Eilenberg--Watts equivalence}
\begin{equation*}
	\ew \colon\  \funcat{\xcat}{\ycat} \cong \opcat{\xcat} \deligne \ycat
\end{equation*}
for linear categories $\xcat$ and $\ycat$, see \cite{fss-ew}.
The equivalence is defined by
\begin{equation}\label{eq_ew_exp}
	\ew (F) := \bigoplus_\si{x} \opcat{\si{x}} \deligne F(\si{x}),
\end{equation}
and comparison with \eqref{eq_def_silentobjs} reveals that $\ewid{\xcat} = \ew (\id_\xcat)$.

The silent object has more structure if the category $\xcat$ is instead a module category $\lmod{\acat}{\mcat}$.
In this case, we saw in \eqref{eq_ewid_balancings_components} that the object $\ewid{\mcat}$ of the $\acat$-$\acat$-bimodule category $\opcat{\mcat} \deligne\mcat$ has a~balancing $\beta$.
In a slight abuse of terminology,
\begin{equation}\label{eq_def_zewid}
	\zewid{\mcat} := (\ewid{\mcat}, \beta) \in Z_\acat\bigl(\opcat{\mcat} \deligne\mcat\bigr) = \opcat{\mcat} \reldel \mcat
\end{equation}
will also be called a \emph{silent} object.
A direct consequence of \eqref{eq_def_zewid} is the relation $U(\zewid{\mcat}) = \ewid{\mcat}$ between the two notions of silent objects.

There is also a \emph{module Eilenberg--Watts equivalence}
\begin{gather}\label{eq_mew}
	\mew \colon\ \lmodfun{\acat}{\mcat}{\ncat} \cong \opcat{\mcat} \reldel \ncat
\end{gather}
for left $\acat$-module categories $\mcat$ and $\ncat$.
The equivalence \eqref{eq_mew} is constructed from \eqref{eq_ew_exp} by noticing that the module constraint of a module functor corresponds to a balancing of the object in the Deligne product.

We conclude this section with a remark on the dimensions of silent objects.
\begin{Remark}\label{rem_discrepancies_in_dimensions}
	Let $\xcat$ be a Calabi--Yau category.
	It is easy to see that
$
		\pdim{\ewid{\xcat}} = \catdim{\xcat}$.
	If $\mcat$ is a~traced left module category over a spherical fusion category $\acat$, then by Lemma~\ref{lem_trace_on_reldel},
$
		\pdim{\zewid{\mcat}} = \frac{\catdim{\mcat}}{\catdim{\acat}}$.
	If $\mcat$ is irreducible, then $\mathrm{Fun}_\acat(\mcat,\mcat)$ is a spherical fusion category with respect to a canonical pivotal structure that does not depend on the choice of trace on $\mcat$ \cite[Proposition~5.10]{schaumann-modtraces}.
	Therefore, the monoidal unit $\id_\mcat$ has pivotal dimension 1.
	This shows that in general, the module Eilenberg--Watts equivalences \eqref{eq_mew} do not preserve traces.
	At several points, categorical dimension factors will be present in our equations, whose appearance is rooted in this fact.\looseness=1
\end{Remark}

\section{Extruded graphs and their evaluation} \label{sec_construction}

\subsection{Defect manifolds and extruded graphs}\label{sec_extgraphs}
The evaluation procedure to be constructed in Section~\ref{sec_eval} is defined for surfaces endowed with a~specific type of labeled graphs which we call extruded graphs.
By a~surface, we always mean a~smooth, oriented, compact 2-manifold $\sigsurf$, possibly with boundary.

We now introduce additional structure on $\sigsurf$: in a first step, we allow for a finite number of embedded disks in its interior $\mathrm{Int}(\sigsurf)$ and a finite number of embedded intervals on its boundary~$\partial \sigsurf$. The embedded disks and intervals are called \emph{nodes}. The surface pictured in Figure~\ref{pic_nodesurf} has three nodes: two disks, and one interval.
\begin{figure}[ht]\centering
\includegraphics[scale=1.5]{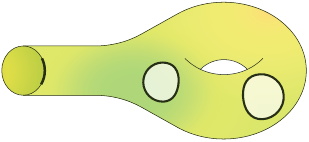}
\caption{}\label{pic_nodesurf}
\end{figure}

An unlabeled \emph{defect surface} has even more structure: it is endowed with finitely many embedded disjoint oriented intervals, called \emph{defect lines}, that must meet the following criteria:
\begin{itemize}\itemsep=0pt
	\item The end points of each defect line must lie on the boundary of a disk-shaped node or in the interior of an interval node.
	\item The interior of the defect lines must be disjoint from the disk and interval nodes.
	\item Each disk node must contain at least one end-point of a defect line.
	\item The boundary $\partial \sigsurf$ of $\sigsurf$ is covered entirely by interval nodes and defect lines.
\end{itemize}
The connected components of the complement of defect lines and disk and interval nodes in $\sigma$ are called \emph{domain}s.
In addition to these requirements, we only allow \emph{fine} unlabeled defect surfaces, which means that all domains are disks (see Figure~\ref{pic_unl_defsurf}).
\begin{figure}[ht]\centering
\includegraphics[scale=1.5]{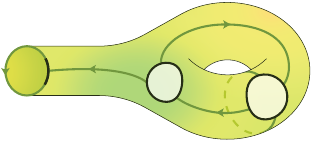}

\caption{}\label{pic_unl_defsurf}
\end{figure}

We introduce the following labeling of defect surfaces by three labels of algebraic data:
\begin{enumerate}\setlength{\leftskip}{0.36cm}\itemsep=0pt
	\item[Level 1] To each of the (two-dimensional) domains, we assign a spherical fusion category $\acat, \bcat, \dots$.
	\item[Level 2] To each defect line adjacent to a domain labeled by~$\acat$ on the left and a domain labeled by~$\bcat$ on the right, we assign a traced $\acat$-$\bcat$-bimodule category $\mcat$.
	In case the defect line lies on the boundary of the surface, there is only one adjacent domain, and the algebraic label is a~one-sided (left or right) module category, depending on the orientation of the defect line.
\end{enumerate}
We illustrate this labeling in the example (Figure~\ref{pic_defsurf}), which contains only one domain, labeled by~$\acat$.
In this case, all bimodule categories are thus $\acat$-$\acat$-bimodule categories, with the exception of $\mcat_1$, which labels a defect line on the boundary and is thus merely a traced right $\acat$-module category.
\begin{figure}[ht]\centering
\includegraphics[scale=1.1]{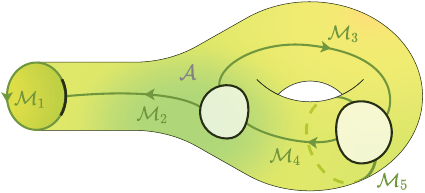}
\caption{}\label{pic_defsurf}
\end{figure}

The complete labeling of a defect surface as in Figure~\ref{pic_extgraph} requires one more layer of algebraic data.
\begin{figure}[ht]\centering
\includegraphics[scale=1.1]{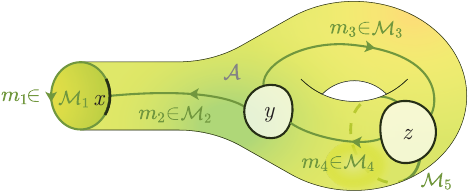}
\caption{}\label{pic_extgraph}
\end{figure}
\begin{enumerate}\setlength{\leftskip}{0.36cm}\itemsep=0pt
	\item[Level 3] To each defect line labeled by a bimodule category $\mcat$, we assign in addition an object~${m \in \mcat}$.
		For each node $L$, we describe a category $\raycat(L)$, and then assign an object $x \in \raycat(L)$ to the node.
	The category $\raycat(L)$ is obtained as follows.
	A node of a defect surface with algebraic labels introduced in the levels 1 and 2 inherits the labels illustrated in the following figure:
	\begin{equation*}
		\pica{1}{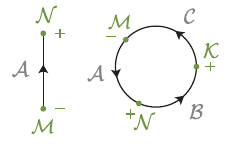}.
	\end{equation*}
	Here, the orientation of the nodes' boundaries is defined using the orientation of the underlying surface, and the signs indicate whether an adjacent defect line is oriented towards or away from the node.
	The orientation endows the set of end points of defect lines on a node with a (cyclic or linear) order, and we thus obtain a (cyclically or linearly) composable string $(\mcat_1 , \mcat_2, \dots)$ of bimodule categories which label the adjacent defect lines.
	If the sign of the end point is negative, we use the opposite bimodule category $\opcat{\mcat_i}$ instead.
	The \emph{ray category} of $L$ is then the relative Deligne product
$\raycat (L) := \mcat_1 \reldel \mcat_2 \reldel \cdots$.
\end{enumerate}
\begin{Definition}\label{def_extgraph}
	A defect surface with all three layers of labels is called an \emph{extruded graph}.
\end{Definition}

\begin{Remark}
	We briefly discuss some differences to \cite{fss-statesum}.
	Here, all manifolds are oriented as opposed to 2-framed.
	To compensate this loss of topological structure, we need additional algebraic structure:
	The tensor categories considered in	\cite{fss-statesum} are only rigid, whereas we here require them to be equipped with a spherical structure.
	We also use a spherical structure on the bimodule categories.
	Furthermore, all categories we consider are semisimple: by \cite[Appendix]{bdspv-modcats-cobreps},
	we cannot expect to construct a 3-dimensional state-sum model with defects from non-semisimple data, which is our motivation for considering extruded graphs.
	We do not discuss the question whether a similar evaluation can be defined for extruded graphs
	with non-semisimple labels.
	
	Another minor difference is that we do not admit circular defect lines: all defect lines must start and end at a node.
	Taking these points into account, our labeled defect surfaces corresponds to labeled defect
	surfaces in the terminology of \cite{fss-statesum}.
\end{Remark}

\begin{Remark}\label{rem_extgraphs_are_really_just_2d}
The terminology {\em extruded} graph needs justification.
	In Definition~\ref{def_extgraph}, they have been introduced as surfaces with decorations and algebraic labels.
	An alternative view is three-dimensional, by considering the cylinder over the surface.
	One component of the boundary of the cylinder then contains defect lines and nodes.
	From the nodes, perpendicular circular {\em rays} emanate and connect the node to the other boundary component.
	These \emph{rays} carry the ray label.
	From this perspective, nodes are no longer labeled.
	Note that a special case of ray	categories are Drinfeld centers which are braided and thus categorically inherently three-dimensional objects.
	Further discussion can be found in \cite[Remarks~3.2--3.4]{jf-thesis}.
	Mathematically, both perspectives are completely equivalent; they are simply different ways of thinking about the same data.
\end{Remark}

\subsection{The state-sum modular functor}\label{sec_fssrecap}

The evaluation of extruded graphs takes values in vector spaces provided by the state-sum modular functor described in \cite{fss-statesum}.
We briefly review an alteration of this modular functor, namely the oriented/spherical version as opposed to the framed/rigid version.
Being a symmetric monoidal functor between bicategories, the modular functor consists of several pieces of data.
One of these pieces of data is the \emph{block space} $\blockspace$, which is a vector space assigned to an extruded graph $\sigsurf$.

\begin{Remark}\label{rem_fss_not_needed}
	We do not need to explicitly construct an oriented/spherical version of the modular functor
	from \cite{fss-statesum}, since we only need block spaces.
	Moreover, we only consider fine defect surfaces, for which block spaces are easier to define.
	In fact, moving to an oriented/spherical setting {\em simplifies} the construction in \cite{fss-statesum} where the role
	of the framing was to control powers of the double dual (which in our situation is trivialized by
	the pivotal structure).
	Nevertheless, it is useful to think of the block space as a component of a larger structure,
	the modular functor, cf.~\cite[Remark~5.28]{fss-statesum}.
\end{Remark}

We now fix an extruded graph $\sigsurf$. Let $\raycat(\sigsurf)$ denote the Deligne product over the ray categories for all nodes $L$ (intervals and disks) in $\sigsurf$
\[
\raycat(\sigsurf) := \bigdeligne_L \raycat(L).
\]
Let $x \btimes y \btimes \cdots \in \raycat(\sigsurf)$ be the node labels of $\sigsurf$. We proceed to recall the definition of the associated block space $\blockspace = \blockspace(x \btimes y \btimes \cdots)$.

The choices of objects that label defect lines and nodes of $\sigsurf$ permit us to associate a vector space to each node $L$:
Let $x \in \raycat(L)$ denote the node label, and let $(m_1 \in \mcat_1, m_2 \in \mcat_2, \dots)$ be the list of objects that label the defect lines adjacent to $L$; they define an object in the \emph{node category} of $L$, which is the \emph{ordinary} Deligne product
$\nodecat (L) := \mcat_1 \deligne \mcat_2 \deligne \cdots
$
of the bimodule categories that label the defect lines adjacent to $L$.
Notice that there is a~forgetful functor~${U\colon \raycat(L) \to \nodecat(L)}$ from the ray category to the node category, which forgets the balancings that come with the relative Deligne product.
We use $U$ to define the vector space
\begin{equation}\label{eq_nodespace}
	\nodespace(m_1,m_2,\dots) := \homsetin{\nodecat(L)}{m_1 \deligne m_2 \deligne \cdots}{U(x)},
\end{equation}
to be called the \emph{node space} for the node $L$.

Varying the objects $m_1, m_2, \dots$ in the expression \eqref{eq_nodespace} gives rise to the \emph{node space functor} for a node $L$ of $\sigsurf$
\begin{equation*}
	\nodespace(L) := \homsetin{\nodecat(L)}{-}{U(x)}\colon\ \opcat{\nodecat(L)} \to \vect.
\end{equation*}
Tensoring over all nodes $L$ of $\sigsurf$ gives the (\emph{total}) node space functor associated to the surface $\sigsurf$
\begin{equation*}
	\nodespace := \bigotimes_L \nodespace(L) \colon\ \bigdeligne_L \opcat{\nodecat(L)} = \opcat{\nodecat} \to \vect.
\end{equation*}
Here, we call the Deligne product \smash{$\bigdeligne_L \opcat{\nodecat(L)} = \opcat{\nodecat}$} denotes the (total) node category for the surface~$\sigsurf$. As each defect line starts and ends at a node, the bimodule category $\mcat$ associated to a defect line appears precisely twice in the total node category, once as $\mcat$ and once as $\opcat\mcat$.
Thus the terms of the node category $\nodecat$ can be rearranged as follows
\begin{equation*}
	\bigdeligne_L \nodecat(L) = \nodecat = \bigdeligne_\mcat \mcat \deligne \opcat\mcat,
\end{equation*}
where now $\mcat$ runs over defect lines of $\sigsurf$.
Recall from \eqref{eq_def_silentobjs} that there is a silent object
\begin{equation*}
	\ewid{\mcat} = \bigoplus_\si{m} \si{m}\deligne \opcat{\si{m}} \in \mcat \deligne\opcat{\mcat}.
\end{equation*}
The Deligne product over all of these objects $\ewid{\mcat}$ provides us with a silent object in the total node category
\begin{equation*}
	\bigdeligne_\mcat \ewid{\mcat} \cong \ewid{(\deligne_\mcat \mcat)} \in \nodecat.
\end{equation*}
Following \cite{fss-statesum}, we define the \emph{pre-block space} as the value of the total node functor on this object (using the Sweedler-like notation from \eqref{eq_reldel_forget_sweed} for the forgetful functor $U$)
\begin{gather}
	\preblock := \nodespace(\ewid{(\deligne_\mcat \mcat)}) = {}_{\nodecat}\Bigl\langle\bigoplus_{\si{m},\si{n}} \si{m} \deligne \opcat{\si{m}} \deligne \si{n} \deligne\opcat{\si{n}} \deligne\cdots,\sweed{x}{\mcat} \deligne \opcat{\sweed{y}{\mcat}} \deligne \sweed{z}{\ncat} \deligne \opcat{\sweed{x}{\ncat}} \deligne\cdots\!\Bigr\rangle.\!\!\!\label{eq_def_preblock}
\end{gather}

\begin{Remark}\label{rem_preblock_shortform}
	By \eqref{eq_sil}, the pre-block space is isomorphic to the tensor product of hom-spaces of bimodule categories
	\begin{equation}\label{eq_preblock_shortform}
		\preblock \cong \homset{\sweed{y}{\mcat}}{\sweed{x}{\mcat}} \otimes \homset{\sweed{x}{\ncat}}{\sweed{z}{\ncat}} \otimes \cdots.
	\end{equation}
\end{Remark}

The block space $\blockspace$ from \cite{fss-statesum}, which we work towards defining, is obtained as a subspace of the pre-block space $\preblock$.
It consists of those vectors satisfying a condition that we will now describe. It can be thought of as a \emph{flatness condition}, see \cite[p.~5]{fss-statesum}.

We pick a domain labeled by a spherical fusion category $\acat$.
Assume that one of the defect lines adjacent to $\acat$ is labeled by a spherical bimodule category $\mcat$.
Moreover, we for now restrict to the case where the two domains adjacent to the defect line are distinct.
Without loss of generality, take $\mcat$ to be a \emph{left} $\acat$-module.
Denote the labels of the other defect lines adjacent to~$\acat$ by $\ncat$, $\kcat$, $\lcat$, \dots, and nodes by $x,y,z,\dots$ (see Figure~\ref{pic_clean_domain}).
\begin{figure}[ht]\centering
\includegraphics[scale=1]{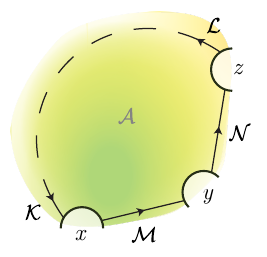}
\caption{}\label{pic_clean_domain}
\end{figure}

Given the functor $S_\mcat\colon \acat \deligne \opcat{\acat} \to \vect$ defined by
\begin{equation}\label{eq_def_sm}
	S_\mcat\big(a, \opcat{a'}\big) := \nodespace_{x,y,\dots}\big(a \ewid{\mcat} \opcat{a'} \deligne \ewid{\ncat} \deligne \ewid{\kcat} \deligne \ewid{\lcat} \deligne, \dots\big)
\end{equation}
a comparison of \eqref{eq_def_preblock} and \eqref{eq_def_sm} reveals that the pre-block space $\preblock$ is obtained from $S_\mcat$ as
\begin{equation}\label{eq_preblock_from_bibal_func}
	\preblock \cong S_\mcat \big(\one, \opcat{\one}\big).
\end{equation}
We now equip $S_\mcat$ with the structure of a bi-balanced functor, using the explicit form of the pre-block space given in \eqref{eq_def_preblock}.
A particular component of the balancing is given by
\begin{equation}\label{eq_def_holbal}
	\begin{tikzcd}[row sep=4mm]
		S_\mcat\big(a, \opcat{\one}\big) = \bigoplus_{\si{mnkl}} \homset{\si{k} \deligne \opcat{a\si{m}} \deligne \cdots}{U(x)} \otimes \homset{\si{m} \deligne \opcat{\si{n}} \deligne \cdots}{U(y)} \otimes \bigl\langle\si{n} \deligne \opcat{\si{l}} \deligne \cdots,U(z)\bigr\rangle \otimes \cdots \arrow[d, "{\text{Hom-balancing}}"]
		\arrow[gray, dashed, rounded corners, to path={	
			([xshift=11.5em,yshift=.8em]\tikztostart.south west)
			-- ([xshift=11.5em,yshift=.4em]\tikztostart.south west)
			-- ([xshift=12.3em,yshift=0.1em]\tikztotarget.north west)
			-- ([xshift=12.3em,yshift=-0.5em]\tikztotarget.north west)}]{d} \\
		\bigoplus_{\si{mnkl}} \homset{\si{k} \deligne \opcat{\si{m}} \deligne \cdots}{U(x) a} \otimes \homset{\si{m} \deligne \opcat{\si{n}} \deligne \cdots}{U(y)} \otimes \bigl\langle\si{n} \deligne \opcat{\si{l}} \deligne \cdots,U(z)\bigr\rangle \otimes \cdots \arrow[d, "{\text{balancing of $x$}}"]
		\arrow[gray, dashed, rounded corners, to path={	
			([xshift=12.3em,yshift=.8em]\tikztostart.south west)
			-- ([xshift=12.3em,yshift=.4em]\tikztostart.south west)
			-- ([xshift=10.1em,yshift=0.1em]\tikztotarget.north west)
			-- ([xshift=10.1em,yshift=-0.5em]\tikztotarget.north west)}]{d} \\
		\bigoplus_{\si{mnkl}} \homset{\si{k} \deligne \opcat{\si{m}} \deligne \cdots}{a U(x)} \otimes \homset{\si{m} \deligne \opcat{\si{n}} \deligne \cdots}{U(y)} \otimes \bigl\langle\si{n} \deligne \opcat{\si{l}} \deligne \cdots,U(z)\bigr\rangle \otimes \cdots \arrow[d, "{\text{Hom-balancing}}"]
		\arrow[gray, dashed, rounded corners, to path={	
			([xshift=10.1em,yshift=.8em]\tikztostart.south west)
			-- ([xshift=10.1em,yshift=.4em]\tikztostart.south west)
			-- ([xshift=4.4em,yshift=0.1em]\tikztotarget.north west)
			-- ([xshift=4.4em,yshift=-0.5em]\tikztotarget.north west)}]{d} \\
		\bigoplus_{\si{mnkl}} \homset{a^* \si{k} \deligne \opcat{\si{m}} \deligne \cdots}{U(x)} \otimes \homset{\si{m} \deligne \opcat{\si{n}} \deligne \cdots}{U(y)} \otimes \bigl\langle\si{n} \deligne \opcat{\si{l}} \deligne \cdots,U(z)\bigr\rangle \otimes \cdots \arrow[d, dotted, "\text{several balancings}"{left}]
		\arrow[gray, dashed, rounded corners, no head, to path={	
			([xshift=4.4em,yshift=.8em]\tikztostart.south west)
			-- ([xshift=4.4em,yshift=.4em]\tikztostart.south west)
			-- ([xshift=9em,yshift=-0.1em]\tikztostart.south west)}]
		\arrow[gray, dashed, rounded corners, to path={	
			([xshift=19em,yshift=-0.4em]\tikztostart.south west)
			-- ([xshift=26.2em,yshift=0.1em]\tikztotarget.north west)
			-- ([xshift=26.2em,yshift=-0.5em]\tikztotarget.north west)}]{d}
		\\
		\bigoplus_{\si{mnkl}} \homset{\si{k} \deligne \opcat{\si{m}} \deligne \cdots}{U(x)} \otimes \homset{\si{m} \deligne \opcat{\si{n}} \deligne \cdots}{U(y)} \otimes \bigl\langle\si{n} \deligne \opcat{a\si{l}} \deligne \cdots,U(z)\bigr\rangle \otimes \cdots \arrow[d, dotted, "\text{several balancings}"{left}]
		\arrow[gray, dashed, start anchor={[xshift=26.2em,yshift=.7em]south west}, end anchor = {[xshift=31.9em,yshift=.4em]south west},bend right=40] 
		\arrow[gray, dashed, start anchor={[xshift=31.8em,yshift=.7em]south west}, end anchor = {[xshift=29.9em,yshift=.5em]south west},bend left=30]
		\arrow[gray, dashed, rounded corners, crossing over, to path={	
			([xshift=29.6em,yshift=.8em]\tikztostart.south west)
			-- ([xshift=29.6em,yshift=.4em]\tikztostart.south west)
			-- ([xshift=24.7em,yshift=0.1em]\tikztotarget.north west)
			-- ([xshift=24.7em,yshift=-0.5em]\tikztotarget.north west)}]{d}\\
		\bigoplus_{\si{mnkl}} \homset{\si{k} \deligne \opcat{\si{m}} \deligne \cdots}{U(x)} \otimes \homset{\si{m} \deligne \opcat{\si{n}} \deligne \cdots}{U(y)} \otimes \homset{a^*\si{n} \deligne \opcat{\si{l}} \deligne \cdots}{U(z)} \otimes \cdots
		\arrow[dashed, gray, start anchor={[xshift=24.7em,yshift=.7em]south west}, end anchor = {[xshift=16.2em,yshift=.4em]south west},bend left=40, looseness=0.8]
		\arrow[gray, dashed, rounded corners, crossing over, to path={	
			([xshift=22.0em,yshift=.8em]\tikztostart.south west)
			-- ([xshift=22.0em,yshift=.4em]\tikztostart.south west)
			-- ([xshift=20.2em,yshift=0.1em]\tikztotarget.north west)
			-- ([xshift=20.2em,yshift=-0.5em]\tikztotarget.north west)}]{d}
		\arrow[gray, dashed, start anchor={[xshift=16.4em,yshift=.8em]south west}, end anchor = {[xshift=21.7em,yshift=.5em]south west},bend right=30, looseness=0.6]
		\arrow[d, dotted, crossing over, "\text{several balancings}\qquad"{left}]
		\\
		\bigoplus_{\si{mnkl}} \homset{\si{k} \deligne \opcat{\si{m}} \deligne \cdots}{U(x)} \otimes \homset{\si{m} \deligne \opcat{\si{n}} \deligne \cdots}{aU(y)} \otimes \bigl\langle\si{n} \deligne \opcat{\si{l}} \deligne \cdots,U(z)\bigr\rangle \otimes \cdots
		\arrow[gray, dashed, rounded corners, to path={	
			([xshift=20.2em,yshift=.8em]\tikztostart.south west)
			-- ([xshift=20.2em,yshift=.4em]\tikztostart.south west)
			-- ([xshift=14.5em,yshift=0.1em]\tikztotarget.north west)
			-- ([xshift=14.5em,yshift=-0.5em]\tikztotarget.north west)}]{d}
		\arrow[d, "{\text{Hom-balancing}}", crossing over]\\
		\bigoplus_{\si{mnkl}} \homset{\si{k} \deligne \opcat{\si{m}} \deligne \cdots}{U(x)} \otimes \homset{a^* \si{m} \deligne \opcat{\si{n}} \deligne \cdots}{U(y)} \otimes \bigl\langle\si{n} \deligne \opcat{\si{l}} \deligne \cdots,U(z)\bigr\rangle \otimes \cdots = S_\mcat(\one,\opcat{a}).
	\end{tikzcd}\hspace{-10mm}
\end{equation}
The left balancing for $S_\mcat$ defined in \eqref{eq_def_holbal} is called ``holonomy'' in \cite[Section~4.11]{fss-statesum}.
It can be pictured as ``walking around the boundary'' of the domain $\acat$ with an object $a \in \acat$
(see Figure~\ref{figure8}).
\begin{figure}[ht]\centering
\includegraphics[scale=1]{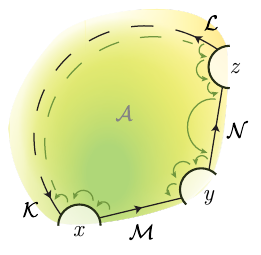}
\caption{}\label{figure8}
\end{figure}

The right balancing of $S_\mcat$ is directly inherited from the balancing of the silent object $\ewid{\mcat}$.
Together, this turns $S_\mcat$ into a bi-balanced functor.

Balanced functors $S_\mcat$, $S_\ncat$, and similar can be constructed for each defect line adjacent to $\acat$, as long as those defect lines separate $\acat$ from a different domain, or lie on the boundary of the defect surface.
It may be the case, however, that a defect line has the domain $\acat$ as its left- and right-adjacent domain, as in Figure~\ref{pic_internal_defline}.
\begin{figure}[ht]\centering
\includegraphics[scale=1]{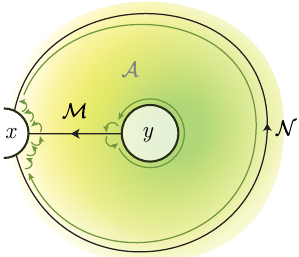}
\caption{}\label{pic_internal_defline}
\end{figure}

Then $\mcat$ is an $\acat$-$\acat$-bimodule category.
In this case, $S_\mcat$ is a functor with two balancings as well, but in the ``disconnected'' sense as in \eqref{eq_bibal_ambiguous_disconn}.

We use the equalizer $\eq {S_\mcat}$ of $S_\mcat$ defined in Section~\ref{sec_spleqs} to define the block spaces.
For this to be consistent, we need the result \cite[Lemma~4.22]{fss-statesum}, which in our language assumes the following form.
\begin{Proposition}\label{prop_holonomy_startpoint_indep}
	There are canonical isomorphisms
$\baleqat{S_\mcat}{\one}{\one} \cong \baleqat{S_\ncat}{\one}{\one}
$
	between equalizers for each pair of defect lines $(\mcat,\ncat)$ adjacent to $\acat$.
\end{Proposition}
In order to define block spaces, we need to consider all domains simultaneously.
For each domain $\acat, \bcat, \ccat, \dots$, choose a defect line $\mcat_\acat$ adjacent to $\acat$, $\mcat_\bcat$ adjacent to~$\bcat$, $\mcat_\ccat$ adjacent to~$\ccat$, and so on.

\begin{Definition}\label{def_blockspace_fine}
	For an extruded graph $\sigsurf$, the associated \emph{block space} is the vector space $\blockspace$ given by the intersection
	\begin{equation*}
		\blockspace := \bigcap_{\text{domains }\acat} \baleqat{S_{\mcat_\acat}}{\one}{\one}
	\end{equation*}
	of the equalizers $\baleqat{S_{\mcat_\acat}}{\one}{\one}$, where for each domain $\acat$, an adjacent defect line $\mcat_\acat$ has been chosen.
	Proposition~\ref{prop_holonomy_startpoint_indep} ensures that this is well defined up to a canonical isomorphism.
\end{Definition}
When unraveling Definition~\ref{def_blockspace_fine}, one can see that it agrees with the original definition of the block functor \cite[Definition 4.21]{fss-statesum}.
Note that the block space $\blockspace$ is by definition a sub-vector space of the pre-block space $\preblock$.

\begin{Example}\label{ex_equat_punctured_sphere}
	This example is a special case of \cite[Example 4.35]{fss-statesum}.
	Let $\acat$ and $\bcat$ be spherical fusion categories, and let $\mcat$ be an $\acat$-$\bcat$-bimodule category and $\ncat$ be a $\bcat$-$\acat$-bimodule category.
	Consider as a labeled defect surface the sphere with two nodes and two defect lines connecting the nodes, such that each node has an incoming and an outgoing defect line attached.
	We label the domains by $\acat$ and $\bcat$, and the defect lines by $\mcat$ and $\ncat$, accordingly.
	We pick node labels~${x \in \opcat{\mcat} \reldel \ncat \reldel}$ and $y \in \opcat{\ncat} \reldel \mcat \reldel$ in the ray categories associated to the nodes.
	(The ``dangling product'' notation has been introduced in \eqref{eq_dangling_product}.)
	The pre-block space for this defect surface is the space of natural transformations
	\begin{equation*}
		\preblock \cong \homsetin{\opcat{\ncat} \deligne \mcat}{U(\opcat{x})}{U(y)} \cong \homsetin{\funcat{\mcat}{\ncat}}{\ew(U\opcat{x})}{\ew(U(y))} = \mathrm{Nat}(\ew(U(\opcat{x})),\ew(U(y)).
	\end{equation*}
	Furthermore, the block space is given by the subspace of bimodule natural transformations
	\begin{equation*}
		\blockspace \cong \homsetin{\opcat{\ncat} \reldel \mcat \reldel}{\opcat{x}}{y} \cong \homsetin{\bimodfun{\acat}{\bcat}{\mcat}{\ncat}}{\mew(\opcat{x})}{\mew(y)} = \mathrm{Nat}(\mew(U(\opcat{x})),\mew(U(y))).
	\end{equation*}
\end{Example}

\begin{Remark}\label{rem_refinement_blockspace}
	The block space of a framed defect surface that is not fine has been expressed~\cite[Definition~5.24]{fss-statesum}, in the case of
	the framed modular functor, as a limit over all framed fine refinements.
	A fine refinement of a non-fine framed defect surface $\sigsurf$ is a fine defect surface, whose defect structure (nodes and defect lines) is a subdivision of the framed defect structure of~$\sigsurf$.
	A similar construction can be expected to exist for the oriented state-sum modular functor on
	surfaces that are not fine, but has not been worked out to our knowledge, and is not needed for our purposes.
\end{Remark}

\subsection{Definition of the evaluation} \label{sec_eval}
To each extruded graph $\sigsurf$ we associate a linear map $\evaluate{\sigsurf}\colon \nodespace \to \blockspace$ from its node space to its block space, called its \emph{evaluation}.
This section is dedicated to defining this evaluation, which factors through the pre-block space $\preblock$ and is thus a composition of two linear maps \smash{$\nodespace \xrightarrow{\delta} \preblock \xrightarrow{\holsurj} \blockspace$}.
We start by explaining the second map $\holsurj$.

In Definition~\ref{def_blockspace_fine} of the block spaces, the equalizers of bi-balanced functors $\baleqat{S_{\mcat_\acat}}{\one}{\one}$ appear, which by Lemma~\ref{lem_speqs_of_bibals} come with a retract structure.
In particular, there is an idempotent $\holidem_\acat$ on~\smash{${S_{\mcat_\acat}}({\one},{\one}) \eqwithref{eq_preblock_from_bibal_func} \preblock$} whose image is the sub-vector space $\baleqat{S_{\mcat_\acat}}{\one}{\one}$.
The idempotents $\holidem_\acat$ for different domains $\acat$ commute, allowing us to define:
\begin{Definition}\label{def_holidem}
	We call the idempotent $\holidem_\acat$ for a domain $\acat$ its \emph{holonomy idempotent}.
	The \emph{holonomy idempotent} of an extruded graph $\sigsurf$ is given by the composition
$\holidem := \holidem_\acat \circ \holidem_\bcat \circ \cdots
$
	of holonomy idempotents for each domain of $\sigsurf$.
\end{Definition}
It follows immediately from Definition~\ref{def_blockspace_fine} that the image of $\holidem$ is the block space
${\mathrm{im}(\holidem) \!=\! \blockspace \!\subset\! \preblock}$.
The corestriction of the holonomy idempotent $\holidem$ onto its image is the surjection $\holsurj$, which was announced in the beginning of the section
\begin{equation}\label{eq_def_holproj}
	\begin{tikzcd}
		\preblock \arrow[r, "\holidem"] \arrow[dr, "\holsurj", two heads] & \preblock \\
		& \blockspace \arrow[u, hook].
	\end{tikzcd}
\end{equation}

Next, we need to define the map $\delta\colon \nodespace \to \preblock$.
It is obtained from the morphisms
\begin{equation*}
	\tau_m \colon\ \ewid{\mcat} = \bigoplus_\si{n} \si{n} \deligne \opcat{\si{n}} \to m \deligne\opcat{m},
\end{equation*}
for $m \in \mcat$, defined by
\begin{equation*}
	\tau_m := \bigoplus_\si{n} \pdim{\si{n}}  (\basisel{m}{\si{n}} \deligne \basisel{\si{n}}{m}).
\end{equation*}
The inclined reader will notice that the $\tau_m$ are the structure morphisms of the end of the functor~${-\deligne-}$.
We then define $\delta$ as the linear map
\begin{equation}\label{eq_eval_delta}
	\nodespace = \nodespace(m_1 \deligne \opcat{m_1} \deligne m_2 \deligne \opcat{m_2} \deligne \cdots) \xrightarrow{\nodespace(\tau_{m_1 \deligne m_2 \deligne \cdots})} \nodespace(\ewid{\mcat_1 \deligne \mcat_2 \deligne \cdots}) = \preblock.
\end{equation}

\begin{Definition}[evaluation] \label{def_eval}
	Let $\sigsurf$ be an extruded graph.
	The \emph{evaluation} $\evaluate{\sigsurf} \colon \nodespace \to \blockspace$ of $\sigsurf$ is the linear map
	\begin{equation}\label{eq_def_eval_shortform}
		\evaluate{\sigsurf} := \holsurj \circ \delta.
	\end{equation}
\end{Definition}

\begin{Remark}\label{rem_evaluate_on_core}
	We often would like to obtain a scalar from an extruded graph $\sigsurf$.
	This is only possible with additional data: a vector $f \in \nodespace$ in the node space and a linear form $\varphi \in \blockspace^*$ on the block space.
	We say that
	\begin{equation}\label{eq_total_eval}
		\totaleval{\varphi}{\sigsurf}{f} = \varphi ( \evaluate{\sigsurf} f ) \in \field
	\end{equation}
	is the \emph{scalar evaluation} of the triple $(\sigsurf, f, \varphi)$.
	The left-hand side of \eqref{eq_total_eval} uses the pairing~${(-,-)\colon \blockspace^* \otimes \blockspace \to \field}$.
\end{Remark}

\begin{Remark}\label{rem_modtraces_not_necessary}
	Our construction supposes that all (bi-)module categories are equipped with a~trace.
	However, the trace structure is never used in the definition of the evaluation procedure.
	We will see in Theorem~\ref{thm_uniqueness} that traces provide distinguished choices for linear forms ${\varphi \in \blockspace^*}$.\looseness=1
\end{Remark}

\begin{Example}[generalized $6j$ symbols]\label{ex_six_j}
	A homeomorphism between the surface of a tetrahedron and the sphere $\mathbb{S}^2$ defines a graph with 4 vertices and 6 edges on the sphere.
	Let us consider an extruded graph $\sigsurf$ on $\mathbb{S}^2$, whose nodes and defect lines are obtained from the vertices and edges of the tetrahedron (after picking an orientation for each edge).
	The domains of $\sigsurf$ are labeled by arbitrary spherical fusion categories, and the defect lines by arbitrary traced bimodule categories.
	The object labels for the defect lines, as well as the node labels, are all simple objects in the respective categories.
	The evaluation of such an extruded graph, composed with some~${\varphi \in \blockspace^*}$ defines a linear form $\varphi \circ \evaluate{\sigsurf} \in \nodespace^*$ on the total node space of $\sigsurf$.
	This linear form further generalizes the \emph{generalized $6j$-symbol} from \cite{meusburger-statesumdefects}.
\end{Example}

\begin{Remark}\label{rem_refinement_evaluation}
	A state-sum oriented modular functor defined on non-fine defect surfaces as mentioned in Remark~\ref{rem_refinement_blockspace} would allow to evaluate also
	non-fine extruded graphs. Such a~construction is beyond the scope of this work.
\end{Remark}

\subsection{Loop graphs and lasso graphs} \label{sec_lassographs}
We now explicitly compute the invariant for two of the simplest possible types of extruded graphs: First, a \emph{lasso graph} $Q$, whose underlying surface is the sphere $\mathbb{S}^2$. We will later see (see Theorem~\ref{thm_uniqueness}) that any extruded graph on the sphere can be transformed -- using the moves from Section~\ref{sec_moves} -- into this lasso graph.
A lasso graph is an extruded graph $Q$ (see Figure~\ref{pic_lassograph}): It is a~sphere with a~single node and a~single circular defect line starting and ending at the node.\looseness=1

\begin{figure}[ht]\centering
\includegraphics[scale=1]{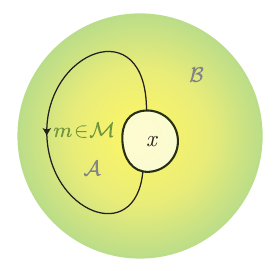}
\caption{}\label{pic_lassograph}
\end{figure}


\pagebreak

There are three levels of algebraic labels to be specified:
\begin{itemize}\itemsep=0pt
	\item The two hemispheres forming the connected components of the complement of the defect line are labeled by spherical fusion categories $\acat$ and $\bcat$.
	\item The single defect line is labeled by a traced bimodule category $\bimod{\acat}{\mcat}{\bcat}$.
	This implies that the ray category associated to the single node is the balanced category
\[
\raycat = \mcat \reldelov{\bcat} \opcat{\mcat} \reldelov{\acat}.
\]
It will be important that this category contains a silent object $\zewid{\mcat}$, as described in Section~\ref{sec_silentobs}.
	In contrast, the node category is given by the (unbalanced) category
\[
\nodecat = \mcat \deligne \opcat{\mcat}.
\]
	\item The defect line carries an additional subordinate label: A choice of an object $m \in \mcat$.
	Moreover, we label the node by an object $x \in \raycat$.
	These choices determine the node space to be the Hom-space
\[
\nodespace(m\btimes \opcat{m}) = \homsetin{\nodecat}{m\btimes \opcat{m}}{U(x)}.
\]
\end{itemize}
This completes the three levels of labels that determine the extruded graph $Q$.

The extruded graphs of the second type, \emph{loop graphs}, form a subset of lasso graphs: they are characterized the fact that the node label $x ={} \zewid{\mcat}$ is silent.
Loop graphs are special because their evaluation against a particular ``transparent'' linear form $\theta \in \blockspace^*$ is a trace.
This is the main result of this section.

\begin{Theorem}\label{thm_elementary_graph}
	\label{sec_loopgraphs}
	Let $O$ be a loop graph, that is, an extruded graph as in Figure~{\rm \ref{pic_lassograph}} whose underlying surface is a sphere with a single node and a single defect line, labeled by some object $m \in \bimod{\acat}{\mcat}{\bcat}$ starting and ending at the node, whose node is labeled by the silent object $\zewid{\mcat}$.
	Then the block space $T$ of $O$ is the endomorphism space \smash{$T= \homsetin{\mcat \reldel \opcat{\mcat} \reldel}{\zewid{\mcat}}{\zewid{\mcat}}$}.
	Moreover, the linear form
	\begin{equation}\label{eq_dualoftheta}
		\theta := \catdim{\acat} \catdim{\bcat}  \trace{\mcat \reldel \opcat{\mcat} \reldel}(-) \in \blockspace^*,
	\end{equation}
	relates the evaluation $\evaluate{O}$ from Definition~{\rm\ref{def_eval}} to the module trace of the bimodule category $\mcat$
	\begin{equation}\label{eq_elementary_graph}
		\theta \circ \evaluate{O} = \trace{\mcat}\circ \soc^{-1} \in \nodespace^*.
	\end{equation}
	The isomorphism $\soc \colon\homset{m}{m} \to \homset{\ewid{\mcat}}{m \deligne \opcat{m}} = \nodespace$ appearing in \eqref{eq_elementary_graph} was introduced in Section~{\rm\ref{sec_genyon}}.
\end{Theorem}
\begin{Remark}
	If the module category $\mcat$ is indecomposable, the block space $\homset{\zewid{\mcat}}{\zewid{\mcat}}$ is one-dimensional.
\end{Remark}
Theorem~\ref{thm_elementary_graph} can be reformulated in the following useful fashion.
\begin{Corollary}\label{cor_loop_eval_scalar}
	Let $O$ be a loop graph as in Theorem~{\rm\ref{thm_elementary_graph}}, and let $f\colon m \to m$ be an endomorphism of $m \in \mcat$.
	Then the scalar evaluation of $(O, \soc(f), \theta)$, see \eqref{eq_total_eval}, is the scalar
	\begin{equation*}
		(\theta, \evaluate{O} \soc(f) ) = \trace{\mcat}(f) \in \field.
	\end{equation*}
\end{Corollary}

The proof of Theorem~\ref{thm_elementary_graph} relies on an intermediate result for general lasso graphs.
To simplify the presentation, we specialize in the following to $\bcat = \vect$.
In essence, $\mcat$ is now just a~left $\acat$-module category.
On the topological level, $\vect$-labeled domains can be removed entirely, so we are left with a disk with a single boundary node as our defect surface.
The more general case where $\bcat$ is not necessarily $\vect$ does not provide additional insight.

The node category $\nodecat = \mcat \deligne \opcat{\mcat}$ is an $\acat$-$\acat$-bimodule category, whose center $Z_\acat\bigl(\nodecat\bigr) = \raycat$ is the ray category.
Because both the silent object $\zewid{\mcat}$ and the node label $x$ are objects of $\raycat$, the pre-block space
$
\preblock = \homsetin{\nodecat}{\ewid{\mcat}}{U(x)}
$
has a subspace $\homsetin{\raycat}{\zewid{\mcat}}{x}$;
in Lemma~\ref{lemma_homspace_projection}, we saw that there is a distinguished retraction $r\colon \homsetin{\nodecat}{\ewid{\mcat}}{U(x)} \to \homsetin{\raycat}{\zewid{\mcat}}{x}$ to this inclusion of hom-spaces.
\begin{Lemma}\label{lem_elemgraph_generalformula}
	Let $Q$ be a lasso graph with node label $x \in \raycat$.
	The block space of $Q$ is the hom-space in the ray category
	\begin{equation}\label{eq_lasso_block}
		\blockspace = \homsetin{\raycat}{\zewid{\mcat}}{x}.
	\end{equation}
	Moreover, the projection $\holsurj \colon \preblock \to \blockspace$ from \eqref{eq_def_holproj} is equal to the retraction into the hom-space of the center from Lemma~{\rm\ref{lemma_homspace_projection}}.
\end{Lemma}
\begin{proof}
	Both claims of the lemma follow immediately if we can show that the holonomy idempotent $h : \preblock \to \preblock$ as defined in Definition~\ref{def_holidem} is equal to the idempotent from \eqref{eq_homspace_projection}.
	Concretely, we must prove
	\begin{equation}\label{eq_loop_holidem_result}
		h = \sum_\si{a} \frac{\pdim{\si{a}}}{\catdim{\acat}} \brev{x}{\si{a}^*} \circ (\si{a}^*  -  \si{a}) \circ \cobrev{\zewid{\mcat}}{\si{a}} = \sum_\si{a} \frac{\pdim{\si{a}}}{\catdim{\acat}} \pica{1}{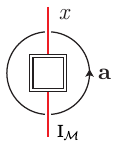}.
	\end{equation}

	Starting with an element $g \in \preblock = \homsetin{\nodecat}{\ewid{\mcat}}{U(x)}$, we unravel Definition~\ref{def_holidem} and calculate its image $\holidem(g)$ under the holonomy idempotent $\holidem$.
	Since $\holidem$ is a composition of several maps, the calculation is broken down into intermediate steps.
	The morphisms composing to $\holidem$ are given vertically in the left column of Figure~\ref{calc_lasso_longcalc}. The right column contains the images of $g$ under the partial compositions of the morphisms in graphical notation. In particular, the bottom-most entry is equal to $\holidem(g)$ (see Figure~\ref{calc_lasso_longcalc}). From the bottom-most row, we can read off the expression for $h$, and recognize that we have obtained \eqref{eq_loop_holidem_result}.
\end{proof}

	\begin{figure}[t]\centering

	\begin{tikzcd}[row sep=4mm, column sep=20mm]
	\homset{\ewid{\mcat}}{U(x)} \arrow[d, "{\frac{1}{\catdim{\acat}}}  \text{evaluation}"]&
	g \arrow[d, mapsto] \\
	\bigoplus_\si{a} \homset{\si{a} \si{a}^* \ewid{\mcat}}{U(x)} \arrow[d, "{\text{balancing of } \zewid{\mcat}}"]&
	\bigoplus_\si{a} \frac{\pdim{\si{a}}}{\catdim{\acat}} \pica{0.8}{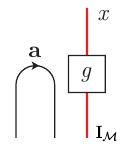} \arrow[d, mapsto]
	\\
	\bigoplus_\si{a} \homset{ \si{a} \ewid{\mcat} \si{a}^*}{U(x)} \arrow[d, "{\text{Hom-balancing}}"]&
	\bigoplus_\si{a} \frac{\pdim{\si{a}}}{\catdim{\acat}} \pica{0.8}{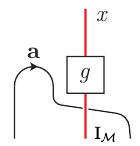} \arrow[d, mapsto]
	 \\
	\bigoplus_\si{a} \homset{\si{a} \ewid{\mcat}}{U(x)\si{a}} \arrow[d, "{\text{balancing of }x}"]&
	\bigoplus_\si{a} \frac{\pdim{\si{a}}}{\catdim{\acat}} \pica{0.8}{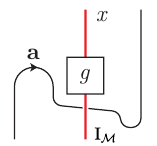} \arrow[d, mapsto]
	 \\
	\bigoplus_\si{a} \homset{\si{a} \ewid{\mcat}}{\si{a}U(x)} \arrow[d, "{\text{Hom-balancing}}"]&
	\bigoplus_\si{a} \frac{\pdim{\si{a}}}{\catdim{\acat}} \pica{0.8}{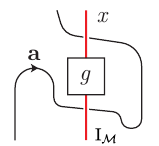} \arrow[d, mapsto]
	 \\
	\bigoplus_\si{a} \homset{\si{a}^* \si{a} \ewid{\mcat}}{U(x)} \arrow[d, "{\text{coevaluation}}"]&
	\bigoplus_\si{a} \frac{\pdim{\si{a}}}{\catdim{\acat}} \pica{0.8}{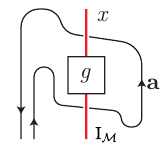} \arrow[d, mapsto]
	 \\
	\homset{\ewid{\mcat}}{U(x)} &
	\sum_\si{a} \frac{\pdim{\si{a}}}{\catdim{\acat}} \pica{0.8}{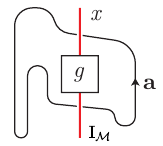}.
	\end{tikzcd}
\caption{}\label{calc_lasso_longcalc}
	\end{figure}

The expression \eqref{eq_lasso_block} for the block space of a lasso graph is in agreement with \cite[Example~4.33]{fss-statesum}.
\begin{Remark}
	The contents of Lemma~\ref{lem_elemgraph_generalformula} generalize to the case where $\bcat$ is not necessarily $\vect$.
	In particular, the general formula for the holonomy idempotent $\holidem$ is given by
	\begin{equation*}
		\holidem(g) = \sum_{\si{a},\si{b}} \frac{\pdim{\si{a}}\pdim{\si{b}}}{\catdim{\acat}\catdim{\bcat}} \brev{x}{\si{a}^* \deligne \si{b}^*} \circ ((\si{a}^* \deligne \si{b}^*) g (\si{a} \deligne \si{b})) \circ \cobrev{\zewid{\mcat}}{\si{a}\deligne \si{b}}.
	\end{equation*}
\end{Remark}

Lemma~\ref{lem_elemgraph_generalformula} is a step towards proving Theorem~\ref{thm_elementary_graph}.
The remainder of the proof is technical, and can be found in Appendix~\ref{app_loopgraphs}.
Here we continue with an application of Corollary~\ref{cor_loop_eval_scalar}, which is a formula for the evaluation of lasso graphs.

\begin{Proposition}\label{move_SV}
	Let $Q$ be a lasso graph with labels as in Figure~{\rm \ref{pic_lassograph}}.
Moreover, let
	\begin{itemize}\itemsep=0pt
		\item $f \in \nodespace = \homset{m \deligne\opcat{m}}{U(x)}$ be a node label,
		\item $\tilde{\varphi} \in \homsetin{\raycat}{x}{\zewid{\mcat}}$ be a morphism, and
		\item $\varphi \in \blockspace^*$ be the corresponding linear form on the block space defined by $\varphi := \trace{\raycat}(\tilde{\varphi} \circ -) \in \blockspace^* = \homset{\zewid{\mcat}}{x}^*$.
	\end{itemize}
	Then
$(\varphi, \evaluate{Q} f ) = \trace{\mcat}\bigl(\soc^{-1}(\tilde{\varphi} \circ f)\bigr) \in \field$.
\end{Proposition}
\begin{proof}
	The proof is a direct calculation. In the second-to-last step, the loop graph $O$ from Theorem~\ref{thm_elementary_graph} appears,
	\begin{align*}
			 (\varphi, \evaluate{Q} f ) &= (\varphi \circ \holsurj \circ \delta)(f) = \trace{\raycat} (\tilde{\varphi} \circ \holsurj (\delta (f)) ) \eqwithref{eq_retraction_sesquimult} \trace{\raycat} (\holsurj (\tilde{\varphi} \circ \delta (f)) ) \\
&\eqwithref{eq_eval_delta} \trace{\raycat} (\holsurj (\tilde{\varphi} \circ f \circ \tau_m) ) = \trace{\raycat} (\holsurj (\delta (\tilde{\varphi} \circ f)) ) \eqwithref{eq_dualoftheta} \frac{1}{\catdim{\acat}\catdim{\bcat}}\theta \circ \holsurj \circ \delta (\tilde{\varphi} \circ f) \\
&= \frac{1}{\catdim{\acat}\catdim{\bcat}} (\theta,\evaluate{O} (\tilde{\varphi} \circ f) ) \hspace{7mm}\eqdesc{\text{Corollary~\ref{cor_loop_eval_scalar}}}\hspace{7mm} \trace{\mcat} \bigl(\soc^{-1}(\tilde{\varphi} \circ f) \bigr).\tag*{\qed}
	\end{align*}\renewcommand{\qed}{}
\end{proof}

\section{Moves of invariance}\label{sec_moves}
Moves are certain transformations of an extruded graph, changing both its topological and algebraic data.
They make the abstractly defined evaluation procedure algorithmically computable.

\subsection{Overview of moves} \label{sec_moves_overview}
A move $\mathbf{M}$ can be performed on a disk-shaped subset $\sigsurf$ of an extruded graph, and results in another extruded graph, which is obtained by replacing the region $\sigsurf$ with another disc region $\postmove{\sigsurf}$.
We require that the boundaries of the discs intersect
the respective extruded graphs in a generic position, pictured, e.g., in Figures~\ref{pic_moveOR} and~\ref{pic_moveC}.
Each move $\mathbf{M}$ can only be performed on disk regions $\sigsurf$ which have a geometry specific to $\mathbf{M}$.
The data of a move are then given by:
\begin{itemize}\itemsep=0pt
	\item An assignment $\sigsurf \mapsto \postmove{\sigsurf}$ describing the local change in geometry.
	\item A prescription to determine the algebraic labels of $\postmove{\sigsurf}$ from the algebraic labels of $\sigsurf$.
	\item A linear map $\movemap \colon \nodespace \to \postmove{\nodespace}$ from the node space $\nodespace$ of $\sigsurf$ to the node space $\postmove{\nodespace}$ of $\postmove{\sigsurf}$.
\end{itemize}
While the regions $\sigsurf$ and $\postmove{\sigsurf}$ have well-defined node spaces in the sense of \eqref{eq_nodespace}, they do not have block- or pre-block spaces.
When in the following, we make a statement about ``the pre-block space of $\sigsurf$'', or similar, we really mean: Consider an extruded graph $\sigsurf_T$ which contains the region $\sigsurf$. Then the statement holds for the pre-block space of $\sigsurf_T$, for any choice of $\sigsurf_T$.

We proceed to define two more linear maps
\begin{gather}
	\Pmovemap \colon\ \preblock \to \postmove{\preblock}, \qquad \Pmovemap := \bigoplus_\si{m} \postmove{\delta}_\si{m} \circ \movemap_\si{m}, \quad\text{ and} \label{eq_Pmovemap} \\
	 \Movemap \colon\ \blockspace \to \postmove{\blockspace}, \qquad \Movemap := \holsurj \circ \Pmovemap \circ \iota \label{eq_Movemap}
\end{gather}
between pre-block spaces and block spaces.
The direct sum in \eqref{eq_Pmovemap} runs over simple object labels $\si{m}$ for the defect lines of $\sigsurf$.
For each such labeling, the move $\mathbf{M}$ provides a linear map between the appropriate node spaces, which is called $\movemap_\si{m} \colon \nodespace_\si{m} \to \postmove{\nodespace}_\si{m}$ in \eqref{eq_Pmovemap}.
The linear map~${\postmove{\delta_\si{m}} \colon \postmove{\nodespace}_\si{m} \to \postmove{\preblock}}$ is the one that appears in the definition of the evaluation \eqref{eq_def_eval_shortform}.
Equation~\eqref{eq_Movemap} involves the morphisms of the retraction
\[\begin{tikzcd}
	\blockspace \arrow[r, "\iota"{below}] & \preblock. \arrow[l, "\holsurj"{above}, bend right=20]
\end{tikzcd}
\]
The map $\Movemap$ is needed for the following definition.

\begin{Definition}\label{def_invariance}
	We call a move $\mathbf{M}$ between regions $\sigsurf$ and $\postmove{\sigsurf}$ with linear map $\movemap$ on the node spaces an \emph{invariance} or a \emph{move of invariance} iff the diagram
	\begin{equation}\label{diag_movediag}
		\begin{tikzcd}
			\nodespace \arrow[r, "\evaluate{\sigsurf_{T}}"{above}] \arrow[d, "\movemap"{left}] & \blockspace \arrow[d, "\Movemap"] \\
			\postmove{\nodespace} \arrow[r, "\evaluate{\postmove{\sigsurf_{T}}}"{below}] & \postmove{\blockspace}
		\end{tikzcd}
	\end{equation}
	commutes for every extruded graph $\sigsurf_{T}$ which contains the region $\sigsurf$. Here, $\postmove{\sigsurf_{T}}$ denotes the extruded graph obtained from $\sigsurf_{T}$ by replacing the region $\sigsurf$ with $\postmove{\sigsurf}$.
\end{Definition}

We now define a set of moves, wherein we distinguish between \emph{elementary moves} and \emph{composite moves}.
For the elementary moves, we show the invariance property using the definition of the evaluation procedure $\evaluate{-}$. The composite moves are combinations of elementary moves.

\begin{Remark}
	For all the moves in our list, the linear map $\Movemap$ is an isomorphism.
\end{Remark}

We consider the following elementary moves.

\subsubsection{Elementary moves}
{\bf $\moveOR$ -- orientation reversal.} 
One of the most basic changes one can make to an extruded graph is the flipping of a defect line's orientation (see Figure~\ref{pic_moveOR}).
In the $\moveOR$-move, a defect line labeled by an object $m$ in an $\acat$-$\bcat$-bimodule category $\mcat$ is replaced by a defect line with opposite orientation.
This new defect line needs a $\bcat$-$\acat$-bimodule category for a label.
For this we choose the opposite category $\opcat{\mcat}$, and as the post-move object label we take $\opcat{m} \in \opcat{\mcat}$.
Note that $\opcat{\mcat}$ is unambiguously defined only because $\acat$ and $\bcat$ are equipped with pivotal structures.

\begin{figure}[ht]\centering
	\includegraphics{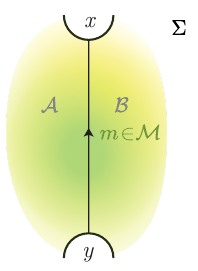} \ \raisebox{20mm}{$\xrightarrow{\moveOR}$} \ \includegraphics{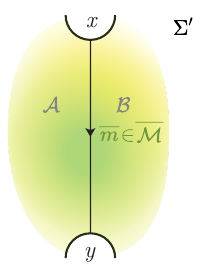}
\caption{}\label{pic_moveOR}
\end{figure}

We denote the node labels of the nodes adjacent to the $\mcat$-labeled defect line by $x$ and $y$.
In the definition of ray categories, a distinction is made between incoming and outgoing defect lines at a certain node.
This determines whether the bimodule category that labels the defect line is used, or its opposite.
Due to this, the ray categories (and the node categories) before and after the move are identical -- up to the identification \smash{$\opcat{\opcat{\mcat}} \cong \mcat$}, which becomes an equality for strict pivotal structures on $\acat$ and $\bcat$.
Hence, there is a canonical identification of the node spaces before and after the move, defining the map
$\movemap \colon \nodespace \to \postmove{\nodespace}$.

The proof that the move $\moveOR$ is a move of invariance is the content of Appendix~\ref{move_OR_pf}.

{\bf $\moveC$ -- contraction.} \label{move_C}
The contraction move allows us to transform two nodes, connected by a~defect line, into a single node, thereby dissolving the defect line (see Figure~\ref{pic_moveC}).
The move~$\moveC$ takes place inside a disk-shaped region $\sigsurf$ with two nodes and three defect lines, labeled by objects $k$, $m$ and $l$ in traced $\acat$-$\bcat$-bimodule categories $\kcat$, $\mcat$ and $\lcat$.
The nodes are labeled by $x \in \kcat\reldel\opcat{\mcat}\reldel$ and $y\in \mcat \reldel \opcat{\lcat}\reldel$.

\begin{figure}[ht]\centering
	\includegraphics{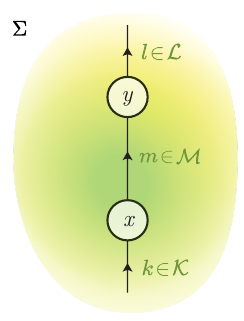} \ \raisebox{25mm}{$\xrightarrow{\moveC}$} \ \includegraphics{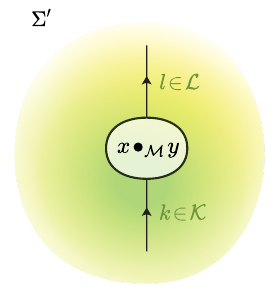}
\caption{}\label{pic_moveC}
\end{figure}

After the move, $\sigsurf$ is replaced by the similar region $\postmove{\sigsurf}$ which only has one node, and from which the defect line labeled by $m$ has disappeared.
To obtain the node label of $\postmove{\sigsurf}$ from the node labels $x$ and $y$ of $\sigsurf$,
recall that the module Eilenberg--Watts equivalence \eqref{eq_mew}
allows us to express the ray categories of $\sigsurf$ and $\postmove{\sigsurf}$ as categories of bimodule functors
\begin{equation*}
	\kcat\reldel\opcat{\mcat}\reldel \cong \bimodfun{\acat}{\bcat}{\mcat}{\kcat}, \qquad \mcat\reldel\opcat{\lcat}\reldel \cong \bimodfun{\acat}{\bcat}{\lcat}{\mcat}, \qquad \kcat\reldel\opcat{\lcat}\reldel \cong \bimodfun{\acat}{\bcat}{\lcat}{\kcat}.
\end{equation*}
The composition of functors defines a composition operation for these particular ray categories, here denoted by the node label $x \mitosis{\mcat} y \in \kcat\reldel\opcat{\lcat}\reldel$ for $\postmove{\sigsurf}$.
Explicitly, we have
\begin{equation}\label{eq_def_gencomp_U}
	U(x \mitosis{\mcat} y) = \homsetin{\mcat}{\sweed{x}{\mcat}}{\sweed{y}{\mcat}} \otimes  \sweed{x}{\kcat} \deligne \sweed{y}{\opcat{\lcat}} \in \kcat \deligne \opcat{\lcat}.
\end{equation}

The last datum needed to specify the move is the linear map
\begin{equation*}
	\movemap \colon\ \nodespace = \homset{k \deligne \opcat{m}}{U(x)} \otimes \bigl\langle m \deligne \opcat{l},U(y)\bigr\rangle \to \bigl\langle k \deligne\opcat{l},U(x\mitosis{\mcat} y)\bigr\rangle = \postmove{\nodespace}
\end{equation*}
between node spaces, appearing in Definition~\ref{def_invariance}.
For this move, $\movemap$ is given by the following composition of morphisms
\begin{gather}
		 \homset{k \deligne \opcat{m}}{U(x)} \otimes \bigr\langle m \deligne \opcat{l},U(y)\bigr\rangle\nonumber\\
\qquad		\cong \homset{k}{\sweed{x}{\kcat}} \otimes \homset{\opcat{m}}{\sweed{x}{\opcat{\mcat}}} \otimes \homset{m}{\sweed{y}{\mcat}} \otimes \bigr\langle\opcat{l},\sweed{y}{\opcat{\lcat}}\bigr\rangle \nonumber\\
\qquad	
		= \homset{k}{\sweed{x}{\kcat}} \otimes \homset{\sweed{x}{\mcat}}{m} \otimes \homset{m}{\sweed{y}{\mcat}} \otimes \bigr\langle\opcat{l},\sweed{y}{\opcat{\lcat}}\bigr\rangle \nonumber \\
	\qquad	\xrightarrow{\text{compose}} \homset{k}{\sweed{x}{\kcat}} \otimes \homset{\sweed{x}{\mcat}}{\sweed{y}{\mcat}} \otimes \bigr\langle\opcat{l},\sweed{y}{\opcat{\lcat}}\bigr\rangle \nonumber \\
	\qquad	\cong \bigr\langle k\deligne \opcat{l},\homsetin{\mcat}{\sweed{x}{\mcat}}{\sweed{y}{\mcat}} \otimes  \sweed{x}{\kcat} \deligne \sweed{y}{\opcat{\lcat}}\bigr\rangle= \bigr\langle k\deligne \opcat{l},U(x \mitosis{\mcat} y)\bigr\rangle.\label{eq_gencomp_def}
\end{gather}

\begin{figure}[ht]\centering
	\includegraphics{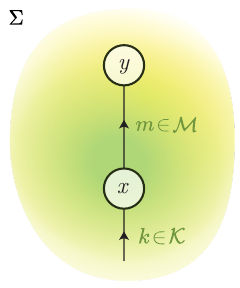} \ \raisebox{20mm}{$\xrightarrow{\moveC}$} \ \includegraphics{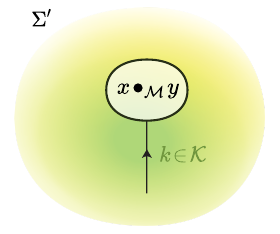}
\caption{}\label{pic_moveC_oneside}
\end{figure}

There is also a ``one-sided'' version of the $\moveC$-move, where the $\lcat$-labeled defect line (or alternatively, the $\kcat$-labeled defect line) is not present (see Figure~\ref{pic_moveC_oneside}).
%
The composition operation for node labels is well defined even in this case.
Note however that a~move without $\kcat$ \emph{and} $\lcat$ does not exist, as this would result in a node without any defect lines, which is not allowed.

The proof that the move $\moveC$ is a move of invariance is the content of Appendix~\ref{move_C_pf}.

\begin{Remark}\label{rem_cmove_specializes_to_comp}
	A special case we are interested in is that where $\mcat = \lcat = \kcat$ and $x = y ={} \zewcoid{\mcat}$.
	In this case, the node spaces are
$\nodespace \cong \homset{k}{m} \otimes \homset{m}{l}$ and $\postmove{\nodespace} \cong \homset{k}{l}$,
	and the linear map $\psi$ specializes to the composition of morphisms in $\mcat$.
\end{Remark}

{\bf $\moveEF$ -- edge fusion.} \label{move_EF}
The edge fusion move allows us to merge parallel defect lines which share start and end nodes (see Figure~\ref{pic_moveEF}).
The region $\sigsurf$ is disk-shaped and contains two defect lines, running in parallel between the only two nodes that lie partially in $\sigsurf$.
In contrast, the post-move region $\postmove{\sigsurf}$ features just a single defect line connecting the node segments.
The three domains of $\sigsurf$ are labeled by spherical fusion categories~$\acat$,~$\bcat$ and $\ccat$, and the defect lines are labeled by objects in bimodule categories ${m \in \bimod{\acat}{\mcat}{\bcat}}$ and $n \in \bimod{\bcat}{\ncat}{\ccat}$, respectively.
The nodes are appropriately labeled by~${x \in \cdots \reldel \opcat{\ncat} \reldel \opcat{\mcat} \reldel \cdots}$ and~${y \in \cdots \reldel \mcat \reldel \ncat \reldel \cdots}$.
The single defect line in the region $\postmove{\sigsurf}$ is labeled by the $\acat$-$\ccat$-bimodule category $\mcat \reldel \ncat$.
The object label $m\lexreldel n$ for the defect line involves the induction functor, see \eqref{eq_induction_reldel_objects}.
The ray categories $\raycat$ and $\postmove{\raycat}$ associated with $\sigsurf$ and $\postmove{\sigsurf}$ are canonically equivalent to those of the pre-move surface $\sigsurf$, allowing us to use the same objects $x$ and $y$ for the rays of $\postmove{\sigsurf}$.

\begin{figure}[ht]\centering
	\includegraphics{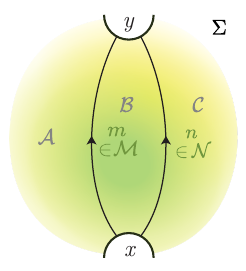} \ \raisebox{23mm}{$\xrightarrow{\moveEF}$} \ \includegraphics{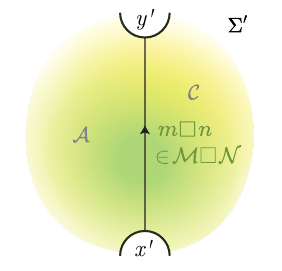}
\caption{}\label{pic_moveEF}
\end{figure}

However, the node categories $\nodecat$ and $\postmove{\nodecat}$ are different; they are related by the forgetful functor $U_\bcat\colon \postmove{\nodecat} \to \nodecat$, which forgets the $\bcat$-balancing.
Together with the forgetful functors~${\postmove{U}\colon \postmove{\raycat} \to \postmove{\nodecat}}$ and $U \colon \raycat \to \nodecat$, we have $U = U_\bcat \circ \postmove{U}$.

The linear map $\movemap \colon \nodespace \to \postmove{\nodespace}$ between node spaces is given by the composition
\begin{align}
		\tilde{\movemap}  \colon\  \nodespace &= \homset{\opcat{n}\deligne\opcat{m}\deligne \cdots}{U(x)} \otimes \homset{m\deligne n\deligne \cdots}{U(y)}\nonumber \\
		& \cong \bigl\langle\opcat{m \lexreldel n} \deligne \cdots,\postmove{U}(x)\bigr\rangle \otimes \homset{m \rexreldel n \deligne \cdots}{\postmove{U}(y)} = \postmove{\nodespace},\label{eq_moveEF_movemap_def}
\end{align}
scaled by the factor $\catdim{\bcat}$, meaning \smash{$\psi := \frac{\tilde{\psi}}{\catdim{\bcat}}$}.
In \eqref{eq_moveEF_movemap_def} the ambidextrous adjunction $U_\bcat \dashv \leftind$ from~\eqref{eq_induction_adjunction} was used.

The proof that the move $\moveEF$ is a move of invariance is the content of Appendix~\ref{move_EF_pf}.

\subsubsection{Composite moves}
In addition to these elementary moves, we will define the following composite moves $\moveDV$, $\moveL$, and~$\moveDE$.
The loop move $\moveL$ removes a defect line that starts and ends at the same node, and the other two moves $\moveDV$ and $\moveDE$ allow us to insert nodes and defect lines which are, in an appropriate sense, transparent.

{\bf ${\moveDV}$ -- silent node.} 
The $\moveDV$-move can be performed on a region $\sigsurf$ containing only a defect line, labeled by some object $m \in \mcat$.
In the post-move region $\postmove{\sigsurf}$, the defect line is split in two by a new node, which is labeled by the silent object $\zewid{\mcat} \in \opcat{\mcat} \reldel \mcat \reldel$ (see Figure~\ref{pic_moveDV}).
Since the region has no node, the node space of $\sigsurf$ is $\nodespace = \field$, and the linear map $\psi \colon \nodespace \to \postmove{\nodespace}$ is specified by the vector
$\psi (1) := \soc (\id_m) \in \homset{\opcat{m} \deligne m}{\ewid{\mcat}} = \postmove{\nodespace}$.

\begin{figure}[ht]\centering
	\raisebox{8mm}{\includegraphics{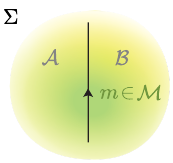}} \ \raisebox{20mm}{$\xrightarrow{\moveDV}$} \ \includegraphics{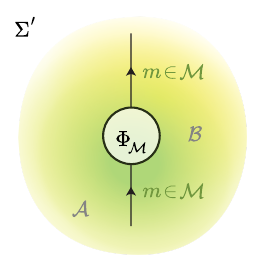}
\caption{}\label{pic_moveDV}
\end{figure}

The proof that the move $\moveDV$ is a move of invariance is the content of Appendix~\ref{move_DV_pf}.

\begin{figure}[ht]\centering
\includegraphics{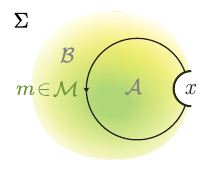} \ \raisebox{13mm}{$\xrightarrow{\moveL}$} \ \includegraphics{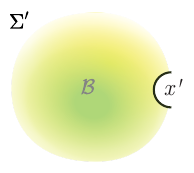}
\caption{}\label{pic_moveL}
\end{figure}

{\bf ${\moveL}$ -- loop move.} 
A loop attached to a node can be removed. Note that this is only possible when at least one other defect line is attached to the node, as otherwise, the move would result in a node without defect lines, which is not allowed (see Figure~\ref{pic_moveL}).
The node label $x'$ is given by the contraction $x' ={} \zewid{\mcat} \mitosis{\opcat{\mcat}\reldel \mcat} x$.
Beware that in general, $x$ and~$x'$ live in non-equivalent ray categories -- the silent object $\zewid{\mcat}$ is a unit with respect to the contraction~$\mitosis{\mcat}$, not $\mitosis{\opcat{\mcat}\reldel \mcat}$.
Concretely, if $U(x) = \sweed{x}{\opcat{\mcat} \reldel \mcat} \deligne \sweed{x}{\rest}$, where $\sweed{x}{\rest}$ is a rest term, we have
\begin{align}\label{eq_lmove_postmove_raylabel}
		U(x') &= \homsetin{\opcat{\mcat}\reldel \mcat}{\sweed{x}{\opcat{\mcat}\reldel \mcat}}{\zewcoid{\mcat}}\otimes \sweed{x}{\rest} .
\end{align}
Thus, we can specify the linear map $\movemap \colon \nodespace \to \postmove{\nodespace}$ on $f \in \nodespace$ as $\movemap(f) \colon n \to \homsetin{\opcat{\mcat}\reldel \mcat}{\sweed{x}{\opcat{\mcat}\reldel \mcat}}{\zewcoid{\mcat}} \otimes \sweed{x}{\rest}$, where $n$ stands for the labels of the other defect lines attached to the node.
Let us write $f = \sweed{f}{\opcat{\mcat} \deligne \mcat} \otimes \sweed{f}{\rest}$, where
\begin{equation*}
	\sweed{f}{\opcat{\mcat} \deligne \mcat} \colon\ \opcat{m} \deligne m \to \sweed{x}{\opcat{\mcat} \deligne \mcat} \qquad\text{and} \qquad \sweed{f}{\rest} \colon\ n \to \sweed{x}{\rest}.
\end{equation*}
Then $\movemap(f)$ is given as
\begin{equation*}
	\movemap(f) := r(\soc(\id_m) \circ \sweed{f}{\opcat{\mcat} \deligne \mcat}) \otimes \sweed{f}{\rest},
\end{equation*}
where the isomorphism $\soc \colon \homset{m}{m} \cong \homset{\opcat{m}\deligne m}{\ewid{\mcat}}$ was defined in \eqref{eq_sil}, and
\[
r\colon\ \homsetin{\opcat{\mcat}\deligne \mcat}{\sweed{x}{\opcat{\mcat}\deligne \mcat}}{\ewid{\mcat}} \to \homsetin{\opcat{\mcat}\reldel \mcat}{\sweed{x}{\opcat{\mcat}\reldel \mcat}}{\zewcoid{\mcat}}\]
 is the retraction from Lemma~\ref{lemma_homspace_projection}.

The proof that the move $\moveL$ is a move of invariance is the content of Appendix~\ref{move_L_pf}.

\begin{figure}[ht]\centering
\includegraphics{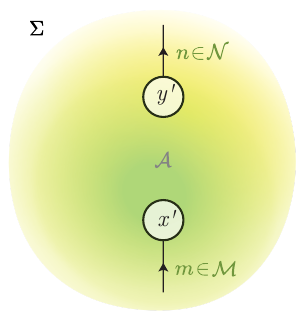} \ \raisebox{25mm}{$\xrightarrow{\moveDE}$} \ \includegraphics{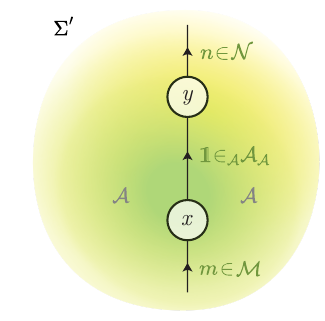}
\caption{}\label{pic_moveDE}
\end{figure}

{\bf ${\moveDE}$ -- transparent edge.} \label{move_DE}
An edge which is \emph{transparent} -- that is, labeled by the regular bimodule category $\acat$, and moreover labeled by the monoidal unit $\one \in \acat$, can be inserted between two nodes as shown in Figure~\ref{pic_moveDE}.
More precisely, let $x \in \modulecent{\acat}{\mcat}$ and $y \in Z_\acat\bigl(\opcat\ncat\bigr)$ be objects in the centers of the $\acat$-$\acat$-bimodule categories $\mcat$ and $\opcat\ncat$.
As such, $x$ and $y$ are appropriate node labels for the nodes of the region~$\sigsurf$.
After the move, the rays of the region $\postmove{\sigsurf}$ must be labeled with objects in the ray categories~${\opcat{\acat} \reldel \mcat \reldel = Z_\acat\bigl(\opcat{\acat} \reldel \mcat\bigr)}$ and $\opcat{\ncat} \reldel \acat \reldel = Z_\acat\bigl(\opcat{\ncat}\reldel \acat\bigr)$.
We denote the (partial) forgetful functors $Z_\acat\bigl(\opcat{\acat} \reldel \mcat\bigr) \to \opcat{\acat} \reldel \mcat$ and $Z_\acat\bigl(\opcat{\ncat}\reldel \acat\bigr) \to \opcat{\ncat}\reldel \acat$ by \smash{$\tilde{U}$}.
The equivalence of bimodule categories $\opcat{\ncat} \cong \opcat{\ncat} \reldel \acat$ from \eqref{eq_reldel_unitality} extends to an equivalence of linear categories ${Z_\acat\bigl(\opcat\ncat\bigr) \cong Z_\acat\bigl(\opcat{\ncat}\reldel \acat\bigr)}$ because $\modulecent{\acat}{-}$ is a 2-functor from bimodule categories to linear categories \cite[Proposition~3.7]{fss-trace}.
The node label $y$ in the region $\postmove{\sigsurf}$ is the image of $y$ under this equivalence, and $\postmove{x}$ is defined similarly.
The precise form \eqref{eq_reldel_unitality_explicit} of the equivalence shows that
$\tilde{U}(\postmove{x}) = \opcat{\one} \lexreldel U(x)$ and $\tilde{U}(\postmove{y}) = U(y) \lexreldel \one$,
and applying the complete forgetful functors, we find using \eqref{eq_induction_functors} that
\begin{equation*}
	U(\postmove{x}) = \bigoplus_\si{a} \opcat{\si{a}} \deligne \si{a}U(x) \qquad\text{and}\qquad U(\postmove{y}) = \bigoplus_\si{a} U(y)\si{a}^* \deligne \si{a}.
\end{equation*}

Finally, we need to specify a linear map $\movemap\colon \nodespace \to \postmove{\nodespace}$ between node spaces.
We pic a~vector~${f \otimes g \in \nodespace}$, where $f$ and $g$ are morphisms
$f\colon m \to U(x)$ and $ g \colon \opcat{n} \to U(y)$.
The value of $\movemap$ on~${f \otimes g}$ is $\movemap(f\otimes g) = \postmove{f} \otimes \postmove{g}$, which is a tensor product of morphisms $\postmove{f}\colon \opcat{\one} \deligne m \to \bigoplus_\si{a} \opcat{\si{a}} \deligne \si{a}U(x)$ and $\postmove{g} \colon \opcat{n} \deligne\one \to \bigoplus_\si{a} U(y)\si{a}^* \deligne \si{a}$ of $\postmove{\sigsurf}$.
Due to Schur's lemma, all simple components of $\postmove{f}$ and~$\postmove{g}$ except for $\tc{\postmove{f}}{\one}$ and $\tc{\postmove{g}}{\one}$ are zero.
Thus setting
$\tc{\postmove{f}}{\one} := \id_{\opcat{\one}} \deligne f$ and $\tc{\postmove{g}}{\one} := g \deligne \id_\one
$
defines $\postmove{f}$ and $\postmove{g}$, and thus the linear map $\movemap$.

The proof that the move $\moveDE$ is a move of invariance can be found in \cite[Section~4.3.3]{jf-thesis}.

With the list of elementary and composite moves complete, we are ready to formulate the next theorem, which states that the moves leave the evaluation invariant.
\begin{Theorem}\label{thm_moves}
	The moves $\moveOR$, $\moveC$, $\moveEF$, $\moveDV$, and $\moveL$ defined above are moves of invariance in the sense of Definition~{\rm\ref{def_invariance}} for the evaluation $\evaluate{-}$ of extruded graphs from Definition~{\rm\ref{def_eval}}.
\end{Theorem}
Before we get to the proof of Theorem~\ref{thm_moves}, which is the content of Appendix~\ref{app_moveproofs},
we will see that Theorem~\ref{thm_moves} implies the uniqueness result discussed in the next section.

\subsection{Uniqueness of the evaluation}\label{sec_uniqueness}
To state the uniqueness property of the evaluation procedure $\evaluate{-}$ defined in Definition~\ref{def_eval}, we introduce some terminology.
\begin{Definition}\label{def_abs_eval}
A \emph{pre-evaluation} $\alteval{-}$ is a choice of linear map $\alteval{\sigsurf}\colon N \to T$ for every extruded graph $\sigsurf$ on $\mathbb{S}^2$ with node space $N$ and block space $T$.	

Consider the set $E$ of triples
\begin{equation*}
		E := \class{(\sigsurf, f, \varphi)}{\,\begin{matrix}
				\text{$\sigsurf$ is an extruded graph on $\mathbb{S}^2$ with node space $\nodespace$} \\
				\text{and block space $\blockspace$, $f\in \nodespace$ and $\varphi \in \blockspace^*$.}
			\end{matrix}}.
\end{equation*}
A \emph{scalar pre-evaluation} is a function $s\colon E \to \field$, which is linear in the variables $f$ and $\varphi$.
\end{Definition}
\begin{Observation}\label{obs_eval_and_scalar}
	There is an obvious bijection between pre-evaluations and scalar pre-evaluations.
	For a fixed extruded graph $\sigsurf$, a scalar pre-evaluation is a linear map $s_\sigsurf \colon \nodespace \otimes \blockspace^* \to \field$.
	Likewise, a pre-evaluation is, for each extruded graph $\sigsurf$, a linear map $\alteval{\sigsurf} \colon \nodespace \to \blockspace$.
	Since $T$ is finite-dimensional, these spaces of linear maps coincide.
\end{Observation}

Clearly, our evaluation from Definition~\ref{def_eval}, restricted to extruded graphs on the sphere, is a~pre-evaluation in this sense.
A pre-evaluation $\alteval{-}$ is said to be \emph{compatible} with a move $\mathbf{M}$, if the diagram
\begin{equation*}
	\begin{tikzcd}
		\nodespace \arrow[r, "\alteval{\sigsurf_{T}}"{above}] \arrow[d, "\movemap"{left}] & \blockspace \arrow[d, "\Movemap"] \\
		\postmove{\nodespace} \arrow[r, "\alteval{\postmove{\sigsurf_{T}}}"{below}] & \postmove{\blockspace}
	\end{tikzcd}
\end{equation*}
commutes in a generalization of \eqref{diag_movediag}.
We say that a scalar pre-evaluation $s$ is compatible with~$\mathbf{M}$ if $s(\sigsurf_T, f, \varphi) = s\bigl(\postmove{\sigsurf_T}, \movemap(f), \varphi \circ \Movemap^{-1}\bigr)$.
\begin{Theorem}\label{thm_uniqueness}
	The evaluation $\evaluate{-}$ defined in Definition~{\rm\ref{def_eval}} is the unique pre-evaluation such that lasso graphs are evaluated as in Proposition~{\rm\ref{move_SV}}, and which is compatible with the moves $\moveOR$, $\moveC$, and $\moveEF$.
	
	The scalar evaluation defined in \eqref{eq_total_eval} is the unique scalar pre-evaluation such that lasso graphs are evaluated as in Proposition~{\rm\ref{move_SV}}, and which is compatible with the moves $\moveOR$, $\moveC$, and~$\moveEF$.
\end{Theorem}
\begin{proof}
	We consider any pre-evaluation $\alteval{-}$ which evaluates lasso graphs as in Proposition~\ref{move_SV} and for which the moves $\moveOR$, $\moveC$, and $\moveEF$ are moves of invariance.
	We then show that $\alteval{\sigsurf} = \evaluate{\sigsurf}$ for any extruded graph $\sigsurf$ on $\mathbb{S}^2$.
	
	The network of $n$ defect lines and $k$ nodes in $\sigsurf$ form a connected graph with $n$ edges and $k$ vertices.
	Any connected graph has a spanning tree, that is, a sub-graph without cycles which contains all vertices of the original graph.
	We pick such a spanning tree, and in it an edge.
	Note that this edge is not a loop.
	We may thus apply the contraction move $\moveC$ to the corresponding defect line in $\sigsurf$, obtaining an extruded graph $\postmove{\sigsurf}$ with $(n-1)$ defect lines and $(k-1)$ nodes.
	Since~$\moveC$ is by assumption compatible with the pre-evaluation $\alteval{-}$, we find $\Movemap_\mathbf{C} \circ \alteval{\sigsurf} = \alteval{\postmove{\sigsurf}} \circ \movemap_\mathbf{C}$,
	where $\movemap_\mathbf{C}\colon \nodespace \to \postmove{\nodespace}$ and $\Movemap_\mathbf{C} \colon \blockspace \to \postmove{\blockspace}$ denote the linear maps defined in~$\moveC$.
	This step is repeated until we arrive at an extruded graph $\sigsurf_1$ with only a single node.
	Viewing all individual $\moveC$-moves as one composite move, it is clear that this move is compatible with $\alteval{-}$: $\Movemap_1 \circ \alteval{\sigsurf} = \alteval{\sigsurf_1} \circ \movemap_1$.
	As depicted in Figure~\ref{pic_bouquet_of_circles}, the defect lines of $\sigsurf_1$, all start and end at the unique node.
	The sphere $\sigsurf_1$ with the interior of the node removed is a closed disk with embedded defect lines that start and end on the boundary.
	Since the lines do not intersect, there is at least one defect line whose start and end points are neighbors on the boundary of the disk.
	To a defect line with this property, the loop move $\moveL$ can be applied.
	As we will see below (see Appendix~\ref{move_L_pf}), the~$\moveL$-move can be realized as a composition of $\moveC$, $\moveOR$, and $\moveEF$-moves.
	Thus, the evaluation~$\alteval{-}$ is compatible with $\moveL$.
	This step of applying the move $\moveL$ is repeated until only one defect line is left.

\begin{figure}[ht]\centering
\includegraphics{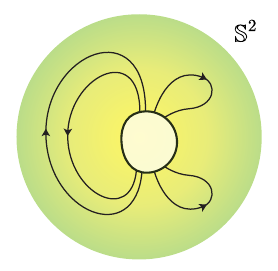}
\caption{}\label{pic_bouquet_of_circles}
\end{figure}
	
	The resulting extruded graph is a lasso graph $Q$.
	Having constructed a move between $\sigsurf$ and $Q$ compatible with $\alteval{-}$, we denote the associated linear maps between node spaces and block spaces by $\movemap\colon \nodespace \to \nodespace_Q$ and $\Movemap \colon \blockspace \to \blockspace_Q$.
	By assumption, Proposition~\ref{move_SV} holds for the pre-evaluation $\evaluate{-}$.
	As a consequence, $\evaluate{-}$ and $\alteval{-}$ are equal on lasso graphs such as $Q$, and the following diagram commutes
	\begin{equation*}
		\begin{tikzcd}
			\nodespace \arrow[d, "\alteval{\sigsurf}"{left}] \arrow[r, "\movemap"] & \nodespace_Q \arrow[d, "\alteval{Q}"{left}, "\evaluate{Q}"{right}] & \nodespace \arrow[d, "\evaluate{\sigsurf}"] \arrow[l, "\movemap"{above}] \\
			\blockspace \arrow[r, "\Movemap"{below}] & \blockspace_Q & \blockspace \arrow[l, "\Movemap"].
		\end{tikzcd}
	\end{equation*}
	Since $\Movemap$ is an isomorphism, we can deduce that $\alteval{\sigsurf} = \evaluate{\sigsurf}$.
	
	The statement for the scalar evaluation follows immediately from the first part of the theorem and Observation~\ref{obs_eval_and_scalar}.
\end{proof}

\appendix

\section{Details on moves}\label{app_moveproofs}

\subsection{Proof strategy for elementary moves} \label{sec_moveproof}
The discussion of the elementary moves in the subsequent sections follows a general pattern, which we outline here.

Recall that any move $\mathbf{M}$ between extruded graphs $\sigsurf$ and $\postmove{\sigsurf}$ involves linear maps
\begin{alignat*}{4}
	&\movemap\colon\ &&\nodespace \to \postmove{\nodespace}\qquad &&\text{between node spaces,}&\\
	&\Pmovemap \colon\ &&\preblock \to \postmove{\preblock} \qquad &&\text{between pre-block spaces, and}&\\
	&\Movemap \colon\ &&\blockspace \to \postmove{\blockspace} \qquad &&\text{between block spaces.}&
\end{alignat*}
In order to show that an elementary move is an invariance, we consider the diagram \eqref{diag_invariance_scheme} below, which is a subdivision of \eqref{diag_movediag}
\begin{equation}\label{diag_invariance_scheme}
	\begin{tikzcd}[column sep = huge]
		\nodespace \arrow[r,"{\delta}"] \arrow[d, "\movemap"] & \preblock \arrow[r, "\holsurj"] \arrow[d, "\Pmovemap"] & \blockspace \arrow[d, "\Movemap"] \\
		\postmove{\nodespace} \arrow[r,"{\postmove{\delta}}"]& \postmove{\preblock} \arrow[r, "\postmove{\holsurj}"] & \postmove{\blockspace}.
	\end{tikzcd}
\end{equation}
We will deduce the commutativity of the outer paths in \eqref{diag_invariance_scheme} from the commutativity of the two cells.
Commutativity of the outer paths then implies by Definition~\ref{def_invariance} that we have a move of invariance at hand.

For the right-hand square in \eqref{diag_invariance_scheme},
this is facilitated by the fact that $\blockspace$ is a retract of $\preblock$.
Consider the following general observation.
\begin{Lemma}\label{lem_morphs_of_retracts}
	Let $\xcat$ be a category, and let
	\begin{equation*}
		\begin{tikzcd}
			x \arrow[r,hook,"\iota"{below}] & y \arrow[l, two heads, bend right, "\pi"{above}]
		\end{tikzcd}
		\qquad\text{and}\qquad
		\begin{tikzcd}
			x' \arrow[r,hook,"\iota'"{below}] & y' \arrow[l, two heads, bend right, "\pi'"{above}]
		\end{tikzcd}
	\end{equation*}
	be retracts in $\xcat$.
	Denote the corresponding idempotents by $h = \iota \circ \pi$ and $h' = \iota' \circ \pi'$.
	A~morphism~${\varphi \colon y \to y'}$ which satisfies
	\begin{equation}\label{diag_map_of_retracts_cond}
		\begin{tikzcd}
			y \arrow[r, "h"] \arrow[d, "\varphi"] & y \arrow[d, "\varphi"] \\
			y' \arrow[r, "h'"] & y'
		\end{tikzcd}
	\end{equation}
	is a morphism of retracts, meaning that
	\begin{equation}\label{diag_map_of_retracts}
		\begin{tikzcd}
			x \arrow[r, "\iota"] \arrow[d, "\psi"] & y \arrow[d, "\varphi"] \\
			x' \arrow[r, "\iota'"] & y'
		\end{tikzcd}
		\qquad\text{and}\qquad
		\begin{tikzcd}
			y \arrow[r, "\pi"] \arrow[d, "\varphi"] & x \arrow[d, "\psi"] \\
			y' \arrow[r, "\pi'"] & x'
		\end{tikzcd}
	\end{equation}
	commute for $\psi = \pi' \circ \varphi \circ \iota$.
	Conversely, any pair of morphisms $(\varphi,\psi)$ which satisfies \eqref{diag_map_of_retracts} also satisfies \eqref{diag_map_of_retracts_cond}, and $\psi = \pi' \circ \varphi \circ \iota$ holds.
\end{Lemma}
Applied to our situation, Lemma~\ref{lem_morphs_of_retracts} becomes the following.
\begin{Lemma}\label{lem_moveproof_hol_condition}
	The commutativity of the diagram
	\begin{equation}\label{diag_moveproof_hol_condition}
		\begin{tikzcd}
			\preblock \arrow[r, "\holidem"] \arrow[d, "\Pmovemap"{left}] & \preblock \arrow[d, "\Pmovemap"]\\
			\postmove{\preblock} \arrow[r, "\postmove{\holidem}"] & \postmove{\preblock}
		\end{tikzcd},
	\end{equation}
	involving the holonomy \emph{idempotents} $\holidem$, $\postmove{\holidem}$ instead of the projectors $\holsurj$, $\postmove{\holsurj}$ implies the commutativity of the right square in \eqref{diag_invariance_scheme}, and thus that the elementary move $\mathbf{M}$ is an invariance.
\end{Lemma}

For every elementary move, we thus have to check two things: the commutativity of the left square in the diagram \eqref{diag_invariance_scheme}, and the commutativity of the square \eqref{diag_moveproof_hol_condition}.
This is done in the following sections.

\subsection{Elementary moves}
This is the first of two sections that comprise the proof of Theorem~\ref{thm_moves}.
For some moves, we have to prove the invariance property by directly using Definition~\ref{def_eval}, the definition of the evaluation procedure.
These moves are called elementary.
Other moves can be obtained as compositions of elementary moves, and we will treat them in Appendix~\ref{sec_compmoves}.

\subsubsection[OR -- orientation reversal]{${\moveOR}$ -- orientation reversal} \label{move_OR_pf}
In the introductory discussion in Section~\ref{sec_moves}, we have already seen that the node spaces before and after the move $\moveOR$ can be identified, and we can indeed view them as equal without loss of generality, by assuming that the acting categories $\acat$ and $\bcat$ are strictly pivotal.
As a consequence of the equality of node spaces, the pre-block spaces are also equal.
This makes the commutativity of the left square in \eqref{diag_invariance_scheme} particularly easy to prove, because the linear maps $\movemap$ and $\Pmovemap$ are identities in this case.

Explicitly, using \eqref{eq_preblock_shortform}, the pre-block space is of the form
\begin{equation*}
	\preblock = \homsetin{\mcat}{\sweed{y}{\opcat{\mcat}}}{\sweed{x}{\mcat}} \otimes \cdots = \homsetin{\opcat{\mcat}}{\sweed{x}{\opcat{\opcat{\mcat}}}}{\sweed{y}{\opcat{\mcat}}} \otimes \cdots = \postmove{\preblock}.
\end{equation*}

We must now show that the diagram \eqref{diag_moveproof_hol_condition} commutes.
Unraveling the definition of the two holonomy idempotents $\holidem$ for $\sigsurf$ and $\holidem'$ for $\sigsurf'$ reveals that $h = h'$.
This is because in the holonomy balancing \eqref{eq_def_holbal}, say for the domain $\acat$, the balancing
\begin{equation*}
	\homsetin{\mcat}{a^*\sweed{y}{\opcat{\mcat}}}{\sweed{x}{\mcat}} \otimes \cdots \cong \homsetin{\mcat}{\sweed{y}{\opcat{\mcat}}}{a\sweed{x}{\mcat}} \otimes \cdots
\end{equation*}
appears as a composition factor.
Thus balancing is the same for both $\preblock$ and $\postmove{\preblock}$.
A similar statement hold for the domain $B$.

As a consequence, the linear map $\Movemap$ between block spaces is an isomorphism.

\subsubsection[C -- contraction]{${\moveC}$ -- contraction} \label{move_C_pf}
The second move for which we show the invariance property is the contraction move, which is depicted in Figure~\ref{pic_moveC}.
Following the general proof strategy outlined in Appendix~\ref{sec_moveproof}, we now need to check that the left cell of the diagram \eqref{diag_invariance_scheme} commutes.
To this end, consider an extruded graph $\sigsurf_T$ which contains the disk $\sigsurf$ from Figure~\ref{pic_moveC}.
The node space $\nodespace$ of~$\sigsurf_T$ is given by $\nodespace = \homset{k\deligne\opcat{m}}{U(x)} \otimes \bigl\langle m\deligne\opcat{l},U(y)\bigr\rangle\otimes\nodespace_\rest$, where $\nodespace_\rest$ is a rest term that takes into account the nodes of $\sigsurf_T$ that lie outside of $\sigsurf$.
The commutativity of \eqref{diag_invariance_scheme} amounts to the equality of the linear maps $\postmove{\delta} \circ \movemap$ and $\Pmovemap \circ \delta$, which we do by evaluating them on a~vector~${f = \sweed{f}{\kcat} \otimes \sweed{f}{\opcat{\mcat}} \otimes \sweed{f}{\mcat} \otimes \sweed{f}{\opcat{\lcat}} \otimes \sweed{f}{\rest} \in \nodespace}$.
We then calculate
\begin{align*}
	\begin{aligned}
		\postmove{\delta} \circ \movemap (f) &= \postmove{\delta} ( (\sweed{f}{\mcat} \circ \sweed{f}{\opcat{\mcat}}) \otimes \sweed{f}{\kcat}\otimes \sweed{f}{\opcat{\lcat}} \otimes \sweed{f}{\rest} ) \\
		&= (\sweed{f}{\mcat} \circ \sweed{f}{\opcat{\mcat}}) \otimes \bigoplus_{\si{k},\si{l},\dots} \pdim{\si{k}}\pdim{\si{l}}\pdim{\cdots} ( (\sweed{f}{\kcat}\otimes \sweed{f}{\opcat{\lcat}} \otimes \sweed{f}{\rest}) \circ (\basisel{\si{k}}{k} \otimes \basisel{\opcat{\si{l}}}{\opcat{l}} \otimes \cdots) ),
	\end{aligned}
\end{align*}
where the first equality uses the definition of $\movemap$ from \eqref{eq_gencomp_def} and the second equality is the definition~\eqref{eq_eval_delta} of $\postmove{\delta}$.
On the other hand, we have
\begin{align*}
		\Pmovemap \circ \delta (f)\eqwithref{eq_eval_delta}{}& \bigoplus_{\si{k}, \si{m}, \si{l},\dots} \pdim{\si{k}}\pdim{\si{m}}\pdim{\si{l}}\pdim{\cdots} \Pmovemap ( f \circ (\basisel{\si{k}}{k} \otimes \basisel{\opcat{\si{m}}}{\opcat{m}} \otimes \basisel{\si{m}}{m} \otimes \basisel{\opcat{\si{l}}}{\opcat{l}} \otimes \cdots) )\\
={}& \bigoplus_{\si{k}, \si{m}, \si{l},\dots} \pdim{\si{k}}\pdim{\si{m}}\pdim{\si{l}}\pdim{\cdots}\Pmovemap ( ((\sweed{f}{\kcat} \circ \basisel{\si{k}}{k}) \otimes
		(\sweed{f}{\opcat{\mcat}} \circ \basisel{\opcat{\si{m}}}{\opcat{m}}) \otimes
		(\sweed{f}{\mcat} \circ \basisel{\si{m}}{m})\\
& \otimes
		(\sweed{f}{\opcat{\lcat}} \circ \basisel{\opcat{\si{l}}}{\opcat{l}}) \otimes
		(\sweed{f}{\rest} \circ \cdots) ) )\\
\eqwithref{eq_Pmovemap}{}& \bigoplus_{\si{k},\si{m}, \si{l},\dots} \pdim{\si{k}}\pdim{\si{m}}\pdim{\si{l}}\pdim{\cdots} \postmove{\delta}_{\si{k},\si{l},\dots} \circ \movemap_{\si{k},\si{l},\dots} ( ((\sweed{f}{\kcat} \circ \basisel{\si{k}}{k}) \otimes
		(\sweed{f}{\opcat{\mcat}} \circ \basisel{\opcat{\si{m}}}{\opcat{m}})\\
& \otimes
		(\sweed{f}{\mcat} \circ \basisel{\si{m}}{m}) \otimes
		(\sweed{f}{\opcat{\lcat}} \circ \basisel{\opcat{\si{l}}}{\opcat{l}}) \otimes
		(\sweed{f}{\rest} \circ \cdots) ) )\\
\eqwithref{eq_gencomp_def}{}& \bigoplus_{\si{k},\si{m}, \si{l},\dots} \pdim{\si{k}}\pdim{\si{m}}\pdim{\si{l}}\pdim{\cdots}
 \postmove{\delta}_{\si{k},\si{l},\dots} ((\sweed{f}{\mcat} \circ \basisel{\si{m}}{m} \circ\basisel{\opcat{\si{m}}}{\opcat{m}} \circ \sweed{f}{\opcat{\mcat}})\\
 & \otimes
		 ((\sweed{f}{\kcat} \circ \basisel{\si{k}}{k}) \otimes
		(\sweed{f}{\opcat{\lcat}} \circ \basisel{\opcat{\si{l}}}{\opcat{l}}) \otimes
		(\sweed{f}{\rest} \circ \cdots) ) )\\
\eqwithref{eq_res_of_id}{}&\bigoplus_{\si{k}, \si{l},\dots}\pdim{\si{k}}\pdim{\si{l}}\pdim{\cdots}
 \postmove{\delta}_{\si{k},\si{l},\dots} ((\sweed{f}{\mcat} \circ \sweed{f}{\opcat{\mcat}}) \otimes
		 ((\sweed{f}{\kcat} \circ \basisel{\si{k}}{k})\\
& \otimes
		(\sweed{f}{\opcat{\lcat}} \circ \basisel{\opcat{\si{l}}}{\opcat{l}}) \otimes
		(\sweed{f}{\rest} \circ \cdots) ) )\\
\eqwithref{eq_eval_delta}{}&(\sweed{f}{\mcat} \circ \sweed{f}{\opcat{\mcat}}) \otimes \bigoplus_{\si{k'}, \si{l'},\si{k}, \si{l},\dots}\pdim{\si{k}}\pdim{\si{l}}\pdim{\si{k'}}\pdim{\si{l'}}\pdim{\cdots}
		 ((\sweed{f}{\kcat} \circ \basisel{\si{k}}{k}) \otimes
		(\sweed{f}{\opcat{\lcat}} \circ \basisel{\opcat{\si{l}}}{\opcat{l}}) \\
&\otimes
		(\sweed{f}{\rest} \circ \cdots) ) \circ (\basisel{\si{k'}}{\si{k}} \otimes \basisel{\opcat{\si{l'}}}{\opcat{\si{l}}} \otimes \cdots) \\
={}&(\sweed{f}{\mcat} \circ \sweed{f}{\opcat{\mcat}}) \otimes \bigoplus_{\si{k}, \si{l},\dots}\pdim{\si{k}}\pdim{\si{l}}\pdim{\cdots}
		 ((\sweed{f}{\kcat} \circ \basisel{\si{k}}{k}) \otimes
		(\sweed{f}{\opcat{\lcat}} \circ \basisel{\opcat{\si{l}}}{\opcat{l}}) \otimes
		(\sweed{f}{\rest} \circ \cdots) ).
\end{align*}
In the last equality, we used Schur's lemma.

\begin{figure}[ht]\centering
\includegraphics{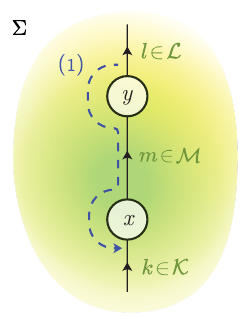} \ \raisebox{25mm}{$\xrightarrow{\moveC}$} \raisebox{1mm}{\includegraphics{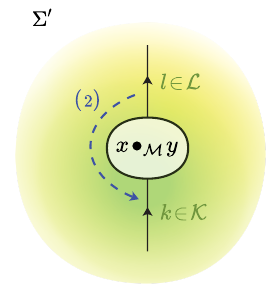}}
\caption{}\label{pic_moveC_pf}
\end{figure}

In order to check the commutativity of the square \eqref{diag_moveproof_hol_condition}, we need to prove that the two composition factors of the holonomy operations $\holidem$ and $\postmove{\holidem}$ corresponding to the paths from $k$ to $l$ agree.
These paths are labeled as (1) and (2) in Figure~\ref{pic_moveC_pf}.
We perform the check only for the holonomy with respect to the domain $\acat$; the argument for the $\bcat$-holonomy is similar.
Specifically, the following diagram \eqref{diag_moveC_pf_diag} needs to commute, in which the left vertical column composes to the path~(1) in Figure~\ref{pic_moveC_pf}, and the right vertical column composes to the path~(2).
For convenience, only the relevant parts of the pre-block spaces are written out (the sums over the variables~$\si{l}$ and~$\si{k}$ are omitted in \eqref{diag_moveC_pf_diag}):
\begin{equation}
	\begin{tikzcd}[row sep=small]
		\bigoplus_\si{n} \homset{\si{n} \deligne \opcat{\si{l}}a}{U(y)} \otimes \homset{\si{k} \deligne\opcat{\si{n}}}{U(x)} \arrow[d] \arrow[r, "\Pmovemap"] & \homset{\si{k} \deligne \opcat{\si{l}}a}{U(x \mitosis{\mcat} y)} \arrow[d] \\
		\bigoplus_\si{n} \homset{\si{n} \deligne \opcat{\si{l}}}{U(y)a^*} \otimes \homset{\si{k} \deligne\opcat{\si{n}}}{U(x)} \arrow[d] \arrow[r] & \homset{\si{k} \deligne \opcat{\si{l}}}{U(x \mitosis{\mcat} y)a^*} \arrow[ddddd] \\
		\bigoplus_\si{n} \homset{\si{n} \deligne \opcat{\si{l}}}{a^*U(y)} \otimes \homset{\si{k} \deligne\opcat{\si{n}}}{U(x)} \arrow[d] & \\
		\bigoplus_\si{n} \homset{a\si{n} \deligne \opcat{\si{l}}}{U(y)} \otimes \homset{\si{k} \deligne\opcat{\si{n}}}{U(x)}\arrow[d] & \\
		\bigoplus_\si{n} \homset{\si{n} \deligne \opcat{\si{l}}}{U(y)} \otimes \homset{\si{k} \deligne\opcat{\si{n}}a}{U(x)} \arrow[d] & \\
		\bigoplus_\si{n} \homset{\si{n} \deligne \opcat{\si{l}}}{U(y)} \otimes \homset{\si{k} \deligne\opcat{\si{n}}}{U(x)a^*} \arrow[d] & \\
		\bigoplus_\si{n} \homset{\si{n} \deligne \opcat{\si{l}}}{U(y)} \otimes \homset{\si{k} \deligne\opcat{\si{n}}}{a^*U(x)} \arrow[d] \arrow[r] & \homset{\si{k} \deligne \opcat{\si{l}}}{a^*U(x \mitosis{\mcat} y)} \arrow[d] \\
		\bigoplus_\si{n} \homset{\si{n} \deligne \opcat{\si{l}}}{U(y)} \otimes \homset{a\si{k} \deligne\opcat{\si{n}}}{U(x)} \arrow[r] & \homset{a\si{k} \deligne \opcat{\si{l}}}{U(x \mitosis{\mcat} y)}\label{diag_moveC_pf_diag}
	\end{tikzcd}
\end{equation}
In the above diagram, the commutativity of the top and bottom squares is trivial.
The large cell commutes by definition of the balancing on $x \mitosis{\mcat} y$, as can be seen by unraveling the definition of the module Eilenberg--Watts functors.

The proof for the one-sided case is completely analogous; however here, only a single domain is adjacent to the $(m\in \mcat)$-labeled defect line, and hence the following paths need to be compared for the holonomy operation, see Figure~\ref{pic_moveC_oneside_pf}.
This completes the proof that the move $\moveC$ is an invariance.	

\begin{figure}[ht]\centering
\includegraphics{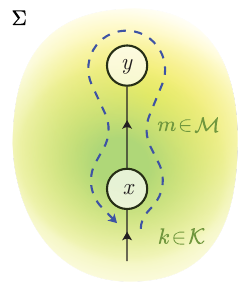} \ \raisebox{25mm}{$\xrightarrow{\moveC}$} \raisebox{5mm}{\includegraphics{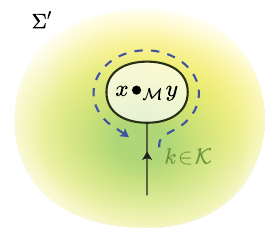}}
\caption{}\label{pic_moveC_oneside_pf}
\end{figure}

The linear map $\Movemap$ between block spaces for the $\moveC$-move also appears in \cite[Lemma~4.39]{fss-statesum}, where it is proved that $\Movemap$ is an isomorphism.

\subsubsection[EF -- edge fusion]{${\moveEF}$ -- edge fusion} \label{move_EF_pf}
As described in Section~\ref{sec_moves_overview}, the edge fusion move allows use to merge parallel defect lines which share start and end nodes.
We proceed to show that $\moveEF$ is a move of invariance, following the proof strategy outlined in Appendix~\ref{sec_moveproof}.
The first step is to show that the left square in the diagram \eqref{diag_invariance_scheme} commutes, i.e., that
\begin{equation}\label{eq_ef_lsq}
	\postmove{\delta} \circ \movemap = \Pmovemap\circ \delta.
\end{equation}
To improve legibility, we perform the calculations using the diagrammatic calculus outlined in Section~\ref{sec_cats}.
The linear map $\movemap \colon \nodespace \to \postmove{\nodespace}$ defined in \eqref{eq_moveEF_movemap_def} is composed of the adjunction isomorphisms~\eqref{eq_induction_adjunction_l_iso} and \eqref{eq_induction_adjunction_r_iso}.
Disregarding tensor factors of the node spaces that are unaffected by the move, $\movemap$ takes for our given defect line labels $m$ and $n$ the form
\begin{equation*}
	\movemap = \bigoplus_{\si{a} \in \bcat} \pica{1}{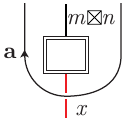} \otimes \bigoplus_{\si{b}\in \bcat} \pdim{\si{b}} \pica{1}{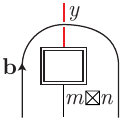}.
\end{equation*}
Implementing the definition \eqref{eq_eval_delta} of $\delta$, we obtain
\begin{equation}\label{eq_EF_lsq_1}
	\postmove{\delta} \circ \movemap = \bigoplus_{\si{k} \in \mcat \reldel \ncat} \pdim{\si{k}} \pdim{\si{b}} \sum_{\si{a}\si{b}} \pica{1}{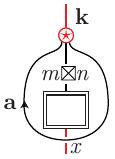} \otimes \pica{1}{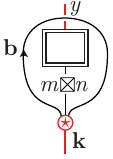},
\end{equation}
where we stress that the $\star$-notation here denotes a sum over dual bases of Hom-spaces in the \emph{relative} Deligne product, which is why the coupons are drawn in red.
Turning to the other side of equation \eqref{eq_ef_lsq}, we read off the expression for $\Pmovemap$ from \eqref{eq_Pmovemap}
\begin{equation}\label{eq_EF_Pmovemap}
	\Pmovemap = \bigoplus_{\si{k}} \bigoplus_{\substack{\si{m}\in \mcat \\ \si{n} \in \ncat}} \pdim{\si{k}}\sum_{\si{a}\si{b}} \pdim{\si{b}} \pica{1}{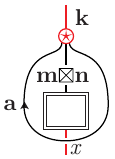} \otimes \pica{1}{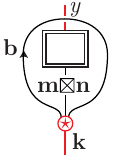},
\end{equation}
and calculate
\begin{align*}
		\Pmovemap \circ \delta &\eqwithref{eq_eval_delta} \bigoplus_{\si{k}\si{m}\si{n}} \pdim{\si{k}} \pdim{\si{m}}\pdim{\si{n}}\sum_{\si{a}\si{b}} \pdim{\si{b}} \pica{1}{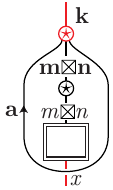} \otimes \pica{1}{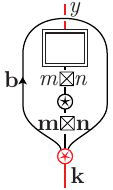}\\
		&\eqwithref{eq_def_indadjr} \bigoplus_{\si{k}\si{m}\si{n}} \pdim{\si{k}} \pdim{\si{m}}\pdim{\si{n}}\sum_{\si{a}\si{b}} \pdim{\si{b}} \pica{1}{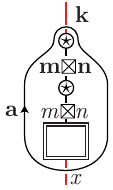} \otimes \pica{1}{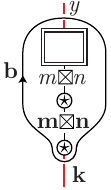}\\
		&\eqwithref{eq_parallel_res} \bigoplus_{\si{k}} \pdim{\si{k}} \sum_{\si{a}\si{b}} \pdim{\si{b}} \pica{1}{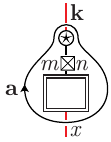} \otimes \pica{1}{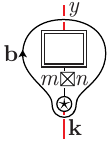}.
\end{align*}
Comparison with \eqref{eq_EF_lsq_1} shows that $\postmove{\delta} \circ \movemap = \Pmovemap\circ \delta$, and hence the left square in \eqref{diag_invariance_scheme} commutes.

We are left to show that the square \eqref{diag_moveproof_hol_condition} (which is the outer square in \eqref{diag_emove_proof} below) commutes for the move $\moveEF$.
To this end, we note that the holonomy idempotent on $\sigsurf$ factorizes into commuting idempotents $\holidem = \holidem_\mathrm{Rest} \circ\holidem_\bcat$, where $\holidem_\bcat$ is the holonomy idempotent with respect to the $\bcat$-labeled domain in $\sigsurf$.
Thus, $\holidem_\bcat \circ \holidem = \holidem = \holidem_\mathrm{Rest} \circ \holidem_\bcat$.
The square \eqref{diag_moveproof_hol_condition} takes, with a~subdivision, the following form
\begin{equation} \label{diag_emove_proof}
	\begin{tikzcd}
		\preblock \arrow[rrr, "\holidem"] \arrow[dr, "\holidem_\bcat"] \arrow[dd, "\Pmovemap"] & & & \preblock \arrow[dd, "\Pmovemap"] \arrow[dl, "\holidem_\bcat"] \\
		& \preblock \arrow[r, "\holidem_\mathrm{Rest}"] & \preblock & \\
		\postmove{\preblock} \arrow[rrr, "\postmove{\holidem}"] \arrow[ur, "c", dashed] & & & \postmove{\preblock} \arrow[ul, dashed, "c"].
	\end{tikzcd}
\end{equation}
We have to find an injective map $c$ such that all smaller cells in \eqref{diag_emove_proof} which involve $c$ commute.
This then implies commutativity of the outer square.
As a map between direct sums
\begin{align*}
		\postmove{\preblock}={}& \bigoplus_{\si{k} \in \mcat\reldel\ncat}\bigl\langle\opcat{\si{k}} \deligne \cdots,\postmove{U}(x)\bigr\rangle \otimes \homset{\si{k} \deligne \cdots}{\postmove{U}(y)}\\
		& \xrightarrow{c} \bigoplus_{\si{q}\in \mcat}\bigoplus_{\si{r}\in\ncat}\homset{\opcat{\si{r}}\deligne\opcat{\si{q}}\deligne \cdots}{U(x)} \otimes \homset{\si{q}\deligne \si{r}\deligne \cdots}{U(y)} = \preblock,
\end{align*}
$c$ is determined by its components $\ctc{c}{\si{k}}{\si{qr}}$, which are
\begin{align}
		\ctc{c}{\si{k}}{\si{qr}}={}& \pdim{\si{q}}\pdim{\si{r}}  	\homset{\basisel{\sweed{\si{k}}{\ncat}}{\si{r}}\deligne \basisel{\sweed{\si{k}}{\mcat}}{\si{q}} \deligne \cdots}{U(x)}\nonumber\\
& \otimes \homset{\basisel{\si{q}}{\sweed{\si{k}}{\mcat}} \deligne \basisel{\si{r}}{\sweed{\si{k}}{\ncat}} \deligne \cdots}{U(y)} \circ U.\label{eq_moveEF_pdef_second}
\end{align}
Here we used the form of Sweedler notation for the forgetful functor from \eqref{eq_reldel_forget_sweed}, writing
\[
U(k) = \sweed{k}{\mcat} \deligne \sweed{k}{\ncat}.
\]

To see that $c \colon \postmove{\preblock} \to \preblock$ is injective, note that under the isomorphisms
\[
\postmove\preblock \cong \homset{\sweed{x}{\mcat \reldel \ncat}}{\sweed{y}{\mcat\reldel\ncat}}\allowbreak \otimes \cdots
,\qquad
\preblock \cong \homset{\sweed{x}{\mcat}\deligne \sweed{x}{\ncat}}{\sweed{y}{\mcat}\deligne \sweed{y}{\ncat}} \otimes \cdots,
\]
that we encountered in Remark~\ref{rem_preblock_shortform},
the map $c$ corresponds to the subspace inclusion, meaning that the following diagram commutes
\begin{equation*}
	\begin{tikzcd}
		\postmove{\preblock} \arrow[d, "c"] \arrow[r, "\cong"] & \homset{\sweed{x}{\mcat \reldel \ncat}}{\sweed{y}{\mcat\reldel\ncat}} \otimes \cdots \arrow[d, hook, "\text{subspace inclusion}"{right}] \\
		\preblock \arrow[r, "\cong"] & \homset{\sweed{x}{\mcat}\deligne \sweed{x}{\ncat}}{\sweed{y}{\mcat}\deligne \sweed{y}{\ncat}} \otimes \cdots.
	\end{tikzcd}
\end{equation*}
This can be checked explicitly, using the isomorphism $\soc$ from Section~\ref{sec_genyon}.
Thus $c$ is monic, and the outer square in the diagram \eqref{diag_emove_proof} commutes if all the cells commute.

To verify that the bottom square in \eqref{diag_emove_proof} commutes, recall that the holonomy idempotents~$\postmove{\holidem}$ and $\holidem_\mathrm{Rest}$ are each defined as a composition of individual holonomy idempotents for every domain except $\bcat$ -- these domains appear in both $\sigsurf$ and $\postmove{\sigsurf}$.
On the other hand, the linear map $c$ affects only the algebraic data associated with the domain $\bcat$.
It thus commutes with the holonomy idempotents for the other domains.

This leaves us with the checking the commutativity of the triangles on the left and the right of \eqref{diag_emove_proof}, which happen to be equal.
They represent the equation
$c \circ \Pmovemap = \holidem_\bcat$,
which needs to be checked now.
Knowing the components of $c$ from \eqref{eq_moveEF_pdef_second} and of $\Pmovemap$ from \eqref{eq_EF_Pmovemap}, we write out the composition $c\circ \Pmovemap$ in components $\si{m},\si{q} \in \mcat$,  $\si{n}, \si{r} \in \ncat$
\begin{align*}
		\ctc{(c\circ \Pmovemap)}{\si{mn}}{\si{qr}} = \sum_{\si{k} \si{a}\si{b}} \pdim{\si{k}} \pdim{\si{q}}\pdim{\si{r}} \pdim{\si{b}}  \pica{1}{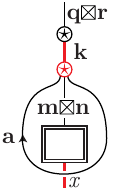} \otimes \pica{1}{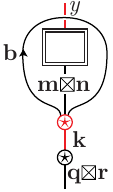}.
\end{align*}
In order to compare this expression to the respective components of $\holidem_\bcat$, it is most convenient to evaluate it on morphisms $g \colon U(x) \to \si{m} \deligne \si{n}$ and $f \colon \si{m} \deligne\si{n} \to U(y)$, and then pass through the isomorphism
\begin{equation}
	\soc^{-1} \colon\ \bigoplus_{\si{q}\si{r}} \homset{U(x)}{\si{q}\deligne\si{r}} \otimes \homset{\si{q}\deligne\si{r}}{U(y)} \to \homset{U(x)}{U(y)}
\end{equation}
from \eqref{eq_sil}.
Doing so, we obtain
\begin{align}
\soc^{-1}( \ct{(c\circ \Pmovemap)}{\si{mn}}(g\otimes f) )
		&= \hspace{-2mm}\sum_{\si{k} \si{q}\si{r} \si{a}\si{b}} \pdim{\si{k}} \pdim{\si{q}}\pdim{\si{r}} \pdim{\si{b}} \hspace{-5mm} \pica{1}{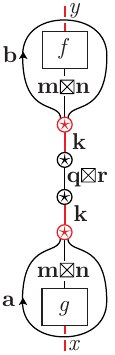}
		\eqwithref{eq_res_of_id} \sum_{\si{k} \si{a}\si{b}} \pdim{\si{k}} \pdim{\si{b}} \hspace{-5mm} \pica{1}{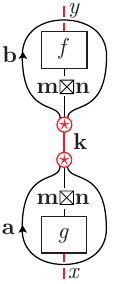}\nonumber\\
		&\eqwithref{eq_res_of_id} \sum_{\si{a}} \pdim{\si{a}} \pica{1}{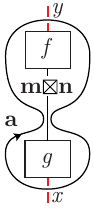}.\label{eq_EF_rsq_calc_1}
\end{align}
On the other hand, an unraveling of the definition of $\holidem_\bcat$ similar to Figure~\ref{calc_lasso_longcalc} reveals the following:
\begin{align*}
		\ctc{(\holidem_\bcat)}{\si{mn}}{\si{qr}} &= \frac{1}{\catdim{\bcat}} \sum_{\si{i}\si{j} \si{a}} \pdim{\si{a}} \pdim{\si{i}}\pdim{\si{j}} \pica{1}{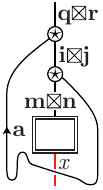} \otimes \pica{1}{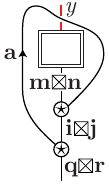}\\
&
		\eqwithref{eq_parallel_res} \frac{1}{\catdim{\bcat}} \sum_{\si{a}} \pdim{\si{a}} \pica{1}{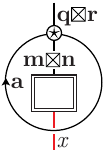} \otimes \pica{1}{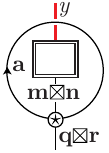}.
\end{align*}
Again, evaluating on $g\otimes f$ and applying $\soc^{-1}$ yields
\begin{align*}
		&\soc^{-1}( \ct{(\holidem_\bcat)}{\si{mn}}(g\otimes f) )
		= \frac{1}{\catdim{\bcat}} \sum_{\si{q}\si{r} \si{a}} \pdim{\si{a}} \pdim{\si{q}}\pdim{\si{r}} \hspace{-5mm} \pica{1}{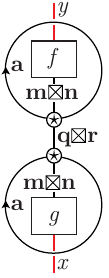} \eqwithref{eq_res_of_id} \frac{1}{\catdim{\bcat}} \sum_{\si{a}}\pdim{\si{a}} \pica{1}{pic_ef_rsq_5},
\end{align*}
which is identical to \eqref{eq_EF_rsq_calc_1}.
This implies $c\circ \Pmovemap = \holidem_\bcat$, and so the outer paths in the diagram~\eqref{diag_emove_proof} commute, and the move $\moveEF$ is an invariance.	

The linear map $\Movemap$ between block spaces for the $\moveEF$-move also appears in \cite[Theorem~4.43]{fss-statesum}, where it is proved that $\Movemap$ is an isomorphism.

\subsection{Composite moves} \label{sec_compmoves}
This section is the second and final part of the proof of Theorem~\ref{thm_moves}.
The remaining moves do not have to be proved by explicitly using the definition of the evaluation procedure, but can be obtained as compositions of moves we already know to be invariances.

\subsubsection[DV -- silent node]{${\moveDV}$ -- silent node} \label{move_DV_pf}
We begin by redrawing Figure~\ref{pic_moveDV} in a slightly larger region, this time involving another node labeled by $x$.
Such a node $x$ exists because the extruded graph on the left-hand side of Figure~\ref{pic_moveDV} is not allowed to feature a loop without a node (see Figure~\ref{pic_moveDV_pf}).

\begin{figure}[ht]\centering
\raisebox{5mm}{\includegraphics{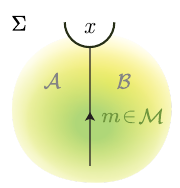}} \ \raisebox{21mm}{$\xrightarrow{\moveDV}$} \includegraphics{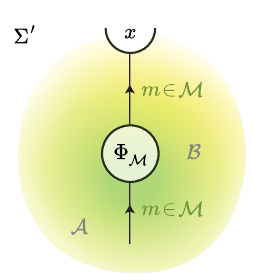}
\caption{}\label{pic_moveDV_pf}
\end{figure}

The invariance under the $\moveDV$-move is a consequence of the following more general lemma.
\begin{Lemma}\label{lem_move_section_invariance}
	Let $\mathbf{M}_1$ be a move between regions $\sigsurf$ and $\postmove{\sigsurf}$, and let $\mathbf{M}_2$ be a move between $\postmove{\sigsurf}$ and $\postmove{\postmove{\sigsurf}}$.
	Suppose that both $\mathbf{M}_2$, and the composite move $\mathbf{M}_1 \mathbf{M}_2$ between $\sigsurf$ and $\postmove{\postmove{\sigsurf}}$ are moves of invariance.
	If the linear map $\Movemap_2$ between block spaces of $\postmove{\sigsurf}$ and $\sigsurf''$ is injective, then $\mathbf{M}_1$ is an invariance as well.
\end{Lemma}
\begin{proof}
	In the following diagram, commutativity of the left square is equivalent to $\mathbf{M}_1$ being an invariance.
	This commutativity follows from the commutativity of the right square (due to the fact that $\mathbf{M}_2$ is an invariance), the commutativity of the outer paths ($\mathbf{M}_1 \mathbf{M}_2$ is an invariance), and the fact that $\Movemap_2$ is a monomorphism
	\begin{equation*}
		\begin{tikzcd}
			\nodespace \arrow[r, "\movemap_1"] \arrow[d, "\evaluate{\sigsurf}"] & \postmove{\nodespace} \arrow[r, "\movemap_2"] \arrow[d, "\evaluate{\postmove{\sigsurf}}"] & \postmove{\postmove{\nodespace}} \arrow[d, "\evaluate{\postmove{\postmove{\sigsurf}}}"] \\
			\blockspace \arrow[r, "\Movemap_1"] & \postmove{\blockspace} \arrow[r, "\Movemap_2"] & \postmove{\postmove{\blockspace}}.
		\end{tikzcd}\tag*{\qed}
	\end{equation*} \renewcommand{\qed}{}
\end{proof}

To apply Lemma~\ref{lem_move_section_invariance} to the present situation by setting $\mathbf{M}_1 = \moveDV$ and $\mathbf{M}_2 = \moveC$,
we now show that the composite $\moveDV\moveC$ is the identity move, which is clearly an invariance.
This is obvious on the topological level.
As for the node labels, it follows immediately from the definition \eqref{eq_def_gencomp_U} of the contraction operation $\mitosis{\mcat}$ that $\zewid{\mcat}$ is a unit for the operation
$\zewid{\mcat} \mitosis{\mcat} x \cong x$.
Lastly we need to verify that the linear maps $\movemap_{\moveDV}$, $\movemap_{\moveC}$ between node spaces compose to the identity.
This is done by checking that for $f\in \nodespace = \homset{m}{\sweed{x}{\mcat}}$,
$\movemap_{\moveC}(\soc(\id_\mcat) \gencomp{m} f) = f$,
which follows from unraveling the definitions \eqref{eq_gencomp_def} of $\gencomp{m}$ and \eqref{eq_sil} of $\soc$.
This proves that $\moveDV$ is an invariance.

Moreover, the $\moveDV$-move's linear map $\Movemap$ between block spaces is an isomorphism, as it is the inverse of the $\moveC$-move's map $\Movemap_{\moveC}$.

\subsubsection[L -- loop move]{${\moveL}$ -- loop move} \label{move_L_pf}
The $\moveL$-move is obtained as the composition of moves, see Figure~\ref{pic_lmove_proof}.
First, we apply the $\moveDV$-move to the loop, which splits it into two $\mcat$-labeled defect lines and creates a new silent node.
Next, we flip the orientation of one of the $\mcat$-labeled defect lines using the $\moveOR$-move, so that both point away from the $x$-labeled node.
The $\moveEF$-move then merges the two defect lines into one, labeled by \smash{$\opcat{m} \lexreldel m \in \opcat{\mcat} \reldelov{\acat} \mcat$}.
In the last step, we use the (one-sided version of the) $\moveC$-move to absorb the silent node into the node labeled by $x$.
This results in the node label \smash{$x' ={} \zewcoid{\mcat} \mitosis{\opcat{\mcat}\reldel \mcat} x$}, as described in \eqref{eq_lmove_postmove_raylabel}.

\begin{figure}[ht]\centering
\includegraphics{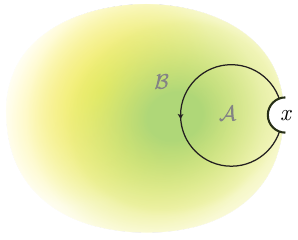} \ \raisebox{21mm}{$\xrightarrow{\moveDV}$} \includegraphics{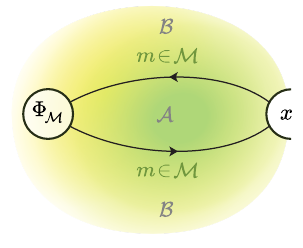}

\raisebox{21mm}{$\xrightarrow{\moveOR, \moveEF}$}\includegraphics{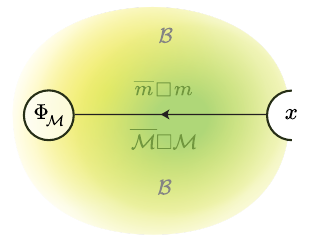} \ \raisebox{21mm}{$\xrightarrow{\moveC}$} \includegraphics{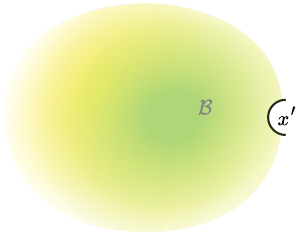}
\caption{}\label{pic_lmove_proof}
\end{figure}

We still have to compute the linear map $\movemap$ between node spaces, which is a composition of the individual maps for each move in the chain (see Figure~\ref{pic_lmove_proof}).
The first morphism is $\movemap_{\moveDV} = \soc(\id_m) \otimes \id_{\homset{\opcat{m}\deligne m}{\sweed{x}{\opcat{\mcat}\deligne\mcat}}}$.
The map $\movemap_{\moveOR}$ is the identity.
In the next step, we obtain
\begin{equation*}
	\movemap_{\moveEF} \circ \movemap_{\moveDV} = \frac{1}{\catdim{\acat}} \bigoplus_{\si{a},\si{b} \in \acat} \pdim{\si{b}} \pica{1}{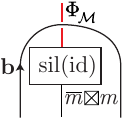} \otimes \pica{1}{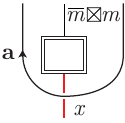}.
\end{equation*}
Lastly, when performing the $\moveC$-move, the two tensor factors are composed
\begin{equation*}
	\movemap_{\moveL} = \movemap_{\moveC} \circ \movemap_{\moveEF} \circ \movemap_{\moveDV} = \frac{1}{\catdim{\acat}} \sum_\si{a} \pdim{\si{a}} \pica{1}{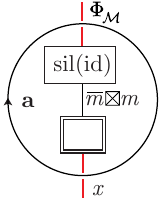} = r(\soc(\id_m) \circ -),
\end{equation*}
where $r$ denotes the retraction from Lemma~\ref{lemma_homspace_projection}.
This agrees with the linear map defined in Section~\ref{sec_moves_overview}, and so the proof is done.

This completes the proof of Theorem~\ref{thm_moves}.

\section{Details on loop graphs}\label{app_loopgraphs}
This section forms the proof of Theorem~\ref{thm_elementary_graph}.
	As in Lemma~\ref{lem_elemgraph_generalformula}, we restrict ourselves without loss of generality to the case $\bcat = \vect$.
	\begin{enumerate}\itemsep=0pt
		\item[(1)] Let $g \in \preblock = \homset{\ewid{\mcat}}{\ewid{\mcat}}$ be a vector in the pre-block space.
		We make the formula
		\begin{equation*}
			\holsurj(g) = \sum_\si{a} \frac{\pdim{\si{a}}}{\catdim{\acat}} \brev{\zewid{\mcat}}{\si{a}^*} \circ (\si{a}^* g \si{a}) \circ \cobrev{\zewid{\mcat}}{\si{a}}
		\end{equation*}
		obtained in Lemma~\ref{lem_elemgraph_generalformula} for the retraction $\holsurj \colon \preblock \to \blockspace$ from \eqref{eq_def_holproj} more explicit.
		From~\eqref{eq_ewid_balancings_components} and \eqref{eq_def_brev}, we can work out how to express the morphisms $\brev{\zewid{\mcat}}{\si{a}^*}$ and $\cobrev{\zewid{\mcat}}{\si{a}}$ in components
		\begin{align*}
			&\ete{(\brev{\zewid{\mcat}}{\si{a}^*})}{\si{l}}{\si{q}}=\pdim{\si{l}}  \basisel{\si{a}^*\si{l}}{\si{q}} \btimes \basisel{\si{q}}{\si{a}^*\si{l}} \colon\ \si{a}^*\si{l} \deligne \opcat{\si{a}^*\si{l}} \to \si{q}\deligne \opcat{\si{q}}, \\
			&\ete{(\cobrev{\zewid{\mcat}}{\si{a}})}{\si{n}}{\si{k}}= \pdim{\si{n}}  \basisel{\si{n}}{\si{a}^*\si{k}} \btimes \basisel{\si{a}^*\si{k}}{\si{n}} \colon\ \si{n} \deligne\opcat{\si{n}} \to \si{a}^*\si{k} \deligne\opcat{\si{a}^*\si{k}}.
		\end{align*}
		This allows us to state the components $\ete{\holsurj(g)}{\si{n}}{\si{q}}\colon \si{n} \deligne \opcat{\si{n}} \to \si{q} \deligne \opcat{\si{q}}$ of $\holsurj(g)$
		\begin{align}
				\ete{\holsurj(g)}{\si{n}}{\si{q}} = \sum_{\si{a},\si{k}, \si{l}} \frac{\pdim{\si{a}}\pdim{\si{n}}\pdim{\si{l}}}{\catdim{\acat}} (\basisel{\si{a}^*\si{l}}{\si{q}} \btimes \basisel{\si{q}}{\si{a}^*\si{l}}) \circ (\si{a}^* \ete{g}{\si{k}}{\si{l}} \si{a}) \circ (\basisel{\si{n}}{\si{a}^*\si{k}} \btimes \basisel{\si{a}^*\si{k}}{\si{n}}).\label{eq_pi_component}
		\end{align}
		
		\item[(2)] For the evaluation of the loop graph $O$, we also need to compute the morphism $\delta$ from~\eqref{eq_eval_delta}.
		Let $f \colon m \to m$ be an arbitrary endomorphism of $m \in \mcat$, as in Theorem~\ref{thm_elementary_graph}.
		We compute
		\begin{align}
				\delta(\soc(f)) &= \bigoplus_\si{k} \pdim{\si{k}}  \soc(f) \circ (\basisel{\si{k}}{m} \deligne \basisel{m}{\si{k}}).\label{eq_thmelempf_calc_1}
		\end{align}
		Ultimately though, we need the components $\ete{\delta(\soc(f))}{\si{k}}{\si{l}}$.
		To obtain them, we make use of the explicit form of the isomorphism $\soc$ from \eqref{eq_sil}, $\te{\soc(f)}{\si{l}} = (\basisel{m}{\si{l}} \circ f) \deligne \basisel{\si{l}}{m}$.
		We insert this into \eqref{eq_thmelempf_calc_1} to find
		\begin{align*}
				\ete{\delta(\soc(f))}{\si{k}}{\si{l}}
				&= \pdim{\si{k}}  \te{\soc(f)}{\si{l}} \circ (\basisel{\si{k}}{m} \deligne \basisel{m}{\si{k}}) = \pdim{\si{k}}  ((\basisel{m}{\si{l}} \circ f) \deligne \basisel{\si{l}}{m}) \circ (\basisel{\si{k}}{m} \deligne \basisel{m}{\si{k}}) \\
				&= \pdim{\si{k}}  (\basisel{m}{\si{l}} \circ f \circ\basisel{\si{k}}{m}) \deligne ( \basisel{m}{\si{k}}\circ \basisel{\si{l}}{m}).
		\end{align*}
		
		\item[(3)] We substitute $\delta(\soc(f))$ for $g$ in \eqref{eq_pi_component}
		\begin{align}
				\ete{\holsurj(\delta(\soc(f)))}{\si{n}}{\si{q}}={}& \sum_{\si{a},\si{k}, \si{l}} \frac{\pdim{\si{a}}\pdim{\si{n}}\pdim{\si{l}}}{\catdim{\acat}} (\basisel{\si{a}^*\si{l}}{\si{q}} \btimes \basisel{\si{q}}{\si{a}^*\si{l}}) \circ (\si{a}^* \ete{\delta(\soc(f))}{\si{k}}{\si{l}} \si{a}) \nonumber \\
&\circ (\basisel{\si{n}}{\si{a}^*\si{k}} \btimes \basisel{\si{a}^*\si{k}}{\si{n}})\nonumber \\
={}& \sum_{\si{a},\si{k}, \si{l}} \frac{\pdim{\si{a}}\pdim{\si{n}}\pdim{\si{l}} \pdim{\si{k}}}{\catdim{\acat}} (\basisel{\si{a}^*\si{l}}{\si{q}} \circ \si{a}^* (\basisel{m}{\si{l}} \circ f \circ\basisel{\si{k}}{m}) \circ \basisel{\si{n}}{\si{a}^*\si{k}}) \nonumber\\
& \btimes (\basisel{\si{a}^*\si{k}}{\si{n}} \circ (\basisel{m}{\si{k}}\circ \basisel{\si{l}}{m}) \si{a} \circ \basisel{\si{q}}{\si{a}^*\si{l}} ).\label{eq_thmelempf_calc_2}
		\end{align}
		To simplify this expression, we use that
		\begin{align}
				&\sum_{\si{l}} \pdim{\si{l}} (\basisel{\si{a}^*\si{l}}{\si{q}} \circ \si{a}^* \basisel{m}{\si{l}}) \otimes (\basisel{\si{l}}{m} \si{a} \circ \basisel{\si{q}}{\si{a}^*\si{l}} ) = \basisel{\si{a}^*m}{\si{q}} \otimes \basisel{\si{q}}{\si{a}^*m}, \qquad \text{and}\nonumber \\
				&\sum_{\si{k}} \pdim{\si{k}} (\si{a}^* \basisel{\si{k}}{m} \circ \basisel{\si{n}}{\si{a}^* \si{k}} ) \otimes ( \basisel{\si{a}^*\si{k}}{\si{n}} \circ \basisel{m}{\si{k}}\si{a} ) = \basisel{\si{n}}{\si{a}^* m} \otimes \basisel{\si{a}^* m}{\si{n}}.\label{eq_thmelempf_calc_3}
		\end{align}
		Using \eqref{eq_thmelempf_calc_3} to continue the calculation \eqref{eq_thmelempf_calc_2}, we obtain
		\begin{align}\label{eq_thmelempf_calc_4}
				\ete{\holsurj(\delta(\soc(f)))}{\si{n}}{\si{q}}
				= \sum_{\si{a}} \frac{\pdim{\si{a}}\pdim{\si{n}}}{\catdim{\acat}} (\basisel{\si{a}^*m}{\si{q}} \circ \si{a}^*f \circ \basisel{\si{n}}{\si{a}^* m} ) \btimes ( \basisel{\si{a}^* m}{\si{n}} \circ \basisel{\si{q}}{\si{a}^*m} ).
		\end{align}
		
		\item[(4)] The cells of the following diagram commute
		\begin{equation}\label{eq_thmelempf_calc_5}
			\begin{tikzcd}
				\homsetin{\raycat}{\zewid{\mcat}}{\zewid{\mcat}} \arrow[r, hook] \arrow[d, "\trace{\raycat}"] & \homsetin{\nodecat}{\ewid{\mcat}}{\ewid{\mcat}} = \bigoplus_\si{n} \homset{\si{n} \deligne \opcat{\si{n}}}{\ewid{\mcat}} \arrow[r, "\bigoplus_\si{n} \soc^{-1}"] \arrow[d, "\trace{\nodecat}"] & \bigoplus_\si{n} \homsetin{\mcat}{\si{n}}{\si{n}} \arrow[d, "\sum_\si{n} \pdim{\si{n}} \trace{\mcat}"] \\
				\field \arrow[r, "\catdim{\acat}"] & \field \arrow[r, "\id"] & \field.
			\end{tikzcd}
		\end{equation}
		Commutativity of the left cell follows from Lemma~\ref{lem_center_calabiyau}; the right square can be checked explicitly using the explicit form of the isomorphism $\soc$ form \eqref{eq_sil}.
		
		\item[(5)] We compute $\bigoplus_\si{n} \soc^{-1}(\et{\holsurj(\delta(\soc(f)))}{\si{n}})$, starting from \eqref{eq_thmelempf_calc_4}
		\begin{align*}
				\soc^{-1}(\et{\holsurj(\delta(\soc(f)))}{\si{n}})
				&= \sum_{\si{a}\si{q}} \frac{\pdim{\si{a}}\pdim{\si{n}}}{\catdim{\acat}} (\basisel{\si{a}^* m}{\si{n}} \circ \basisel{\si{q}}{\si{a}^*m} \circ \basisel{\si{a}^*m}{\si{q}} \circ \si{a}^*f \circ \basisel{\si{n}}{\si{a}^* m} )\\
				&= \sum_{\si{a}} \frac{\pdim{\si{a}}}{\catdim{\acat}} (\basisel{\si{a}^* m}{\si{n}} \circ \si{a}^*f \circ \basisel{\si{n}}{\si{a}^* m} ).
		\end{align*}
		To arrive at the first equality, we used the explicit form of the isomorphism $\soc$.
		For the second equality, we recognized a resolution of the identity. To perform this step, it is necessary to replace the dimension factor $\pdim{\si{n}}$ with $\pdim{\si{q}}$, which is possible thanks to Schur's lemma.
		
		\item[(6)] We are now in a position to compute the evaluation
		\begin{align}
				(\theta, \evaluate{O}\soc(f))
				&\eqwithref{eq_total_eval} \theta \circ \pi \circ \delta \circ \soc(f) \eqwithref{eq_dualoftheta} \catdim{\acat}  \trace{\raycat}(\holsurj \circ \delta \circ \soc(f) ) \nonumber\\
				&\eqwithref{eq_thmelempf_calc_5} \sum_\si{n} \pdim{\si{n}} \trace{\mcat} \bigl(\soc^{-1}(\et{\holsurj (\delta (\soc(f)))}{\si{n}})\bigr) \nonumber\\
				&= \frac{1}{\catdim{\acat}} \sum_{\si{n}, \si{a}} \pdim{\si{n}} \pdim{\si{a}} \trace{\mcat} (\basisel{\si{a}^* m}{\si{n}} \circ \si{a}^*f \circ \basisel{\si{n}}{\si{a}^* m}) \nonumber\\
				&\eqdesc{(*)} \frac{1}{\catdim{\acat}} \sum_{\si{a}} \pdim{\si{a}} \trace{\mcat} ( \si{a}^*f ) = \trace{\mcat} ( f ).\label{eq_thmelempf_calc_7}
		\end{align}
		The equality $(*)$ uses the fact that for any endomorphism $\alpha \colon x \to x$ in a Calabi--Yau category $\xcat$, $\trace{\xcat}(\alpha) = \sum_\si{y} \pdim{\si{y}} \trace{\xcat}(\basisel{x}{\si{y}} \circ \alpha \circ \basisel{\si{y}}{x})$.
		To check this, one uses symmetry of the trace and a resolution of the identity.
		
		It is easy to see that the result \eqref{eq_thmelempf_calc_7} is equivalent to the statement
$\theta \circ \evaluate{O} = \trace{\mcat} \circ \soc^{-1}
$
		of Theorem~\ref{thm_elementary_graph}.
		The claim also holds true in the case where $\bcat$ is not necessarily $\vect$. Thus, the proof is done.
		\end{enumerate}

\subsection*{Acknowledgements}
We are grateful to Catherine Meusburger and Alexis Virelizier for detailed
discussions and to Alexis Virelizier for insisting that extruded graphs can
be dealt with as two-dimensional objects.
We thank David Jaklitsch for helpful discussions and Katrin Suchowski for creating the pictures.
The authors are supported by the Deutsche Forschungsgemeinschaft (DFG, German Research Foundation)
under SCHW1162/6-1; CS is partially supported under Germany's Excellence Strategy - EXC 2121
``Quantum Universe'' - 390833306.
Finally, we thank the anonymous referees for their thorough reviews of the submission which helped to improve the exposition
of the results.

\pdfbookmark[1]{References}{ref}
\LastPageEnding

\end{document}